\documentclass[11pt]{article}
\usepackage{graphicx, amsmath, amsfonts, amssymb, amsthm, fullpage, authblk, bbm, comment, authblk, enumitem, etoc}
\usepackage[dvipsnames]{xcolor}
\usepackage{algorithm}
\usepackage{algpseudocode}

\algnewcommand\algorithmicinput{\textbf{INPUT:}}
\algnewcommand\INPUT{\item[\algorithmicinput]}
\algnewcommand\algorithmicoutput{\textbf{OUTPUT:}}
\algnewcommand\OUTPUT{\item[\algorithmicoutput]}

\usepackage{hyperref}[]
\hypersetup{
    colorlinks=true,
    linkcolor=blue,
    filecolor=magenta,      
    urlcolor=cyan,
      citecolor=blue,
    }
\usepackage{cleveref, setspace}
\newtheorem{theorem}{Theorem}
\newtheorem{lemma}[theorem]{Lemma}
\newtheorem{proposition}[theorem]{Proposition}
\newtheorem{corollary}[theorem]{Corollary}

\newtheorem{definition}{Definition}
\newtheorem{remark}{Remark}
\newtheorem{assumption}{Assumption}

\allowdisplaybreaks

\newcommand{\dint}{\,\mathrm{d}}

\usepackage[round]{natbib}

\title{Optimal Cox regression under federated differential privacy: \\coefficients and cumulative hazards}
\author{Elly K.~H.~Hung}
\author{Yi Yu}
\affil{Department of Statistics, University of Warwick}
\date{}

\begin{document}

\maketitle

\begin{abstract}
We study two canonical problems in distributed survival analysis under federated differential privacy (FDP): estimation of cumulative hazard functions and Cox regression coefficients, allowing for heterogeneous per-server sample sizes and privacy budgets. To quantify the fundamental cost of privacy, we derive minimax lower and upper bounds that match up to poly-logarithmic factors, revealing server-level phase transitions between private and non-private regimes.  Our algorithms are computationally-efficient and are based on adapting differentially private binary tree and gradient descent algorithms to handle estimating equations with dependent terms. 

We also study three related problems: estimation of the cumulative baseline hazard and survival functions in the Cox model under FDP, and estimation of regression coefficients under the federated label differential privacy, a relaxed privacy notion. For these problems, we develop efficient algorithms, establish theoretical guarantees and investigate minimax optimality.

Numerical experiments, including a real-data application, support the theoretical findings.  The proposed methods are implemented in the accompanying \texttt{R} package \texttt{FDPCox}. 
\end{abstract}
\addtocontents{toc}{\protect\setcounter{tocdepth}{0}}
\section{Introduction}
Privacy concerns have become increasingly prominent in recent years, fuelled by the widespread collection of personal data and the growing number of incidents involving unauthorised access, misuse and re-identification.  These concerns are particularly acute in high-stakes domains such as healthcare, biomedical research, social sciences and revenue management, as datasets in these areas often contain highly sensitive information about individuals.  In such settings, ensuring data confidentiality is not only an ethical imperative but also a regulatory requirement \citep[e.g.][]{GDPR2016a, hipaa}.  

A common characteristic of data in the aforementioned domains is the presence of censoring, which typically occurs when follow-up of study subjects is discontinued before the event of interest is observed.  The backbone of statistical analysis to address censoring is survival analysis, a well-established area with decades of development, much of which has traditionally assumed an environment where privacy was not a central concern. 

A variety of techniques have been explored for preserving data privacy while enabling meaningful statistical analysis.  Classical methods such as data anonymisation or restricting access to summary statistics have been shown to be vulnerable to re-identification attacks \citep[e.g.][]{narayanan2008deanonymization, rocher2019reidentification, abowd20232010}. Differential privacy \citep[DP,][]{dwork2006differential, dwork2014algorithmic} has been gaining increasing attention for offering not only mathematically rigorous guarantees but also practical feasibility. Notably, DP was adopted in the 2020 US Census \citep[e.g.][]{us_census_2021_key_parameters} and has been deployed by other leading organisations to safeguard sensitive data \citep[e.g.][]{apple2021_dp}.  

Despite substantial progress in differentially private statistical estimation, the fundamental cost of privacy in survival analysis under central or federated DP remains poorly understood.  Standard survival estimators involve at-risk sets, so changing one individual's record can affect multiple terms in an estimator, complicating both privacy mechanisms and statistical analysis.  When data are distributed across servers with different sample sizes and privacy budgets, these challenges are compounded by the need to combine information subject to different local constraints.  In this paper, we study the fundamental limits of estimation and develop computationally-efficient methods for survival analysis under federated DP.

\subsection{List of contributions}\label{subsec-contributions}

We summarise the contributions of this paper.

\begin{itemize}
    \item To the best of our knowledge, this is the first systematic minimax analysis for survival analysis under central or federated DP. We study estimation of cumulative hazard functions and Cox regression coefficients under FDP, allowing heterogeneous server-level sample sizes and privacy budgets. We establish matching minimax lower and upper bounds up to poly-logarithmic factors for cumulative hazard estimation under non-interactive mechanisms, and for Cox regression coefficient estimation under the considered interactive mechanisms.  These results quantify the fundamental cost of privacy, and reveal server-specific phase transitions between privacy-dominated and non-private estimation regimes.

    \item We propose FDP Nelson--Aalen and Breslow estimators using binary tree mechanisms, and a private stochastic gradient descent algorithm for Cox regression.  Our analysis extends private estimation methods to survival data by addressing the dependence induced by risk sets and allowing time-varying covariates. For cumulative baseline hazard estimation, we establish an upper bound and derive an accompanying lower bound. 
    
    \item 
    We develop estimators of conditional survival functions under the Cox
    model and establish matching minimax upper and lower bounds, up to logarithmic factors,
    under squared uniform loss over time and bounded covariate paths.
    
    \item We also investigate Cox regression when covariates are publicly available and privacy protection is required only for the observed times and censoring indicators.  We derive minimax rates for this setting, identifying conditions under which the weaker privacy requirement does not improve the minimax rate relative to full FDP.  
    
    \item Lastly, we provide extensive numerical experiments to support our theory, and include a real data application to show the practicality of our algorithms. In addition, we release an accompanying \texttt{R} package \texttt{FDPCox}\footnote{Available at \url{https://github.com/EKHung/FDPCox}}, which implements the proposed methods. 

\end{itemize}

\subsection{Related works}\label{subsec-lit}
We briefly review the literature most closely related to this paper: differential privacy in distributed learning, survival analysis for distributed data and private survival analysis.

Differential privacy (DP) is a standard framework for rigorous privacy protection in statistical inference \citep{dwork2006differential}.  Beyond the classical central model, several variants have been developed for different settings, including local DP \citep[e.g.][]{duchi2018minimax, acharya2019contextawareLDP}, user-level DP \citep[e.g.][]{levy2021userDP,  kent2024rateoptimalityphasetransition} and label DP \citep[e.g.][]{chaudhuri2011sample, zhao2025labelDP}.  Our main focus is on federated DP, where data are distributed across multiple servers and privacy constraints may vary across them \citep[e.g.][]{wei2019federated, li2024federated, cai2024optimal}. 

Survival analysis provides the main statistical framework for analysing censored time-to-event data. A fundamental question is how the distribution of the event time can be estimated while properly treating the censored observations, with the Nelson--Aalen estimator \citep{nelson1972hazard, aalen1978multipledecrement} being a classical solution. When covariate effects are of interest, one of the most widely used regression models for survival data is the Cox proportional hazards model \citep{cox1972regression, cox1975partial}.  Motivated by large-scale and multi-institutional applications, there is a growing literature on distributed and communication-efficient survival analysis \citep[e.g.][]{lu2015webdisco, li2023distributed, bayle2023communication, zhang2025dsubcoxfastsubsamplingalgorithm}.  However, the privacy guarantees of these distributed procedures are often unclear, especially when the released summaries depend on small risk sets.

Compared with the large volume of literature on DP, rigorous work on private survival analysis remains limited. For cumulative hazard estimation, \cite{gondora2020DPsurvival} present a central DP Kaplan--Meier estimator for discrete-time data, and derive related statistics using the post-processing property of DP. \cite{veeraragavan2024differentiallyprivatekaplanmeierestimator} propose a central DP Kaplan--Meier estimator that relies on clipping the survival probabilities and smoothing.
\cite{Nguyen2017DPsurvival} study discrete-time Cox regression problems and propose several central DP mechanisms based on methods from \cite{chaudhuri2011DPERM} and \cite{li2016preconditioned}. \cite{rahimian2022practical} consider a federated setting with a trusted central server injecting noise to satisfy DP constraints. \cite{niu2025adaptive} develop a method for client-level DP updates for estimating Cox model parameters with neural networks. These previous works have provided empirical results for their methods. \cite{EgeaEscobar2025survival} focus on the case where only the censoring indicators need to satisfy local DP, and propose a minimax optimal estimator for the cumulative hazard function based on the Beran estimator.  In view of these existing results, characterising the fundamental cost of privacy in survival analysis under central or federated DP constraints remains uncharted.

\section{Privacy framework} \label{sec-fdp}
In this section, we give a precise formulation of federated DP. At a high level, federated DP is a version of DP tailored towards distributed data, where central DP is considered within each server and only private information is allowed to be communicated across servers. We define central and federated DP in Definitions~\ref{def-cdp} and \ref{def:FDP}, respectively.  

\begin{definition}[Central differential privacy, CDP] \label{def-cdp} 
For $\epsilon, \delta > 0$, a privacy mechanism $M$ is an $(\epsilon, \delta)$-CDP mechanism, if it is a conditional distribution on space $\mathcal{R}$ that satisfies 
\begin{equation}
    M(R \in A \mid D) \leq e ^\epsilon M(R \in A \mid D') + \delta,
    \label{eq:DP_def}
\end{equation}
for all measurable sets $A \in \sigma(\mathcal{R})$, and all pairs of datasets $(D, D')$ differing by at most one data point, denoted as $D \sim D'$.  
\end{definition}

\Cref{def-cdp} was originally proposed in \cite{dwork2006calibrating} and implicitly assumes that there is a trusted central server which aggregates the information to release a privatised output.  Since its debut, other variants have been proposed for other types of privacy constraints and communication architecture, including the version of federated DP defined below.  A very similar form has been used in \cite{li2024federated}, \cite{cai2024optimal}, \cite{xue2024optimal} and \cite{cai2025FDPfunctional}, among others.

\begin{definition}[Federated differential privacy, FDP] \label{def:FDP} 
For $S \in \mathbb{N}_+$, let $\epsilon_s, \delta_s > 0$, $s \in \{1, \ldots, S\} =: [S]$, be privacy parameters.  For $K \in \mathbb{N}_+$, we say that a privacy mechanism $Q := \{Q_s^{(k)}\}_{s, k=1}^{S, K}$ satisfies $(\{\epsilon_s, \delta_s\}_{s=1}^S, K)$-FDP, if for any $s \in [S]$ and $k \in [K]$, the data $R_s^{(k)}$ shared by the server $s$ satisfies $(\epsilon_s, \delta_s)$-CDP, i.e.
\[
    Q_s^{(k)}(R_s^{(k)} \in A_s^{(k)} \mid M^{(k-1)}, D_s^{(k)} ) \leq e^{\epsilon_s}Q_s^{(k)}(R_s^{(k)} \in A_s^{(k)} \mid M^{(k-1)}, (D_s^{(k)})') + \delta_s,
\]
for any measurable set $A_s^{(k)}$, $M^{(k-1)} := \cup_{l=1}^{k-1} \cup_{s=1}^S R_s^{(l)}$, and pairs of datasets $D_s^{(k)} \sim (D_s^{(k)})'$.  For each $s \in [S]$, $D_s = \cup_{k=1}^K D_s^{(k)}$ forms a fixed, data-independent partition. 
\end{definition}

\Cref{def:FDP} specified a $K$-iteration mechanism, which involves independent batches of data at each iteration.  At the end of each iteration, private data $M^{(k-1)}$ obtained from all servers are allowed to communicate across servers for the next iteration. Throughout this paper, our upper and lower bounds will assume that independent sources of randomisation are used in the privacy mechanism for each batch of data $D_s^{(k)}$.

In both of the aforementioned definitions, the privacy constraints are quantified by the parameters $\epsilon$ and $\delta$.  The parameter $\epsilon$ controls how much the two conditional distributions can change when a single data point in the dataset is modified. The parameter $\delta$ allows for a small probability event where the privacy constraint does not hold, usually taken to be polynomial functions of the sample sizes.  The smaller the privacy budgets are ($\epsilon$ and/or $\delta$), the more stringent the privacy constraint is, and therefore the harder the statistical inference tasks are.  

One common way to quantify the hardness of a statistical inference task is through the minimax risk.  For a private statistical inference task, the private minimax risk \citep[e.g.][]{duchi2018minimax} is formalised via $\mathcal{P}$ the family of data generating distributions, and $\mathcal{Q}$ the collection of all mechanisms satisfying the required DP constraints.  The private minimax risk is defined as
\[
    \inf_{Q \in \mathcal{Q}} \inf_{\widehat{\theta}} \sup_{P \in \mathcal{P}} \mathbb{E}_{P, Q} \big\{L\big(\theta(P), \widehat{\theta}\big)\big\},
\]
where $\theta(P)$ is the quantity of interest, $L(\cdot, \cdot)$ is a loss function, the inner infimum is taken over all private estimators, the outer infimum is over all possible privacy mechanisms and the expectation is with respect to both the data generating distributions and the conditional distributions generating the privatised data.

Unless otherwise specified, we focus on the challenging, high-privacy regime where $\epsilon_s$ is upper bounded by an absolute constant, say $\epsilon_s < 1$, and $\delta_s <e^{-1}$ for each server. While the theoretical guarantees derived in this paper for our proposed methods hold for general $\epsilon$, the minimax lower bounds require this upper bound condition on $\epsilon$. 

\medskip
\noindent\textbf{Notation}\label{subsec-notation}
For any positive integer $a$, denote $[a] := \{1, \dots, a\}$, and for any $x \in \mathbb{R}$, let $\lfloor x \rfloor$ be the largest integer not exceeding $x$. We write $x \wedge y$ to mean $\min\{x, y\}$. For any vector $v$, denote $v ^{\otimes 0} := 1$, $v^{\otimes 1} := v$, $v^{\otimes 2} := vv^\top$ and denote $\|v\|_p$ as its $\ell_p$-norm.  For any square matrix $M$, let $\|M\|$, $\lambda_{\min}(M)$ and $\lambda_{\max}(M)$ be its $2\to2$-operator norm, smallest and largest eigenvalues respectively.  For any two Hermitian matrices of the same dimension, we write $M \preceq N$ if $N-M$ is positive semi-definite. We use $I_d$ to denote the $d$-dimensional identity matrix. For any $x \in \mathbb{R}^d$ and $r > 0$, we write $B_r(x)$ as the Euclidean ball of radius $r$ centred at $x$, and $\Pi_r(\cdot)$ as the projection onto $B_r(0)$. For any two sequences of positive numbers $\{a_n\}$ and $\{b_n\}$, if there exists $n_0 \in \mathbb{N}$ and $0 < c < \infty$ such that $a_n \leq c b_n$ for $n \geq n_0$, we write $a_n \lesssim b_n$, and we write $a_n \asymp b_n$ if $a_n \lesssim b_n$ and $a_n \gtrsim b_n$.

\section{Cumulative hazard estimation}
\label{sec:cumulative_hazard} 

The cumulative hazard function characterises the event time distribution through its relationship with the survival function, making its estimation a fundamental task in survival analysis.  Despite its importance, we are not aware of theoretical guarantees for cumulative hazard function estimation under federated DP constraints. In this section, we develop a federated DP Nelson--Aalen estimator, and study its minimax optimality. The event times are assumed to follow a distribution with unknown hazard function $\lambda: \mathbb{R}_{\geq 0} \rightarrow \mathbb{R}_{\geq 0}$ that is Lebesgue integrable on $[0, 1]$; the cumulative hazard function is then defined as $\Lambda:[0, 1] \rightarrow [0, \infty)$, with $\Lambda(t) := \int_0^t \lambda(s) \dint s$.  We start by stating the assumptions needed for our theoretical results. 

 \begin{assumption} The event time $\widetilde{T}$ and censoring time $C$ are independent. We observe $(T, \Delta) := (\min\{\widetilde{T}, C, 1\}, \mathbbm{1}\{\widetilde{T} \leq \min\{C, 1\}\})$.
\label{assp:indep}
\end{assumption}

\begin{assumption} \label{assp:baseline}
There exist absolute constants $C_\lambda>0$ and $0<C_p<1$ such that $\mathbb{P}(\min\{\widetilde{T}, C\}\geq1) \geq C_p$ and  $\lambda(t) \leq C_\lambda$ for all $t \in [0, 1]$. 
\end{assumption}

We consider a finite time horizon $[0, 1]$ for simplicity, but our theory can be extended to $[0, T_\mathrm{end}]$ for any absolute constant $T_{\mathrm{end}}$. \Cref{assp:indep} is a standard non-informative censoring condition in the survival analysis literature \citep[e.g.][]{aalen2008survival}. The lower bound on the at-risk probability 
in \Cref{assp:baseline} ensures that the at-risk population does not become arbitrarily small at the end of the observation period. The upper bound on the hazard is a boundedness assumption used in nonparametric literature. 

In our distributed data setup, we assume that there are $S$ servers with heterogeneous sample sizes $\{n_s\}_{s=1}^S$ and privacy budgets $\{(\epsilon_s, \delta_s)\}_{s=1}^S$. Server $s \in [S]$ holds data $\{(T_{s, i}, \Delta_{s, i})\}_{i=1}^{n_s}$.  We define the effective sample size as 
\begin{equation}\label{eq-effect-ss}
    N_S = \sum_{s=1}^S \min(n_s, n_s^2 \epsilon_s^2).    
\end{equation}
Across all servers, these observations are independent and identically distributed (i.i.d.) copies of $(T, \Delta)$, satisfying Assumptions~\ref{assp:indep} and \ref{assp:baseline}. Following the counting process representation from \cite{andersen1982}, for $s \in [S]$, $i \in [n_s]$ and $t \in [0, 1]$, define $N_{s, i}(t) := \mathbbm{1}\{T_{s, i} \leq t, \, \Delta_{s, i} = 1\}$ and $Y_{s, i}(t) := \mathbbm{1}\{T_{s, i} \geq t\}$. 

\begin{remark}In this section, to privately estimate the cumulative hazard function, we focus on the supremum norm loss to follow suit with the existing literature, and on non-interactive mechanisms, which have the advantage of minimising communication costs. Further results on sequentially interactive FDP mechanisms (\Cref{def:FDP} with $K>1$) are provided in \Cref{app:lambda_equiv_optimality}. 
\end{remark}

\subsection{A federated differentially private Nelson--Aalen estimator}

The Nelson--Aalen estimator \citep{nelson1972hazard, aalen1978nonparametric} is a classical nonparametric estimator of the cumulative hazard, defined as $\widehat{\Lambda}(t) := \sum_{i=1}^n \int_0^t  \dint N_i(u)/\sum_{i=1}^n Y_i(u)$, $t \in [0, 1]$. The key idea of our federated DP Nelson--Aalen estimator is to partition the time horizon $[0, 1]$ into equal-sized intervals and privatise the Nelson--Aalen estimator within each interval using the Gaussian mechanism. The interval width controls approximation bias, and is chosen to balance the non-private and private rates for estimating the integral. Since releasing cumulative estimates up to a time point $t \in [0, 1]$ would aggregate noise from privatising the preceding intervals, we instead consider publishing  $(\epsilon_s, \delta_s)$-CDP binary trees that are then averaged across servers. This allows $\widehat{\Lambda}(t)$ for all $\,t \in [0, 1]$ to be reconstructed from only a logarithmic number of privatised interval estimators. In addition, to accommodate heterogeneous per-server sample sizes and privacy budgets, we adopt a weighted average of the cumulative hazard estimated across the servers, where the weights are based on the variance of each server's estimator. Our estimator is formalised in \Cref{alg:FDP-NA}. 

Tree-based private algorithms have been studied in \cite{chan2011private} and \cite{kulkarni2018answeringrangequerieslocal} but in very different settings: \cite{chan2011private} consider counting-based queries under central DP and \cite{kulkarni2018answeringrangequerieslocal} analyse range queries under local DP. In both papers, altering a single data point affects at most one node at each level of the tree. Our setup has the additional complication where changing a data point can change multiple nodes at a level of a tree due to the involvement of at-risk sets in the denominator of the estimator. The stochastic at-risk set in the denominator further complicate sensitivity and error analysis. As commonly done in the privacy literature, our estimator uses a truncation that allows us to mitigate against the low probability datasets where the at-risk-sets are too small at times $t<1$. This truncation level is chosen so that with high probability, truncation does not occur, and is constructed based on a local estimate of $\mathbb{P}(Y(1)=1)$ at each server.

\begin{algorithm}[h!] 
\caption{FDP Nelson--Aalen estimator}
\textbf{Input}: Dataset $\{(T_{s, i}, \Delta_{s, i})\}_{s, i = 1}^{S, n_s}$, privacy parameters $\{(\epsilon_s, \delta_s)\}_{s=1}^S$.

\begin{algorithmic}[1]
\State Set $h:= \max \left\{\lceil \log_2(\sum_{s=1}^S \min \{n_s, n_s^2 \epsilon_s^2\})/2 \rceil , 1 \right\}$.
\For{$s=1, \dots, S$}
\State Set
\[
\hat{p}_s := \frac{1}{\lfloor n_s/10 \rfloor} \sum_{i=1}^{\lfloor n_s/10 \rfloor} Y_{s, i}(1)
+ \frac{\sqrt{2 \log(1/\delta_s) + \epsilon_s}}{\lfloor n_s/10 \rfloor \epsilon_s}W_s, \quad W_s \sim \mathcal{N}(0,1).
\]
\State Set $n'_s := n_s - \lfloor n_s/10\rfloor$ and $c_s := \max\{C_p/2, 0.9 \hat{p}_s\}$. 
    \For{$m = 1, \ldots, 2^h$}
        \State Set 
            \begin{equation*}
                x_{s, h, m} := \int_{(m-1)/2^h}^{m/2^h} \sum_{i=\lfloor n_s/10\rfloor + 1}^{n_s} \frac{\dint  N_{s, i}(t)}{\max \left\{c_sn_s',  \sum_{j=\lfloor n_s/10\rfloor + 1}^{n_s}Y_{s, j}(t) \right\}}.
            \end{equation*}
    \EndFor
        \For{$l = h-1, h-2, \ldots, 1$}
            \For{$m = 1, \ldots, 2^l$}
                \State Set $x_{s, l, m} = x_{s, l+1, 2m-1} + x_{s, l+1, 2m}$.
            \EndFor
        \EndFor
    \For{$l = 1, \dots, h$}
        \For{$m = 1, \dots, 2^l$}
            \State Generate independently \begin{equation*}
                W_{s, l, m} \sim \mathcal{N}\left(0, \left[\frac{1}{c_s^4} + \frac{2}{c_s^3} +  \frac{2}{c_s^2}\right]\frac{2 \log(1/\delta_s)/\epsilon_s + 1}{(n_s')^2 \epsilon_s/h} \right).
            \end{equation*}
            \State Set $x_{s, l, m} = x_{s, l, m} + W_{s, l, m}$.
        \EndFor
    \EndFor
\EndFor
\State Set weights $\{v_s\}_{s \in [S]}$ as $v_s := \min(n_s, n_s^2 \epsilon_s^2)/\left\{\sum_{u=1}^S \min(n_u, n_u^2 \epsilon_u^2)\right\}$. 
\State For any $t \in [0, 1)$, let $(b_1, \dots b_h)$ be the $h$-bit binary representation of $\lfloor 2^h t \rfloor$ with $b_1$ the most significant bit and set 
    \begin{equation*}
        \widehat{\Lambda}(t) :=  \max \left\{0, \sum_{s=1}^S v_s\sum_{l : b_l = 1} x_{s, l, \, \sum_{k=1}^l 2^{l-k} b_k}\right\}.  
    \end{equation*} 
\State Set $\widehat{\Lambda}(1) := \max\{0, \sum_{s=1}^S v_s (x_{s, 1, 1} + x_{s, 1, 2})\}$. 
\State \textbf{Output}: $\widehat{\Lambda}(t)_{t \in [0, 1]}$.
\end{algorithmic}
\label{alg:FDP-NA}
\end{algorithm}

\begin{theorem} \label{thm:na_upper}
Suppose that the data $\{(T_{s, i}, \Delta_{s, i})\}_{s, i = 1}^{S, n_s}$ are i.i.d.~from a distribution satisfying Assumptions \ref{assp:indep} and \ref{assp:baseline}. Letting $\widehat{\Lambda}_\mathrm{NA}$ be the output of \Cref{alg:FDP-NA}, with $N_S$ defined in \eqref{eq-effect-ss}, it holds~that
\begin{align}
        \mathbb{E} \left[\sup_{t \in [0, 1]} |\widehat{\Lambda}_\mathrm{NA}(t) - \Lambda(t)|\right]
        \lesssim  N_S^{-1/2} \log^2\bigg(\sum_{s=1}^S n_s\bigg) \max_{s \in [S]} \log(1/\delta_s),
        \label{eq:na_upper}
    \end{align}
provided that $n_s\gtrsim\log(S N_S)+\epsilon_s^{-1}\sqrt{\{2\log(1/\delta_s)+\epsilon_s\}\log(e+SN_S)} + 10$.
\end{theorem}

 \Cref{thm:na_upper} provides an upper bound on the supremum norm error of an output from \Cref{alg:FDP-NA}, matching the lower bound in \Cref{prop:hazard_lower} up to poly-logarithmic factors. The additional logarithmic factors are from privacy compositions and using the Gaussian mechanism \citep[e.g.][]{dwork2014algorithmic}. The sample size condition in \Cref{thm:na_upper} is enforced so that the probability of the truncation event is comparable to the cumulative hazard estimation error; the details of the proof of  \Cref{thm:na_upper} can be found in \Cref{subsec-proof-na-upper}.

 \subsection{Lower bound}
We have the following minimax lower bound on the supremum norm error of estimating $\{\Lambda(t)\}_{t \in [0, 1]}$, subject to non-interactive federated DP constraints.   
\begin{proposition} \label{prop:hazard_lower}
   Fix $\{(\epsilon_s, \delta_s)\}_{s=1}^S$ that satisfy $\sum_{s=1}^S n_s \delta_s \mathbbm{1} \left\{\epsilon_s^2 \leq 1/n_s \right\} \lesssim \sqrt{N_S}$. Let $\mathcal{Q}$ be the class of $(\{(\epsilon_s, \delta_s)\}_{s \in [S]}, 1)$-FDP estimators. Let $\mathcal{P}$ be the set of distributions satisfying Assumptions~\ref{assp:indep} and \ref{assp:baseline}. It then holds that
    \begin{equation*}
        \inf_{Q \in \mathcal{Q}} \inf_{\widehat{\Lambda}} \sup_{P  \in \mathcal{P}} \mathbb{E} \left[ \sup_{t \in [0, 1]} \left|\widehat{\Lambda}(t) - \Lambda_P(t) \right| \right] \gtrsim 1 \wedge N_S^{-1/2},
    \end{equation*}
    where we write $\Lambda_P$ for the cumulative hazard of the event time distribution under model $P$. 
\end{proposition}

The proof of \Cref{prop:hazard_lower} can be found in \Cref{subsec-lambda-no-covs}, based on a DP version of the Le Cam lemma.  The particular challenge lies in the right-censoring nature of the data, dealt with by constructing two distributions with identical censoring mechanisms. 

\Cref{thm:na_upper} and \Cref{prop:hazard_lower} together establish the optimal rate of FDP cumulative hazard estimation is $N_S^{-1/2}$, up to logarithmic factors. We note that the minimum with one can also be taken in the upper bound in \eqref{eq:na_upper}, by returning the zero function when $N_S^{-1/2}$ is large. As defined in \eqref{eq-effect-ss}, $N_S$ is the sum of per-server contributions, where each server's effective contribution is the minimum of its non-private sample size $n_s$ and its central DP effective sample size $n_s^2 \epsilon_s^2$, reflecting that the central DP constraint is imposed within each server. Allowing for a fully heterogeneous setting, depending on the local sample size and privacy budget, we see that each server exhibits its own phase transition between the private and non-private rates. 

To better understand this result, consider a homogeneous setting with $n_s = n$ and $\epsilon_s = \epsilon$, for all $s \in [S]$. In this setting, the rate becomes $1/(Sn)^{1/2} + 1/(Sn^2 \epsilon^2)^{1/2}$ -- the first term is the usual sampling error while the second captures the cost of imposing the FDP constraint. When $S=1$, we recover the central DP rate of $1/n^{1/2} + 1/(n\epsilon)$, which matches the minimax rate of estimating a population proportion. If all $Sn$ observations were pooled at a single server, we would instead have a privacy term of $1/(Sn \epsilon)$. Requiring privacy at each of the $S$ servers inflates the privacy term by a factor of $S^{1/2}$, consistent with the results on federated DP \citep[e.g.][]{li2024federated}.

\section{Cox regression coefficients}\label{sec-beta}  
We now incorporate individual-level covariates, extending the covariate-free model studied in \Cref{sec:cumulative_hazard}. The Cox proportional hazard model \citep{cox1972regression} provides a standard framework for modelling the effect of covariates on the event time distribution. This section establishes the minimax rates of estimating the Cox regression coefficients under federated DP constraints. 

\subsection{The Cox regression model}\label{subsec-cox}
\Cref{sec:cumulative_hazard} is concerned with a covariate-less setting, where all individuals have the same hazard rate $\lambda$. Under the Cox proportional hazards model, we instead allow the hazard to depend on a potentially time-varying covariate $Z(t) \in \mathbb{R}^d$ through 
\begin{equation}\label{eq-hazard}
    \lambda(t \mid Z) := \lambda_0(t) \exp \{\beta_0^{\top} Z(t)\}, \quad t \in [0, 1].
\end{equation}
Here, $\beta_0 \in \mathbb{R}^d$ is an unknown regression coefficient vector and $\lambda_0(\cdot)$ is a non-negative baseline hazard function.  We write $\Lambda_0(t) := \int_0^{t} \lambda_0(s)\,\mathrm{d}s$ for the cumulative baseline hazard function. Extending the notation from \Cref{sec:cumulative_hazard}, the data held at the servers are now $\{(T_{s, i}, \Delta_{s, i}, \{Z_{s, i}(t):\, t\in [0, T_{s, i}]\})\}_{s, i = 1}^{S, n_s}$, which are i.i.d.~copies of the generic triple $(T, \Delta, \{Z(t):\, t \in [0, T]\})$.  We state our assumptions on the data generating process in \Cref{sec-4.2}. 

Using the counting process notation from \Cref{sec:cumulative_hazard}, the partial log-likelihood of the Cox model \citep{cox1975partial} based on the dataset in server $s \in [S]$ can be written as
\begin{align}
    & \ell_s(\beta) := \ell(\beta; \{T_{s, i}, \Delta_{s, i}, \{Z_{s, i}(u), \, u \in [0, 1]\}\}_{i=1}^{n_s}) \nonumber \\
    = & \frac{1}{n_s} \sum_{i=1}^{n_s} \int_0^{1} \beta^\top Z_{s, i}(t) \dint N_{s, i}(t) -\frac{1}{n_s} \sum_{i=1}^{n_s} \int_0^{1} \log \left[ \sum_{j=1}^{n_s} Y_{s, j}(t) \exp \left\{ \beta^\top Z_{s, j}(t) \right\} \right] \dint N_{s, i}(t).
\label{eq:Cox_partial_likelihood}
\end{align}
The gradient and the Hessian of $\ell(\beta)$ are subsequently written as
\begin{equation}\label{eq:cox_gradient}
    \dot{\ell}_s(\beta) := \frac{\partial \ell_s(\beta)}{\partial \beta} = \frac{1}{n_s}  \sum_{i=1}^{n_s} \int_0^{1} \left\{ Z_{s, i}(t) - \bar{Z}_s(t, \beta) \right\} \dint N_{s, i}(t)
\end{equation}
and 
\begin{equation} \label{eq:Cox_hessian}
    \ddot{\ell}_s(\beta) := \frac{\partial^2 \ell_s(\beta)}{\partial \beta \partial \beta^\top} = -\frac{1}{n_s} \sum_{i=1}^{n_s} \int_0^{1}  V_s(t, \beta)\dint N_{s, i}(t),    
\end{equation}
where
\begin{equation}
    S^{(k)}_s(t, \beta) := \frac{1}{n_s} \sum_{i=1}^{n_s} Z_{s, i}(t)^{\otimes k} Y_{s, i}(t) \exp\{\beta^{\top} Z_{s, i}(t)\}, \,\, k \in \{0, 1, 2\}, \quad \bar{Z}_s(t, \beta) := \frac{S^{(1)}_s(t, \beta)}{S^{(0)}_s(t, \beta)}
    \label{eq:weighted-mean-def}
    \end{equation}
    and
    \begin{equation*}
    V_s(t, \beta) := \frac{1}{n_s S^{(0)}_s(t, \beta)}\sum_{i=1}^{n_s} Y_{s, i}(t) \exp\{\beta^{\top} Z_{s, i}(t)\} \left\{Z_{s, i}(t) - \bar{Z}_s(t, \beta) \right\}^{\otimes 2}= \frac{S^{(2)}_s(t, \beta)}{S^{(0)}_s(t, \beta)} - \bar{Z}_s(t, \beta)^{\otimes 2}.
    \label{eq:var_notation}
    \end{equation*}
On empty at-risk sets, define $\bar Z_s(t,\beta)=0$ and $V_s(t,\beta)=0$, and interpret untruncated reciprocal and log-risk-set integrands as zero.

\subsection{Federated differentially private Cox regression} \label{sec-4.2}

We first state some modelling assumptions used to obtain our theoretical results. 

\begin{assumption} \label{assump-1}
Assume that $\{(\widetilde{T}_{s, i}, C_{s, i}, \{Z_{s, i}(t): \, t \in [0, 1]\})\}_{s, i = 1}^{S, n_s}$ are mutually independent copies of $(\widetilde{T}, C, \{Z(t):\, t \in [0, 1]\})$, where the conditional hazard of $\widetilde{T}$ given the covariate process $\{Z(t):\, t \in [0, 1]\}$ is specified in \eqref{eq-hazard}.  
    \begin{enumerate}[label=\textbf{\alph*}]
        \item \label{assump-ind} The event time $\widetilde{T}$ and the censoring time $C$ are conditionally independent given the covariate process $\{Z(t):\, t \in [0, 1]\}$.   We observe $(T, \Delta) := (\min\{\tilde{T}, C, 1\}, \mathbbm{1}\{\tilde{T} \leq \min\{C, 1\}\})$.
        \item \label{assp:covariate_bound}   The process $\{Z(t):\, t \in [0, 1]\}$ is predictable and there exists an absolute constant $C_Z >0$ such that 
        \[
            \mathbb{P}\left(\sup_{t \in [0, 1]} \|Z(t)\|_2 \leq C_Z \right) = 1.
        \]
        \item \label{assp:covariate_Lipschitz} There exists an absolute constant $L_Z > 0$ such that 
        \[
            \mathbb{P} \left( \sup_{0 \leq s <t \leq 1} \frac{\|Z(s) - Z(t)\|_2}{|s-t|} \leq L_Z \right) = 1.
        \]
        \item \label{assp:coefficient_bound}  There exists an absolute constant $C_{\beta} > 0$ such that $\|\beta_0\|_2 \leq C_{\beta}$.
        \item \label{assp:eigenvalues} With the counting process notation introduced in \Cref{subsec-cox}, for any $t \in [0, 1]$ and $\beta \in \mathbb{R}^d$, let
        \[
            G(t, \beta) := Y(t) \exp\{\beta^\top Z(t)\}\{Z(t) - \mu(t, \beta)\}^{\otimes 2}, \quad \text{with } \mu(t, \beta) := \frac{\mathbb{E}[Z(t) Y(t) \exp\{\beta^\top Z(t)\}]} {\mathbb{E}[Y(t) \exp\{\beta^\top Z(t)\}]},        
        \]
        satisfying that, for some absolute constants $\rho_+ \geq \rho_- > 0$,
        \[
            \lambda_{\min} \left(\mathbb{E} \left[\int_0^1 G(s, \beta_0) \dint \Lambda_0(s) \right] \right)\geq \frac{\rho_-}{d} \,\text{ and } \,\lambda_{\max} \left(\mathbb{E} \left[\int_0^1 G(s, \beta_0) \dint \Lambda_0(s) \right] \right) \leq \frac{\rho_+}{d}.
        \]        
        \item \label{assp:mod_atrisk}There exists absolute constants $C_\lambda>0$ and $0<C_p<1$  such that $\mathbb{P}(\min\{\widetilde{T}, C\}\geq1) \geq C_p$ and  $\lambda_0(t) \leq C_\lambda$ for all $t \in [0, 1]$. 
    \end{enumerate}
\end{assumption}

\Cref{assump-1} collects the assumptions on the covariate processes and the regression coefficient vector. \Cref{assump-1}\ref{assump-ind} weakens \Cref{assp:indep} by allowing censoring to depend on the observed covariate process. Assumptions~\ref{assump-1}\ref{assp:covariate_bound}~and \ref{assp:covariate_Lipschitz}~regulate the behaviours of the covariate processes.  Similar boundedness and smoothness conditions are ubiquitous in the literature on survival analysis \citep[e.g.][]{liu1978cox_mle, fleming1991counting, yu2018cox}, and in particular in DP literature \citep[e.g.][]{cai2021cost, avella2023noisyopt}, where a bounded $\ell_2$ norm is considered when applying the Gaussian mechanism for optimisation. 
\Cref{assump-1}\ref{assp:coefficient_bound}~imposes an upper bound on the $\ell_2$-norm of the regression coefficients.   Similar assumptions are used in the DP literature \citep[e.g.][]{bassily2014private, cai2021cost} to scale the noise added in the privacy mechanism. One more relax Assumptions~\ref{assump-1}\ref{assp:covariate_bound} and \ref{assp:coefficient_bound} by adopting adaptive clipping methods \citep[e.g.][]{varshney2022adaptive} to control the sensitivity of gradients. 

\Cref{assump-1}\ref{assp:eigenvalues} introduces the de-facto population Hessian matrix for $-\ell(\beta)$ as
\begin{equation}
    \mathbb{E} \left[\int_0^1 G(s, \beta_0) \dint \Lambda_0(s) \right]
    \label{eq:population_hessian}
\end{equation}
and requires that the eigenvalues of \eqref{eq:population_hessian} are all of order $1/d$. \cite{huang2013oracle} and \cite{yu2018cox} also impose regularity conditions on \eqref{eq:population_hessian}, but tailored for a high-dimensional setting. \Cref{assump-1}\ref{assp:covariate_bound} implies an upper bound on the eigenvalues of $\eqref{eq:population_hessian}$ and a positive lower bound is needed for convex optimisation.  The choice of $1/d$ normalisation is also used in \cite{cai2021cost}.  Relaxing the homogeneity across all eigenvalues will lead to an extra factor in the rates depending on the condition number. \Cref{assump-1}\ref{assp:mod_atrisk} is the corresponding version of \Cref{assp:baseline} for the baseline hazard.  

With the assumptions in hand, in \Cref{thm-fdp-cox}, we state a minimax rate, up to poly-logarithmic factors, for the $\ell_2$-squared error in estimating the Cox regression coefficients under federated DP constraints.

\begin{theorem}\label{thm-fdp-cox}
Denote by $\mathcal{P}$ the class of distributions satisfying \Cref{assump-1} and denote $\mathcal{Q}$ the class of $(\{(\epsilon_s, \delta_s)\}_{s\in [S]}, K)$-FDP mechanisms over all $K \in \mathbb{N}_+$.  We have that
\begin{align*}
    1 \wedge \frac{d^2}{\sum_{s=1}^S \min(n_s, n_s^2\epsilon_s^2/d)} &\lesssim \inf_{Q \in \mathcal{Q}} \inf_{\widehat{\beta}} \sup_{P \in \mathcal{P}} \mathbb{E}_{Q, P} \bigl\{\|\widehat{\beta} - \beta_0\|_2^2\bigr\} \\
    &\lesssim 1 \wedge \frac{d^2}{\sum_{s=1}^S \min(n_s, n_s^2\epsilon_s^2/d)} \cdot \mathrm{polylog}\left(\{n_s, \epsilon_s, \delta_s\}_{s \in [S]}, d\right),
\end{align*}
The lower bound holds provided that $\delta_s \log(1/\delta_s) \lesssim \epsilon_s^2/d$.  Writing $L:=\log(d+\sum_{s=1}^S n_s)$, the upper bound holds under the sample size conditions $n_s \gtrsim \log(\sum_{s=1}^S n_s/d^2)$, 
\begin{equation}
dL^4\sum_{s=1}^S\left\{1\wedge\frac{n_s\epsilon_s^2}{d}\right\}
\lesssim\sum_{s=1}^S\left\{n_s\wedge\frac{n_s^2\epsilon_s^2}{d}\right\}
\,\, \mbox{and} \,\,
d^2L^6\sum_{s=1}^S\left\{n_s\wedge\frac{n_s^3\epsilon_s^4}{d^2}\right\}
\lesssim\left(\sum_{s=1}^S\left\{n_s\wedge\frac{n_s^2\epsilon_s^2}{d}\right\}\right)^2.
\label{eq:sample_size}
\end{equation}
\end{theorem}

The upper bound in \Cref{thm-fdp-cox} can be attained by a private stochastic gradient descent algorithm detailed in \Cref{alg:FDP-SGD} in \Cref{subsec-cox-upperss}. We defer further discussions regarding the sample size conditions there, but note here that, under homogeneous per-server privacy budgets and sample sizes, the sample size conditions can be simplified to $n_s \gtrsim_{\log} \max\{d, d^2/S\}$.

The lower bound in \Cref{thm-fdp-cox} is proved in \Cref{app:Cox_lower_proof} using an application of the van Trees inequality \citep[e.g.][]{gill1995applications} and the score attack arguments from \cite{cai2024optimal}.  To the best of our knowledge, this is the first time that a minimax rate has been proven for Cox regression under DP constraints. The minimax rate in the non-private setting is considered folklore, and the lower bound can be attained by those for generalised linear regression (e.g.~Exercise 15.17 in \citealt{wainwright2019high}). A rigorous functional version was derived in \cite{qu2016optimal}.

To better understand this result, recall that under \Cref{assump-1}\ref{assp:eigenvalues}, which imposes that all the eigenvalues of the Hessian matrix are of order $1/d$.  Compared to the analyses based on different normalisations, this $1/d$ eigenvalue rate leads to an extra $d$ in the squared-$\ell_2$ rate relative to constant-order eigenvalues. With $S=1$, the central DP rate is $d^2/n + d^3/(n^2\epsilon^2)$, which agrees with the central-DP regression rates in \cite{cai2021cost}. 

In contrast to \Cref{prop:hazard_lower}, estimating a $d$-dimensional regression coefficient introduces an additional dimension-dependent privacy cost: each server's contribution is reduced from $\min(n_s, n_s^2 \epsilon^2)$ to $\min(n_s, n_s^2 \epsilon^2/d)$. This privacy cost has also appeared in other settings under federated DP, including multivariate mean estimation and low-dimensional linear regression \citep{li2024federated}. The unspecified baseline hazard and the dependence induced by the risk sets in \eqref{eq:Cox_partial_likelihood} require additional sensitivity and convergence arguments beyond those used for private mean estimation and linear regression.

\subsection{Federated differentially private gradient descent}\label{subsec-cox-upperss}
Without the presence of privacy concerns, it is well-studied that the maximum partial likelihood estimator \citep{cox1975partial} enjoys desirable statistical properties including consistency and asymptotic normality \citep[e.g.][]{andersen1982}.  To obtain a version of this estimator that adheres to the federated DP constraint, we propose the FDP-Cox algorithm in \Cref{alg:FDP-SGD}. 

\begin{algorithm} 
\caption{FDP-Cox}
\textbf{Input}: Datasets $\{(T_{s, i}, \Delta_{s, i}, \{Z_{s, i}(t):\, t \in [0, T_{s, i}]\})\}_{s, i = 1}^{S, n_s}$, number of iterations $K$, step size $\eta$, privacy parameters $\{(\epsilon_s, \delta_s)\}_{s \in [S]}$, parameter truncation level $C_\beta > 0$, covariate bound $C_Z > 0$. 

\begin{algorithmic}[1]
\State Set $\beta^{(0)} := 0$ and $d := \mathrm{dimension}(Z_{1, 1})$.
\For{$s \in [S]$}
    \State Set batch size $b_s: = \lfloor n_s / K \rfloor$ and weight 
        $v_s := \min(b_s, b_s^2 \epsilon_s^2 / d)/\{\sum_{u=1}^S \min(b_u, b_u^2 \epsilon_u^2/d)\}$.
\EndFor
\For{$k \in [K]$}
    \For{$s \in [S]$}
    \State Sample independently 
    \begin{equation*}
    W^{(k)}_s \sim \mathcal{N}\left(0, \frac{(2 \log(1/\delta_s) + \epsilon_s)\left\{4 C_Z + \exp(2C_Z \|\beta^{(k-1)}\|_2)(2C_Z + C_Z^2) (1+\log(b_s))\right\}^2}{\epsilon_s^2 b_s^2} I_d\right).
    \end{equation*}
    \EndFor
    \State Set $\beta^{(k)} := \Pi_{C_\beta} \left( \beta^{(k-1)} + \eta \sum_{s=1}^S v_s \left[\dot{\ell}(\beta^{(k-1)}; \{T_{s, i}, \Delta_{s, i}, Z_{s, i}\}_{i=(k-1)b_s+1}^{kb_s}) + W_s^{(k)} \right] \right)$.
\EndFor
\State \textbf{Output}: $\beta^{(K)}$.
\end{algorithmic}
\label{alg:FDP-SGD}
\end{algorithm}

\Cref{alg:FDP-SGD} is a private stochastic gradient descent algorithm \citep[e.g.][]{bassily2014private, bassily2019private, avella2023noisyopt, li2024federated}. The privacy guarantee is achieved through the use of the Gaussian mechanism (\Cref{lemma:gaussian_mechanism}): at each iteration, every server adds noise to its gradient evaluated at the current parameter estimate so that the update it sends satisfies $(\epsilon_s, \delta_s)$-CDP. The magnitude of noise needed to preserve privacy is determined by the $\ell_2$-\textit{sensitivity} (\Cref{def-sensitivity}) of the gradient in \eqref{eq:cox_gradient} -- the proof of its bound can be found in \Cref{lemma:grad_sensitivity2} in \Cref{app:beta_sensitivity}. 
The aforementioned references consider objective functions in the form of sample averages over i.i.d.~terms. The key complication in our setting is that the summands in \eqref{eq:cox_gradient} depend on the full dataset through the weighted sum of at risk subjects, i.e.~$Y_{s, i}(t) = 1$, which are both time-dependent and random.
This is partially reflected in \Cref{lemma:grad_sensitivity2} in \Cref{app:beta_sensitivity}, which shows that under \Cref{assump-1}\ref{assp:covariate_bound} and \ref{assp:coefficient_bound}, the $\ell_2$ sensitivity of \eqref{eq:cox_gradient} is of order $\log(n)/n$.

To satisfy the federated DP constraint, each iteration of \Cref{alg:FDP-SGD} uses an independent batch of data from each server. To optimise for a heterogeneous distributed setting, we use a choice of weights $\{v_s\}_{s=1}^S $ that are proportional to the effective sample size at each server. 

\begin{proposition} \label{prop:fdp_upper}
    Suppose that the data $\{T_{s, i}, \Delta_{s, i}, \{Z_{s, i}(t): \, t \in [0, T_{s, i}]\}_{}\}_{s, i=1}^{S, n_s}$ are i.i.d.~from a distribution satisfying \Cref{assump-1} and the input parameters $C_Z, C_\beta$ are the same as those in Assumptions~\ref{assump-1}\ref{assp:covariate_bound} and \ref{assp:coefficient_bound}. We then have that the output of \Cref{alg:FDP-SGD} $\widehat{\beta}$ satisfies $(\{\epsilon_s, \delta_s\}_{s \in [S]}, K)$-FDP. 
    
    If $\sum_{s=1}^S n_s \geq cd^2$ for a fixed constant $c>1$, set the tuning parameter inputs as $\eta \asymp d$ and $K \asymp \log(\sum_{s=1}^S n_s / d^2)$.  Assume in addition that $\min_{s \in [S]}n_s \gtrsim K$ for all $s \in [S]$, and sample size condition in \eqref{eq:sample_size}. It then holds that
    \begin{equation*} 
    \mathbb{E}\|\widehat{\beta} - \beta_0\|_2^2  \lesssim \frac{d^2 \log^2\left(\sum_{s=1}^S n_s / d^2 \right)}{\sum_{s=1}^S \min\{ n_s, n_s^2 \epsilon_s^2/d\}}  \max_{s \in [S]} \left\{ \log(1/\delta_s) \log^2\left(n_s/\log \left(\sum_{s=1}^S n_s / d^2 \right) \right) \right\}. 
    \end{equation*}
\end{proposition}

\Cref{prop:fdp_upper} provides the theoretical guarantees for the convergence of \Cref{alg:FDP-SGD}, and an upper bound on the minimax rate in \Cref{thm-fdp-cox}. The proof of \Cref{prop:fdp_upper} can be found in \Cref{app:Cox_upper_proof}. The sample size condition $\min_{s \in [S]} n_s \gtrsim K$ ensures that each server has enough independent data points for each of the $K$ rounds. The sample size conditions in \eqref{eq:sample_size} are used to control the Hessian of the distributed likelihood, in particular, to control its condition number in a high probability event. We conjecture that it may be possible to relax this condition by first obtaining private versions of $S^{(1)}(t, \beta)$ and $S^{(0)}(t, \beta)$ that average the local means at each server over a sufficiently fine grid of $[0, 1]$, in order to reduce the bias in approximating $\mu(t, \beta)$ by $\bar{Z}_s(t, \beta)$ at each server.

The convergence result in \Cref{prop:fdp_upper} is also conditional on appropriately chosen tuning parameters. Overestimating $C_Z$ inflates the magnitude of noise added to preserve privacy, while the consequence of under estimating is more severe as it may lead to privacy leakage. Underestimating the truncation level $C_\beta$ would prevent parameter convergence; we therefore recommend a conservatively large truncation level in practice. As in non-private gradient descent, the step size needs to be sufficiently small, yet a smaller step size means that a larger number of iterations are required, which increases the magnitude of noise needed to preserve privacy due to privacy compositions. A sensitivity analysis regarding the tuning parameters is provided in \Cref{subsec-sensitivity_analysis}.  

\section{Extensions: Breslow estimator, survival functions, and label differential privacy}
\label{sec:extensions}

In Sections \ref{sec:cumulative_hazard} and \ref{sec-beta}, we studied the two canonical problems in survival analysis. In this section, we consider three further problems. First, extending the covariate-free setting of \Cref{sec:cumulative_hazard}, we study estimation of the \emph{baseline} cumulative hazard function of the Cox model and provide a FDP version of the Breslow estimator along with its theoretical guarantees in \Cref{subsec:fdp_breslow}. We then consider individual-level survival function estimation in \Cref{subsec:survival_function}. Finally, in \Cref{subsec:label_dp}, we consider an extension to \emph{label} differential privacy, to discuss the effect of allowing the covariates to be publicly available.

\subsection{Federated differentially private Breslow estimator}
\label{subsec:fdp_breslow}

The cumulative hazard function in the Cox model is of the form $\Lambda(t \mid Z) = \int_0^t\exp\{\beta_0^{\top}Z(s)\} \dint \Lambda_0(s)$, depending on both the regression coefficient $\beta_0$ and the cumulative baseline hazard function $\Lambda_0(\cdot)$.  To estimate $\Lambda_0(\cdot)$, in the discussion of \cite{cox1972regression}, Breslow proposed an estimator that extends the Nelson--Aalen estimator, by replacing the at-risk set size with a covariate-weighted analogue induced by the Cox model. 

To adhere to the federated DP constraints, we construct a FDP version of the Breslow estimator, described in \Cref{app:FDP_breslow_alg}. The construction closely follows the FDP Nelson--Aalen estimator detailed in \Cref{alg:FDP-NA}, with two modifications: the at-risk set is replaced by a covariate-weighted analogue evaluated at a $\widehat{\beta}$ estimate, and the truncation level is adjusted to account for the weighting. We have the following theoretical guarantees for our FDP Breslow estimator.

\begin{proposition}\label{prop:hazard_upper}
Suppose that the data $\{(T_{s, i}, \Delta_{s, i}, \{Z_{s, i}(t): \, t\in [0, T_{s, i}]\})\}_{s, i = 1}^{S, n_s}$ are i.i.d.~from a distribution satisfying \Cref{assump-1}.  Let $\widehat{\Lambda}_{\mathrm{B}}(\cdot)$ be the output of \Cref{alg:FDP-Breslow}, with inputs $(\widehat{\beta},\hat p)$ jointly independent of the data, where $\widehat{\beta}$ estimates $\beta_0$ satisfying $\|\widehat{\beta}\|_2 \leq C_\beta$; $\hat{p}$ estimates $p_0:=\mathbb{P}(Y(1)=1)$, satisfying $\mathbb{P}(\hat p>19p_0/18)\lesssim N_S^{-1}$ (recalling the definition in \eqref{eq-effect-ss}) and $\hat{p}\geq C_p/2$ a.s.; and $C_Z$ satisfies \Cref{assump-1}\ref{assp:covariate_bound}.  If $\min_{s\in[S]}n_s\geq C\log(e+N_S)$ for a sufficiently large constant $C>0$, then $\widehat{\Lambda}_{\mathrm{B}}(\cdot)$ satisfies $(\{\epsilon_s,\delta_s\}_{s\in[S]},1)$-FDP and  
\begin{align}
        \mathbb{E} \left[\sup_{t \in [0, 1]} |\widehat{\Lambda}_{\mathrm{B}}(t) - \Lambda_0(t)|\right] &\lesssim  N_S^{-1/2} \log \left(\sum_{s=1}^S n_s \right)^2 \max_{s \in [S]} \log(1/\delta_s)
        \notag\\ & \quad 
        + \mathbb{E}[\|\widehat{\beta} - \beta_0\|_2] N_S^{-1} \sum_{s=1}^S \min(n_s, n_s^2 \epsilon_s^2) \log(n_s).
        \label{eq:breslow_upper}
    \end{align}
\end{proposition}

The proof of \Cref{prop:hazard_upper} can be found in \Cref{subsec-proof-hzard-upper}. The upper bound contains two terms, with the first being the same as in \Cref{thm:na_upper}, reflecting the FDP cost of estimating a univariate integral. The second term in \eqref{eq:breslow_upper} is a direct consequence of using an estimate $\widehat{\beta}$ of $\beta_0$ within $\widehat{\Lambda}$ -- a similar dependence can be seen in, for instance, Theorem 3.4 in \cite{andersen1982}. 

One way of specifying a full FDP procedure is by taking the FDP Breslow estimator's inputs~$\hat{p}$ to be a Gaussian perturbed sample mean (see details in \Cref{app-hazard-atrisk}) and $\widehat{\beta}$ as the output of \Cref{alg:FDP-SGD} using independent datasets. We denote this procedure by $\widehat{\Lambda}_{\mathrm{B}, \widehat{\beta}}$. By \Cref{prop:fdp_upper}, the resulting rate is 
\begin{equation*}
\mathbb{E} \left[\sup_{t \in [0, 1]} |\widehat{\Lambda}_{\mathrm{B}}(t) - \Lambda_0(t)|\right] \lesssim_{\log}  \frac{1}{\sqrt{\sum_{s=1}^S \min(n_s, n_s^2 \epsilon_s^2)}} + \frac{d}{\sqrt{\sum_{s=1}^S \min(n_s, n_s^2 \epsilon_s^2/d)}},
\end{equation*}
so the error is dominated by the dimension-dependent error in $\beta_0$ estimation. 

In \Cref{prop:cumulative-baseline-lower} of \Cref{app-cumulative-hazard}, we show a lower bound of $1 \wedge \left(d/\sum_{s=1}^S \min\{n_s, n_s^2 \epsilon_s^2\} \right)^{1/2}$ for $(\{\epsilon_s, \delta_s\}_{s \in [S]}, 1)$-FDP mechanisms. The same lower bound holds for any $K$-iteration mechanism in the homogeneous setting with the additional restriction that $n/K$ samples used per-iteration at each server. We can split samples to use across the $\hat{p}$, $\widehat{\beta}$ and $\widehat{\Lambda}_{B}$ estimation procedures correspondingly to obtain a version of $\widehat{\Lambda}_{\mathrm{B}, \widehat{\beta}}$ that satisfies this constraint. Comparing upper and lower bounds in this setting, the gap is an extra factor of $\sqrt{d}$ in the non-private rate and $d$ in the private rate. We conjecture this gap is due to the use of a plug-in estimator. An alternative approach may be to consider debiased estimators, which have been studied without privacy constraints in \cite{fang2017testing}, for example. Extending this construction to satisfy FDP constraints is nontrivial, requiring a private estimator of the score correction, in which the privacy-preserving noise may offset estimation accuracy gains from debiasing. We leave this for future work.

\subsection{Estimation of the survival function}
\label{subsec:survival_function}
A natural downstream goal of Cox regression and cumulative baseline hazard estimation is predicting survival probabilities. The conditional survival function of the event time $\widetilde{T}$, given a covariate trajectory $\{z(s), s\in [0, 1]\}$, is $S(t;z) = \exp\left\{- \int_0^t \exp\{\beta^{\top}_0 z(s)\} \,\mathrm{d}\Lambda_0(s) \right\}$.  We therefore study private estimation of $S(t; z)$, uniformly over both $t \in [0, 1]$ and covariate paths. \Cref{prop-optimality-survival} provides the corresponding minimax rates. 

\begin{proposition}\label{prop-optimality-survival}
Denote by $\mathcal{P}$ the class of distributions satisfying \Cref{assump-1}, and by $\mathcal{Q}$ the class of $(\{(\epsilon_s,\delta_s)\}_{s\in[S]},K)$-FDP mechanisms over all $K\in\mathbb{N}_+$, with $\delta_s\log(1/\delta_s)\lesssim\epsilon_s^2/d$ for all $s\in[S]$. Let $\mathcal{Z}$ denote the set of covariate paths satisfying the boundedness and Lipschitz conditions in \Cref{assump-1}\ref{assp:covariate_bound} and \ref{assp:covariate_Lipschitz}. The lower bound below holds under these conditions. For the upper bound, assume additionally that the coefficient-estimation sample sizes satisfy \eqref{eq:sample_size}, and that the independent inputs $(\widehat{\beta},\hat{p})$ and the Breslow sample sizes satisfy the conditions of \Cref{prop:hazard_upper}. Then
\begin{align}
    1 \wedge \frac{d^2}{\sum_{s = 1}^S \min(n_s, \, n_s^2\epsilon_s^2/d)} &\lesssim \inf_{Q \in \mathcal{Q}} \inf_{\widehat{S}} \sup_{P \in \mathcal{P}} \mathbb{E}\left[\sup_{t \in [0, 1]} \sup_{z \in \mathcal{Z}} \left|\widehat{S}(t; z) - S(t; z) \right|^2\right] \notag\\
    & \lesssim 1 \wedge \frac{d^2}{\sum_{s = 1}^S \min(n_s, \, n_s^2\epsilon_s^2/d)} \mathrm{polylog} \left(\{n_s, \epsilon_s, \delta_s\}_{s \in [S]}, d\right).
    \label{eq-survfunc-rate}
\end{align}
\end{proposition}

The detailed version of the upper bound for general $\{(n_s, \epsilon_s, \delta_s)\}_{s\in [S]}$ can be found in the proof of \Cref{prop-optimality-survival}, which is deferred to \Cref{app-proof-prop-optimality-survival}. The upper bound in \eqref{eq-survfunc-rate} is obtained by the plug-in estimator $\widehat{S}(t;z) = \exp\left\{- \int_0^t \exp\{\widehat{\beta}^{\top} z(s)\} \,\mathrm{d}\widehat{\Lambda}_B(s) \right\}$, where the data are split to obtain $\hat{p}$ via \Cref{alg:FDP-probabilities}, $\widehat{\beta}$ via \Cref{alg:FDP-SGD}, and $\widehat{\Lambda}_{\mathrm{B}, \widehat{\beta}}$ via \Cref{alg:FDP-Breslow} with $\widehat{\beta}$ as input. This therefore leads to the same estimation error for $\widehat{S}$ as the $\ell_2^2$ error for $\beta_0$, as it dominates the estimation error of the cumulative baseline hazard. For the lower bound, we show that there exists a covariate path for which its survival function loss is bounded below by a multiple of the $\ell_2^2$ loss for $\beta_0$. Together, these bounds show that the difficulty of estimating survival functions uniformly over covariate paths is governed by the difficulty of estimating the regression coefficients rather than the cumulative baseline hazard. 

\subsection{Label differential privacy}\label{subsec:label_dp}
We have so far treated a survival data point as the full triple $(T, \Delta, Z)$, upon which the privacy constraints are imposed.  In some applications, such as healthcare datasets, however, covariates such as demographic information may already be publicly available \citep[e.g.][]{bussone2020trust}, and only the time and censoring indicators require privacy protection.  This motivates a relaxed privacy model in which the covariates are public and only other information are privatised.

A natural formalisation is the label-DP framework \citep{chaudhuri2011sample}.  In our survival setting, to make the definition well posed, we include \Cref{def:labelCDP} below and consider time-invariant covariates only for simplicity.

\begin{definition}[Label-CDP] \label{def:labelCDP}
For $\epsilon > 0$ and $\delta \geq 0$, a privacy mechanism $M$ is an $(\epsilon, \delta)$-label-CDP mechanism for survival data, if it is a conditional distribution on space $\mathcal{R}$ that satisfies 
\begin{align*}
M \left(R \in A\mid \{(T_i,\Delta_i,Z_i)\}_{i\in[n]}\right) \leq e^\epsilon
M \left( R\in A\mid \{(T_i',\Delta_i',Z_i)\}_{i\in[n]}\right)+\delta.
\end{align*}
for all measurable set $A \in \sigma(\mathcal{R})$, all possible $\{Z_i\}_{i\in[n]}$ and all possible $\{(T_i, \Delta_i, T_i', \Delta_i')\}_{i \in [n]}$ such that $\sum_{i = 1}^n\mathbbm{1}\{(T_i, \Delta_i) \neq (T_i', \Delta_i')\} \leq 1$. The label-FDP definition follows from substituting the label-CDP constraint in place of the CDP constraint in \Cref{def:FDP}; the full details can be found in \Cref{def:label-FDP} in \Cref{app-sec-labelDP}. 
\end{definition}

We have the following result for Cox regression under label-FDP. 
\begin{proposition} \label{prop-label-FDP}
Let $\mathcal{P}$ denote the class of distributions satisfying \Cref{assump-1} with time-invariant covariates, and let $\mathcal{Q}$ denote the class of $(\{(\epsilon_s,\delta_s)\}_{s\in[S]},K)$-label-FDP mechanisms over all $K\in\mathbb{N}_+$, restricted to mechanisms whose batch sizes satisfy $b_s^{(k)}\gtrsim d\log(d)$ for every server $s\in[S]$ and iteration $k\in[K]$. Suppose that $\delta_s\log(1/\delta_s)\lesssim\epsilon_s^2/d$ for all $s\in[S]$. Then
\begin{align*}
    1 \wedge \frac{d^2}{\sum_{s = 1}^S\min(n_s, \,n_s^2\epsilon^2_s/d)} \lesssim \inf_{Q \in \mathcal{Q}} \inf_{\widehat{\beta}} \sup_{P \in \mathcal{P}} \mathbb{E}_{Q, P} \big\{\|\widehat{\beta} - \beta_0\|^2_2\big\} \lesssim 1 \wedge \frac{d^2 \,\mathrm{polylog}(\{n_s, \epsilon_s, \delta_s\}_{s \in [S]}, d) }{\sum_{s = 1}^S\min(n_s, \,n_s^2\epsilon^2_s/d)}.
\end{align*}
For the upper bound we additionally require  $n_s \gtrsim d\log(d) \log(\sum_{s=1}^S n_s/d^2)$ for all $s \in [S]$ and the condition in \eqref{eq:sample_size}.
\end{proposition}

The upper bound in \Cref{prop-label-FDP} is achieved by \Cref{alg:FDP-SGD}, since any DP algorithm automatically satisfies label-DP. More details and a numerical comparison are provided in \Cref{subsec-labelDP-sims}.  To the best of our knowledge, this is the first ever minimax rate (up to polylogarithmic factors) under label-FDP constraints in any problems.  For Cox regression, we establish the minimax rate (up to polylogarithmic factors) when each server's batch size is at least of order $d\log(d)$ observations. The batch-size condition ensures concentration of the conditional score covariance, such that the leading privacy cost comes from protecting the response. The arguments used in the proof suggest that the same label-DP lower bound can be shown for other regression models with suitable score and curvature conditions. 

The single-server case is also of independent interest; we are unaware of previous work with matching bounds for regression coefficient estimation under label-CDP. Our finding is consistent with \cite{zhao2025labelDP}, who show that, for nonparametric regression, label central DP has the same minimax rate as full central DP. Our result identifies a broad regime where label-FDP behaves like full FDP, although the precise transition at which label-DP can yield sharper rates remains open,  particularly since \cite{zhao2025labelDP} show that label privacy can improve rates in the local model of DP.

\section{Numerical experiments}\label{sec-numerical}
In this section, we conduct numerical experiments to support the theoretical results for our private Cox regression and Breslow estimators. We first focus on central DP in \Cref{subsec-sim-cox}, followed by the federated DP results in \Cref{sec-simu-fdp} and an application to real data in \Cref{sec:real_data}. We provide some sensitivity analysis, further simulations exploring the effect of censoring, and comparisons to theoretical results in \Cref{app:simulations}. The code for reproducing our experiments can be found at \url{https://github.com/EKHung/FDPCox/tree/main/inst}.

\subsection{Central differential privacy}
\label{subsec-sim-cox}

To better investigate the effects of  privacy budgets, sample sizes and dimensions, we start with a simple case where there is only one server, i.e.~$S = 1$, which is equivalent to the central DP constraint. Due to the different types of constraints on the interactions allowed within the algorithm, we include a central DP version of estimating $\beta_0$ as \Cref{alg:cdp_sgd} of \Cref{app:cdp_Cox}, improving constant and logarithmic factors by tighter privacy composition analysis.  We can directly set $S = 1$ in \Cref{alg:FDP-NA} for a central DP estimator of $\Lambda$.  

\textbf{Simulation setup.} To investigate the effects of $n$ and $\epsilon$, on the \Cref{alg:cdp_sgd}, we generate i.i.d.~samples from an $\mathrm{Exp}(1)$ distribution. We generate the data fed into \Cref{alg:cdp_sgd} from \eqref{eq-hazard}, with $\beta_0=(0, 0.5, 0.8)$, $\lambda_0(t) = 1$ for all $t$ and covariates generated as $Z_{ij} \overset{\textrm{i.i.d.}}{\sim} \textrm{Uniform}(-1/\sqrt{d}, 1/\sqrt{d})$ for $i \in [n]$ and $j \in [d]$. The censoring times in both cases are generated from an Exp(0.3) distribution, and we take $T_i := \min(\tilde{T}_i, C_i, 1)$, $\Delta_i := \mathbbm{1}\{\tilde{T}_i \leq \min(C_i, 1)\}$. In our numerical experiments, we vary the privacy budget as  $\epsilon\in \{0.6, 0.8, 1, 1.5, 3, 6\}$ and set $\delta = 0.001$.  For each setting, we vary the number of samples $n \in \{2000, 4000, \dots, 14000\}$. To investigate the effect of the dimension $d$, we fix $n=10000$, vary $d \in \{2, \dots, 8\}$ and generate the coefficients for the simulations by taking the $j$-th coordinate to be $\{\beta_0\}_j \overset{\mathrm{i.i.d.}}{\sim}N(0, 1)$ for $j \in [d]$ and then truncating $\beta_0$ to $B_1(0)$. 

In our numerical experiments, we post process every server's private binary tree in \Cref{alg:FDP-NA} to enforce each parent node of the private binary tree to be the sum of its children, following the approach in Section 4.1 in \cite{hay2009boosting}. This improves numerical performances without incurring further privacy leakage.

\medskip
\noindent \textbf{Evaluation metrics.} We carry out 200 Monte Carlo experiments for each setting and report the mean and standard error of $\sup_{t \in [0, 1]} |\widehat{\Lambda}_\mathrm{NA} - \Lambda|$ across the 200 repetitions, where $\widehat{\Lambda}_\mathrm{NA}$ is the output from \Cref{alg:FDP-NA}. We do the same for $\|\widehat{\beta} -\beta_0\|_2^2$, where $\widehat{\beta}$ is the privatised output from \Cref{alg:cdp_sgd}. 

\medskip
\noindent \textbf{Tuning parameters.}  The two tuning parameters in \Cref{alg:cdp_sgd} are the step size $\eta$ and the constant $C$ in the number of iterations $C \log(n/d^2)$. For the experiments in \Cref{fig:cdp_cox} we set $\eta=0.5$ and $C=20$. The constants we use from \Cref{assump-1}\ref{assp:covariate_bound} and \ref{assp:coefficient_bound} are $C_Z=C_\beta=1$. Some sensitivity analysis for tuning parameters can be found in \Cref{subsec-sensitivity_analysis}. The two tuning parameters in \Cref{alg:FDP-NA} are $C_p$ and $C_\lambda$, which we set to $0.01$ and $2$ respectively. 

\begin{figure}[htbp]
  \centering
\includegraphics[width=0.99\textwidth]{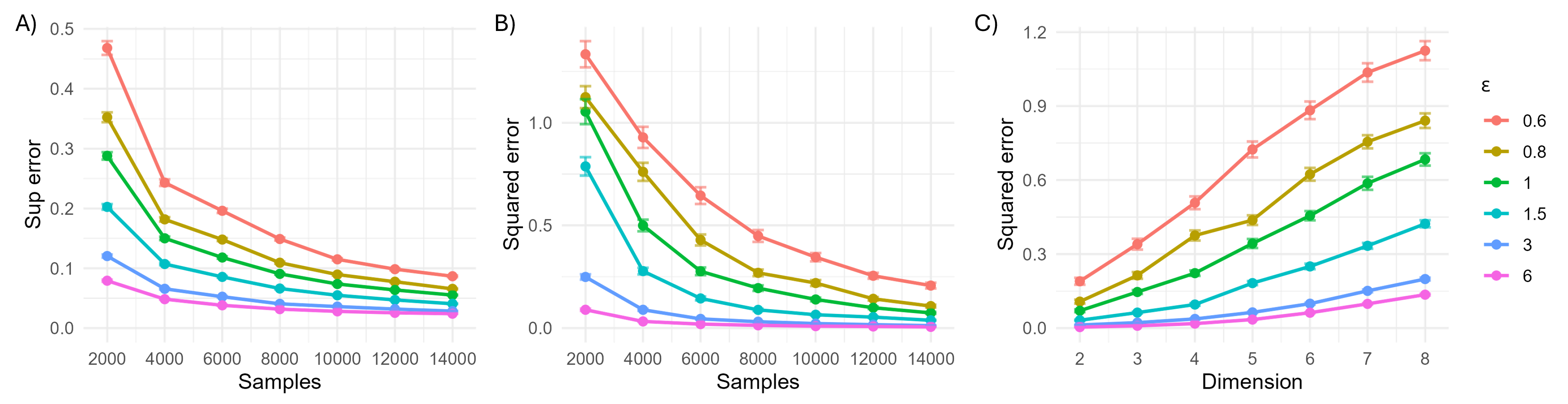}
  \caption{Simulation results for CDP Nelson--Aalen (panel A) and Cox coefficient estimation (panel B), with varying sample sizes and privacy budgets; and for CDP Cox regression with varying dimensions (panel C).}
  \label{fig:cdp_cox}
\end{figure}

\medskip
\noindent \textbf{Results.} The simulation results for estimating $\Lambda$ and $\beta_0$ are collected in \Cref{fig:cdp_cox}. As expected from the theoretical results, \Cref{fig:cdp_cox}(A) and (B) show that the estimation error of $\Lambda$ and $\beta_0$ and decreases with increasing sample size $n$ and privacy budget $\epsilon$.  \Cref{fig:cdp_cox}(C) depicts how the rate of estimating $\beta_0$ depends on the dimension $d$. It shows that the estimation error increases with increasing $d$.  The gaps between different $\epsilon$ curves narrows as $\epsilon$ grows, reflecting roughly the squared dependence on the reciprocal of $\epsilon$ when $\epsilon$ is small, and the transition to the non-private rate when $\epsilon$ is large.

\subsection{Federated differential privacy} \label{sec-simu-fdp}

\textbf{Simulation setup.} We consider a homogeneous setting and fix the sample size $n_s = n= 20000$, $\epsilon_s = \epsilon$, $s \in [S]$, and vary the number of servers $S \in \{2, 4, 8, 6, 10, 12\}$.   The other simulation setups, tuning parameters, and reported evaluation metrics are the same as in \Cref{subsec-sim-cox}. 

\medskip
\noindent \textbf{Results.} In \Cref{fig:FDP_sims}(A), we report the supremum-norm loss in estimating the cumulative hazard with \Cref{alg:FDP-NA}. In \Cref{fig:FDP_sims}(B), we show the squared-$\ell_2$ loss in estimating $\beta_0$ using \Cref{alg:FDP-SGD}. We see the general trend of improving accuracy as the number of servers and privacy budgets increase, echoing our theoretical findings.

\begin{figure}[htbp]
  \centering
\includegraphics[width=0.95\textwidth]{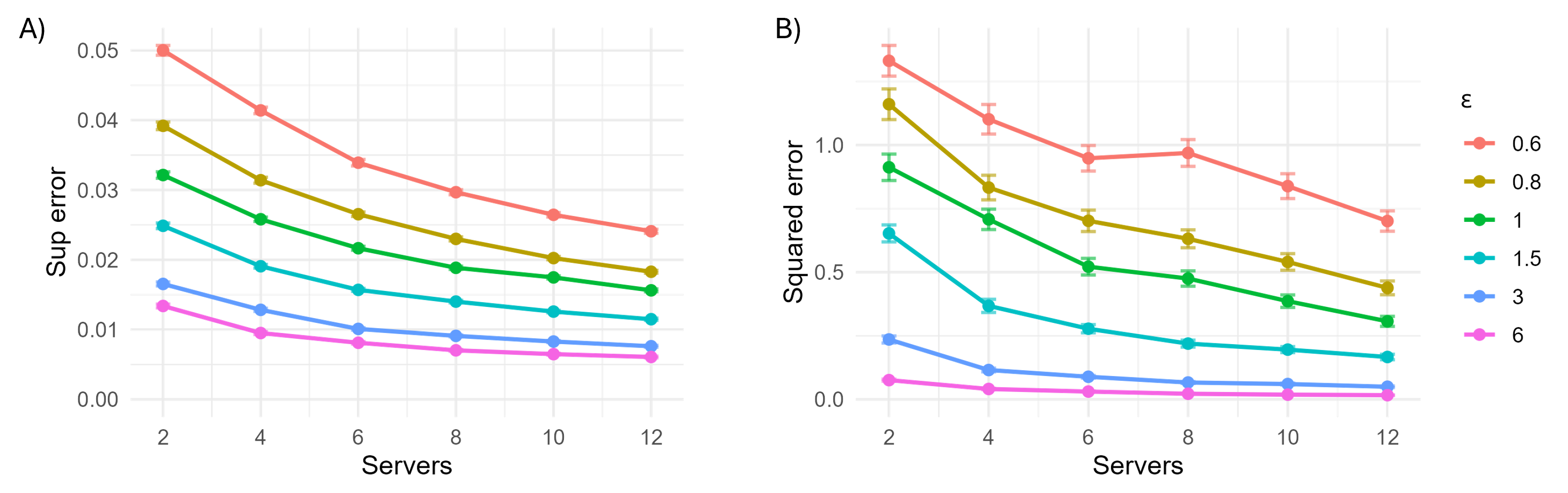}
  \caption{Simulations results from \Cref{alg:FDP-NA} (panel A) and \Cref{alg:FDP-SGD} (panel B), varying the number of servers and the $\epsilon$ privacy budget.}
  \label{fig:FDP_sims}
\end{figure}

\medskip
\noindent \textbf{Finite sample considerations.} We observe in \Cref{fig:FDP_sims} that, even with a much larger total number of samples, the estimation error of $\beta_0$ remains comparable to that in the left panel of \Cref{fig:cdp_cox} with $S = 1$ and smaller sample sizes.  This is because \Cref{alg:cdp_sgd} uses the entire dataset in each stochastic gradient descent iteration, and the privacy budget is allocated using a tighter privacy accounting argument; see \Cref{app:cdp_Cox} for details. However, under the FDP constraint in \Cref{def:FDP}, each server uses new data for each iteration's transcript; this interaction constraint plays a key role in the proof of the lower bound. 

Motivated by this observation, we propose a fully-interactive federated DP algorithm in \Cref{alg:FDP-SGD-interactive}, which ensures $(\epsilon, \delta)$-CDP guarantees within each server to maintain the privacy of the transcripts communicated across servers. In \Cref{fig:fdp_interactive}, we demonstrate the results of using this algorithm, showing improved estimation error even with a 75\% reduction in the number of samples. This leaves an important open problem about obtaining a minimax lower bound for fully-interactive federated DP mechanisms, and we leave this as future work. 

\subsection{Real data example} \label{sec:real_data}
We apply Algorithms~\ref{alg:FDP-NA} and \ref{alg:FDP-SGD-interactive} to the Rotterdam breast cancer dataset \citep{Royston_Altman_2013} from \texttt{R survival::rotterdam} \citep{survival-package}.  We split the dataset randomly to create multiple servers for the federated DP experiments. The dataset consists of 2982 primary breast cancer patients and 10 covariates, and the observed times range from 36 to 7043. By defining both deaths and cancer re-occurrences to be event times, 42.6\% of the observations are considered censored. We select the five categorical covariates for scaling convenience, two of which are significant at the 0.05 level when fitting a Cox regression model non-privately to the full set of covariates. We consider estimating the cumulative hazard on $[0, 3500]$, where 19\% of observations are still at risk after time $3500$.  We fix $\delta = 0.001$ and compare the effects of both the privacy budget $\epsilon \in \{3, 6, 9, 12, 15\}$ and the number of servers on the estimation accuracy. The full details of our experiments can be found in \Cref{app:real_data}.

We report the average value of $\sup_{t \in [0, 3500]} |\widehat{\Lambda}_{\mathrm{NA}}(t) - \tilde{\Lambda}(t)|$ and its standard error in \Cref{fig:real_data_experiments}(A), where $\tilde{\Lambda}(t)$ is the non-private Nelson--Aalen estimator. We report the mean standardised square error $\|\widehat{\beta} - \tilde{\beta}\|_2^2/\|\tilde{\beta}\|_2^2$ and its standard error in \Cref{fig:real_data_experiments}(A), where $\tilde{\beta}$ is the coefficient vector from fitting a Cox regression model to the variables without DP constraints. We see that the estimation error decreases as the privacy budget $\epsilon$ or the number of servers~$S$ increases, as expected from our theoretical results. 
\begin{figure}[htbp]
  \centering
\includegraphics[width=0.95\textwidth]{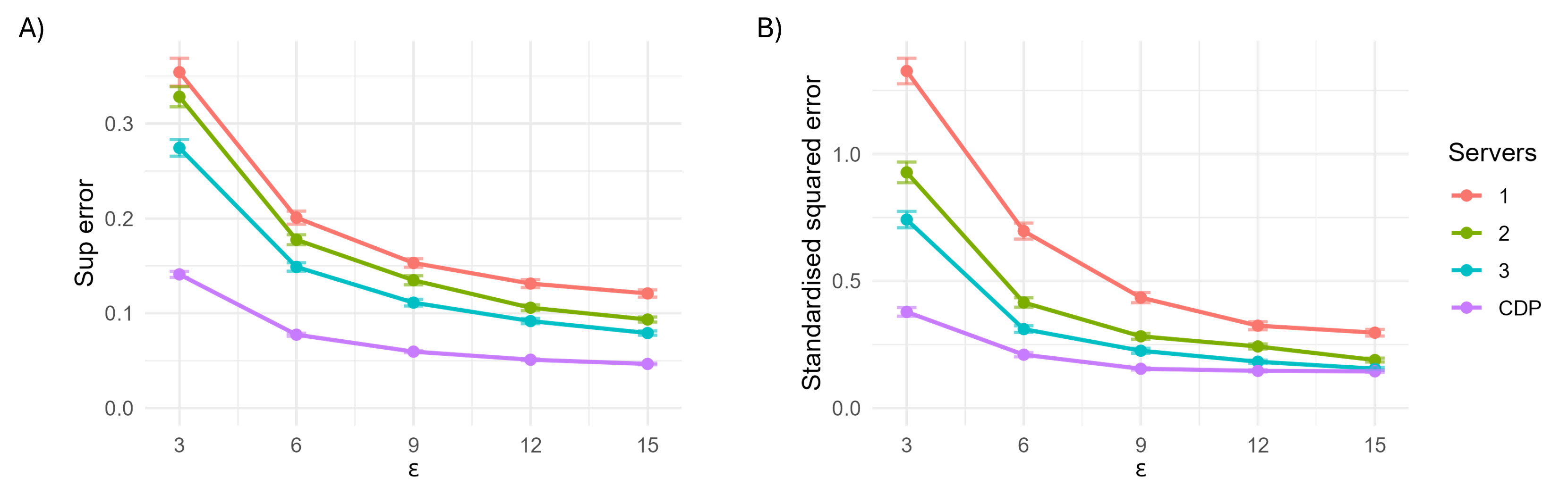}
  \caption{Errors of the $\widehat{\Lambda}_{\mathrm{NA}}$ (panel A) and $\widehat{\beta}$ (panel B) estimators. The number of servers refers to 1, 2, 3 servers with 994 observations each, and `CDP' means a single server with all 2982 observations.}
  \label{fig:real_data_experiments}
\end{figure}

\section{Conclusions}\label{sec-conclusions}

This paper is the first systematic minimax study of cumulative hazard estimation and Cox regression under central and federated differential privacy. By laying the foundation for rigorous private survival analysis, our work also opens up several promising directions for future research. Beyond the open problems discussed in \Cref{sec:extensions}, we highlight a few more interesting directions below. 

Firstly, in our distributed setting, we assume that the observations at each server are generated from the same distribution. While the stratified Cox partial likelihood allows for different baseline hazards across the different servers when estimating coefficients, more complicated methods would be needed to estimate heterogeneous regression coefficients. \cite{li2016transferCox} propose a transfer learning-based method that uses a $\ell_{2, 1}$-norm penalty to try to share the sparsity representation across the source and target. Both transfer learning and variable selection under DP constraints would be interesting directions for further investigation. 

Secondly, uncertainty quantification under privacy constraints is an important next step beyond point estimation. Bootstrap-type methods \citep[e.g.][]{WangChengAwan2022DPBootstrap}, subsampling or resampling approaches \citep[e.g.][]{chadha2024resampling, WangChengAwan2022DPBootstrap} and covariance-based private inference methods \citep[e.g.][]{avella2023noisyopt} have been considered in the literature.  In our setting, the main technical difficulty is that the final private estimator is produced by an iterative noisy optimisation procedure, rather than by a single additive noise release.  As a result, both non-asymptotic bootstrap-type inference and deconvolution-based corrections are substantially more complicated here than in simpler private estimation problems. We provide more discussion on heterogeneity between servers and uncertainty quantification in \Cref{sec-additional-discussions}.

\section*{Acknowledgements}

Hung is supported by the Chancellors' Scholarship scheme and the Statistics Centre for Doctoral Training at the University of Warwick. Yu is partially supported by the Philip Leverhulme Prize and EPSRC programme grant EP/Z531327/1.

\newpage
\bibliographystyle{apalike}
\bibliography{bibliography}

\appendix

\newpage
\section*{Appendices}

In \Cref{app:cdp_Cox}, we present results on central DP, including an intermediate result of \Cref{sec-beta}, and algorithms for a fully-interactive federated DP Cox estimation procedure used in \Cref{sec-numerical}. \Cref{app:simulations} contains additional simulation results and the details of the real data experiments. \Cref{sec-additional-discussions} elaborates on the discussions in conclusion. The technical details and proofs of results in Sections~\ref{sec:cumulative_hazard}, \ref{subsec:fdp_breslow} and \ref{subsec:survival_function} are in Appendix \ref{app-cumulative-hazard}, and those for \Cref{sec-beta} are collected in \Cref{app-beta}. \Cref{app-sec-labelDP} contains the proofs of the label-DP results in \Cref{subsec:label_dp} and
additional numerical comparisons. Some additional definitions and lemmas can be found in \Cref{app-background}. We also introduce the additional notation $a_n \lesssim_{\log} b_n$ if $a_n \lesssim b_n$ up to poly-logarithmic factors. 

\renewcommand{\contentsname}{Contents of Appendices}
\addtocontents{toc}{\protect\setcounter{tocdepth}{5}}
\tableofcontents

\section{Cox regression under central DP}\label{app:cdp_Cox}

We first present a direct corollary of \Cref{thm-fdp-cox} by setting $S = 1$ on the minimax lower bound on the estimation error of $\beta_0$ subject to central DP. 

\begin{corollary} \label{cor-beta-lowb-central}
Denote by $\mathcal{P}$ the class of distributions satisfying \Cref{assump-1}, and denote by $\mathcal{Q}$ the class of $(\epsilon, \delta)$-CDP mechanisms, but allow the assumption on the upper bound for the eigenvalue of the Hessian in \Cref{assump-1}\ref{assp:eigenvalues} to be weakened from $\rho_+/d$ to $\rho_+$. If $d \delta \log (1/\delta) \lesssim  \epsilon^2$, then 
    \begin{equation*}
        \inf_{Q \in \mathcal{Q}} \inf_{\widehat{\beta}} \sup_{P \in \mathcal{P}} \mathbb{E}_{P,Q} \|\widehat{\beta} - \beta_0 \|_2^2 \gtrsim 1 \wedge \max \left( \frac{d^2}{n}, \frac{d^3}{n^2 \epsilon^2} \right).
    \end{equation*}
    \label{prop:CDP_lowerbound}
\end{corollary}

Comparing \Cref{prop:CDP_lowerbound} against \Cref{prop:FDP_lower_bound}, we see that in a homogeneous setting, i.e.~$n_s = n, \epsilon_s = \epsilon$, and $\delta_s = \delta$ for all $s \in [S]$, the private term in the federated DP square risk is inflated by $S$ compared to the central DP rate from using $Sn$ samples.

An immediate consequence of \Cref{thm-fdp-cox}, in particular \Cref{prop:fdp_upper}, also leads to an upper bound for central DP.  As mentioned at the start of \Cref{subsec-sim-cox}, rather than simply setting $S = 1$ in \Cref{alg:FDP-SGD}, we provide an alternative algorithm detailed in \Cref{alg:cdp_sgd}. Instead of splitting data in each iteration, \Cref{alg:cdp_sgd} uses the entire data at each iteration, which can lead to smaller absolute constant and log factors. The privacy guarantees hold due to the tighter privacy composition in R\'{e}nyi DP \citep[RDP,][see \Cref{def:rdp}]{mironov2017renyi} and the equivalence between $(\alpha, \epsilon)$-RDP to $(\epsilon, \delta)$-CDP guarantees in \Cref{lemma:RDP_to_CDP}. The convergence guarantee for \Cref{alg:cdp_sgd} is presented in \Cref{prop:cdp_sgd}.

\begin{algorithm}[htbp]
\caption{CDP-Cox}
\label{alg:cdp_sgd}
\textbf{Input}: Dataset $\{(T_i, \Delta_i, \{Z_i(t): t \in [0, T_i]\})\}_{i \in [n]}$,  number of iterations $K$, step size $\eta$, privacy parameters $\epsilon, \delta > 0$, parameter truncation level $C_\beta > 0$, covariate bound $C_Z > 0$

\begin{algorithmic}[1]
\State Set $\beta^{(0)} := 0$ and $d := \mathrm{dimension}(Z_{1})$
\For{$k \in [K]$}
    \State Sample independently 
    \begin{equation*}
    W^{(k)} \sim \mathcal{N}\left(0, \frac{(4 C_Z + \exp(2C_Z \|\beta^{(k-1)}\|_2)(2C_Z + C_Z^2) (1+\log(n)))^2}{n^2} \frac{2 \log(1/\delta)/\epsilon + 1}{\epsilon/K} I_d \right).
    \end{equation*}
    \State Set $\beta^{(k)} := \Pi_{C_\beta} \left( \beta^{(k-1)} + \eta \left[\dot{\ell}(\beta^{(k-1)}; \{T_{i}, \Delta_{i}, Z_{i}\}_{i=1}^n) + W^{(k)} \right]\right)$.
\EndFor
\State \textbf{Output}: $\beta^{(K)}$.
\end{algorithmic}
\end{algorithm}

\begin{proposition}
\label{prop:cdp_sgd}
    Suppose that $\{T_i, \Delta_i, (Z_i(t))\}_{i=1}^n$ are i.i.d.~from a distribution that satisfies \Cref{assump-1}, but allow the assumption on the upper bound for the eigenvalue of the Hessian in \Cref{assump-1}\ref{assp:eigenvalues} to be weakened from $\rho_+/d$ to $\rho_+$. Denote the corresponding lower and upper eigenvalue bounds by $\mu$ and $L$ respectively. If $C_\beta, C_Z$ are the constants in \Cref{assump-1}, the output $\beta^{(K)}$ from \Cref{alg:cdp_sgd} is $(\epsilon, \delta)$-CDP. If in addition we have $n > cd \max\{d, (1+\log(n))^3\}$, for fixed $c>1$, set $K \asymp (L/\mu)\log(n/d^2)$, then for $\eta \asymp 1/L$ and large enough sample size satisfying $n^{1/2} \gtrsim_{\log} d$, we have that 
    \begin{equation*}
        \mathbb{E}[\|\beta^{(K)} - \beta_0\|^2] \lesssim \frac{d^2}{n} + \frac{d^3 \log(n)^3\log(1/\delta)}{n^2 \epsilon^2}.
    \end{equation*}
\end{proposition}

The proof of \Cref{prop:cdp_sgd} is deferred to \Cref{app:Cox_upper_proof}.  The key advantage of considering RDP is its tighter composition results: it can be shown that an equivalent privacy guarantee can be achieved with a smaller logarithmic factor of noise, and in particular, it is possible to obtain matching minimax rates under relaxing the assumption on the upper bound for the eigenvalue of the Hessian in \Cref{assump-1}\ref{assp:eigenvalues} from $\rho_+/d$ to $\rho_+$. 

In light of the tighter composition results in RDP, in \Cref{sec-simu-fdp} we provided experiment results using a fully-interactive mechanism where the per-server privatised gradients at each iteration are computed in a similar way to \Cref{alg:cdp_sgd}. For completeness, we include this fully-interactive FDP mechanism as \Cref{alg:FDP-SGD-interactive}.

\begin{algorithm}[htbp]
\caption{FDP-Cox-fully-interactive}
\label{alg:FDP-SGD-interactive}
\textbf{Input}: Datasets $\{(T_{s, i}, \Delta_{s, i}, \{Z_{s, i}(t):\, t \in [0, T_{s, i}]\})\}_{s, i = 1}^{S, n_s}$, number of iterations $K$, step size $\eta$, privacy parameters $\{(\epsilon_s, \delta_s)\}_{s \in [S]}$, parameter truncation level $C_\beta > 0$, covariate bound $C_Z > 0$. 

\begin{algorithmic}[1]
\State Set $\beta^{(0)} := 0$ and $d := \mathrm{dimension}(Z_{1, 1})$.
\For{$s \in [S]$}
    \State Set weight 
        $v_s := \min(n_s, n_s^2 \epsilon_s^2 / d)/\{\sum_{u=1}^S \min(n_u, n_u^2 \epsilon_u^2/d)\}$.
\EndFor
\For{$k \in [K]$}
    \For{$s \in [S]$}
    \State Sample independently 
    \begin{equation*}
    W^{(k)}_s \sim \mathcal{N}\left(0, \frac{\{4 C_Z + \exp(2C_Z \|\beta^{(k-1)}\|_2)(2C_Z + C_Z^2) (1+\log(n_s))\}^2}{n_s^2} \frac{2 \log(1/\delta_s)/\epsilon_s + 1}{\epsilon_s/K} I_d\right).
    \end{equation*}
    \EndFor
    \State Set $\beta^{(k)} := \Pi_{C_\beta} \left( \beta^{(k-1)} + \eta \sum_{s=1}^S v_s \left[\dot{\ell}(\beta^{(k-1)}; \{T_{s, i}, \Delta_{s, i}, Z_{s, i}\}_{i=1}^{n_s}) + W_s^{(k)} \right] \right)$.
\EndFor
\State \textbf{Output}: $\beta^{(K)}$.
\end{algorithmic}
\end{algorithm}

\section{Additional numerical experiments} \label{app:simulations}
This Appendix contains sensitivity analysis to the tuning parameters of \Cref{alg:cdp_sgd} (\Cref{subsec-sensitivity_analysis}), simulations under different censoring rates and discussion (\Cref{subsec-censoring-rates}), log-log plots to show the numerical dependence of the estimation error on sample size $n$ and privacy budget $\epsilon$ (\Cref{app:loglogplots}), and a more detailed description of the procedure for the real-data experiments (\Cref{app:real_data}).

\subsection{Sensitivity analysis}
\label{subsec-sensitivity_analysis}
\Cref{fig:cdp_sensitivityanalysis} contains a sensitivity analysis to two tuning parameters in \Cref{alg:cdp_sgd}: the magnitude of the noise used to preserve privacy and the step size used in the gradient descent. The sample size is fixed at $n=10000$ and all other setups are the same as \Cref{subsec-sim-cox} unless otherwise specified, 

The absolute constant in the bound for the sensitivity $C \log(n)/n$, detailed in \Cref{lemma:grad_sensitivity2}, is an increasing function of $C_Z$ from \Cref{assump-1}\ref{assp:covariate_bound}, which is a bound on the norm of the covariates. To see the impact of the noise scaling, we consider varying $C_Z \in \{0.8, 1, 1.2, 1.4, 1.6, 1.8, 2.0 \}.$  This is numerically shown in \Cref{fig:cdp_sensitivityanalysis}(A). Note that the true value of $C_Z$ is 1, so we do not have the required privacy guarantees for $C_Z < 1$. In \Cref{fig:cdp_sensitivityanalysis}(B), we vary the step size in \Cref{alg:cdp_sgd} as $\eta \in \{0.2, 0.3, \dots, 0.8\}$; in all of our other experiments, the step size was set to $\eta =0.5$.

\begin{figure}[htbp]
  \centering
\includegraphics[width=0.95\textwidth]{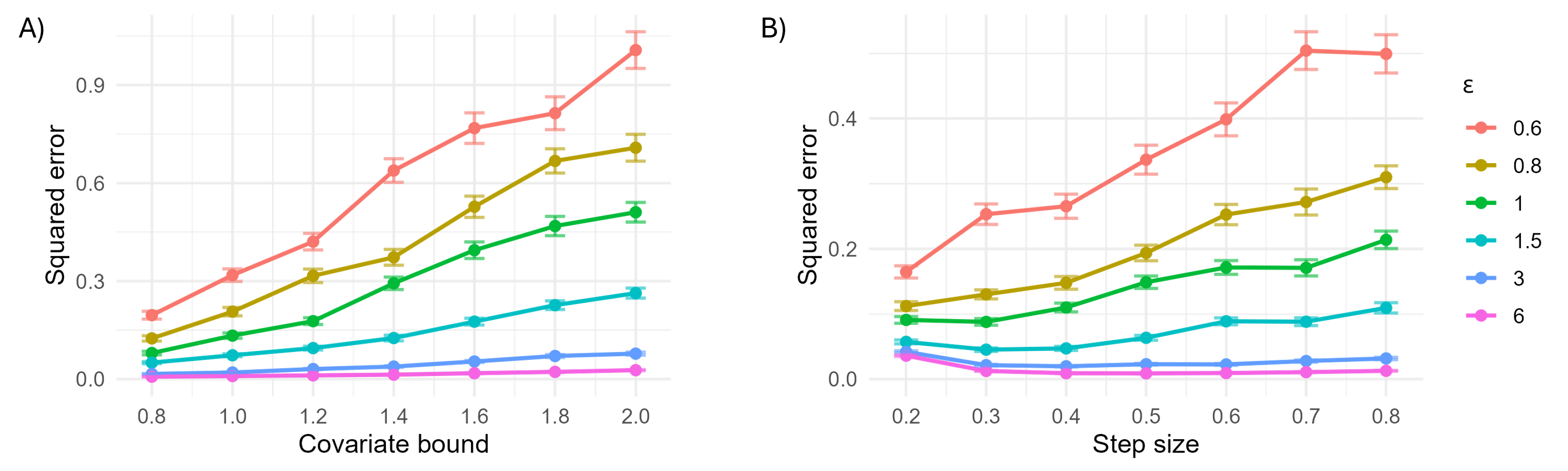}
  \caption{Sensitivity analysis of \Cref{alg:cdp_sgd} to privacy noise scaling (panel A) and gradient descent step size (panel B).}
  \label{fig:cdp_sensitivityanalysis}
\end{figure}

The increasing error seen in \Cref{fig:cdp_sensitivityanalysis}(A) is due to the increase in the variances of the Gaussian noise.  This is especially severe in the high privacy regime where the private rate dominates.  The dependence slows down in the low privacy regime where the non-private rate dominates.  \Cref{fig:cdp_sensitivityanalysis}(B) shows the relatively robust performances of our algorithm with respect to the step size, especially in the low privacy regimes.

\subsection{Censoring rates}
\label{subsec-censoring-rates}
We again fix the sample size $n=10000$ and keep all other setups the same as \Cref{subsec-sim-cox} unless otherwise specified. \Cref{fig:cdp_censoring} collects the results on investigations on the effect of censoring by varying the censoring distribution as Exp$(\alpha)$, with $\alpha \in \{0.1, 0.3, 0.5, 0.7, 0.9, 1.1, 1.3\}$. \Cref{tab:censoring_experiments} contains Monte Carlo estimates of $\mathbb{P}(\Delta =0 \mid T<1)$ and $\mathbb{P}(Y(1)=1)$ under these distributions. 
\begin{table}[htbp]
\centering
\begin{tabular}{crrrrrrr}
  \hline
 $\alpha$ & 0.1 & 0.3 & 0.5 & 0.7 & 0.9 & 1.1 & 1.3 \\
 \hline 
 $\mathbb{P}(\Delta=0 \mid T < 1)$ & 0.090 & 0.229 & 0.330 & 0.410 & 0.471 & 0.520 & 0.561\\
 $\mathbb{P}(Y(1) = 1)$ & 0.33 & 0.273 & 0.223 & 0.183 & 0.150 & 0.123 & 0.100 \\
   \hline
\end{tabular}
\caption{Monte Carlo estimates (from $10^6$ samples) of $\mathbb{P}(\Delta=0 \mid T< 1)$ and $\mathbb{P}(Y(1)=1)$ under different rates for the censoring distribution Exp$(\alpha)$.}
\label{tab:censoring_experiments}
\end{table}

\begin{figure}[htbp]
  \centering
\includegraphics[width=0.95\textwidth]{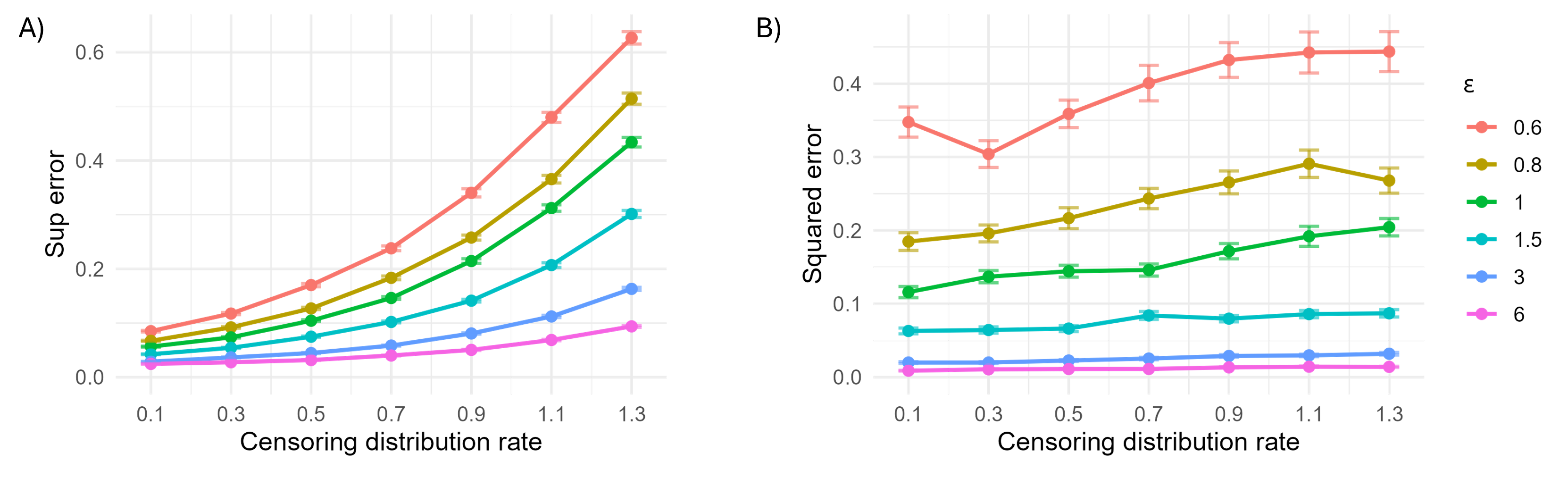}
  \caption{Effect of censoring on $\Lambda$ (panel A) and $\beta_0$ estimation (panel B), by varying the censoring distribution as $\mathrm{Exp}(\alpha)$, where $\alpha \in \{0.1, 0.3, \dots, 1.3\}$.}
  \label{fig:cdp_censoring}
\end{figure}

It can be seen from \Cref{fig:cdp_censoring} that as censoring rates grow, the estimation error grows.  
The non-private Cox regression and Breslow estimators are functions of the integral of the counting processes on the whole time interval $[0, 1]$. At each time-point, the censoring can be seen as reducing effective information and therefore increasing the estimation error. Since the censoring rate varies across time and interacts with the decreasing sizes of the at-risk sets, its effect on increasing estimation variance is difficult to quantify analytically. We provide some approximations in \Cref{tab:censoring_experiments}. As for the private rates in the upper bounds, in \Cref{alg:FDP-Breslow}, the magnitude of noise added to preserve privacy has a factor of $\{\mathbb{P}(Y(1)=1)\}^{-2}$. On the other hand, there is no such dependence on at-risk probabilities or censoring rates in the Gaussian mechanism used in \Cref{alg:FDP-SGD}.

Comparing the two panels in \Cref{fig:cdp_censoring}, we observe that the estimation error for $\beta_0$ increases with the censoring rate at a similar pace across different privacy budgets, whereas for $\Lambda_0$, this effect is more pronounced in the high-privacy regimes. We conjecture that this is due to differences between whether the non-private or private rate dominates in the upper bound after considering constant and poly-logarithmic factors.  

\subsection{Numerical demonstration of theoretical rates} \label{app:loglogplots}
To further demonstrate the phase transition phenomenon numerically, we depict  log(error) against $\log(n)$ in \Cref{fig:CDP_log_samples} and log(error) against $\log(\epsilon)$ in \Cref{fig:CDP_log_epsilon}. 

From \Cref{prop:fdp_upper}, with fixed dimension, we expect that plugging in $\hat{\beta}$ into the Breslow estimator gives 
\[
    \log \left\{\sup_{t \in [0, 1]} |\widehat{\Lambda}_\mathrm{NA}(t) - \Lambda(t)| \right\} \approx - \log(n) - \log(\epsilon) + \mathrm{intercept}
\]
in the private regime, and 
\[
    \log \left\{\sup_{t \in [0, 1]}|\widehat{\Lambda}_\mathrm{NA}(t) - \Lambda(t)| \right\}\approx -\frac{1}{2}\log(n) + \mathrm{intercept}
\]
in the non-private regime.

Ignoring iterated logarithmic factors, from the theoretical results in \Cref{prop:cdp_sgd} for $n$ large enough, we expect the output of \Cref{alg:cdp_sgd} to scale as 
\[
    \log(\|\widehat{\beta} - \beta_0\|^2) \approx -2 \log(n) - 2 \log(\epsilon) + \mathrm{intercept}
\]
in the private regime, and 
\[
    \log(\|\widehat{\beta} - \beta_0\|^2) \approx -\log(n)+ \mathrm{intercept}
\]
in the non-private regime. 

\begin{figure}[htbp]
  \centering
\includegraphics[width=0.95\textwidth]{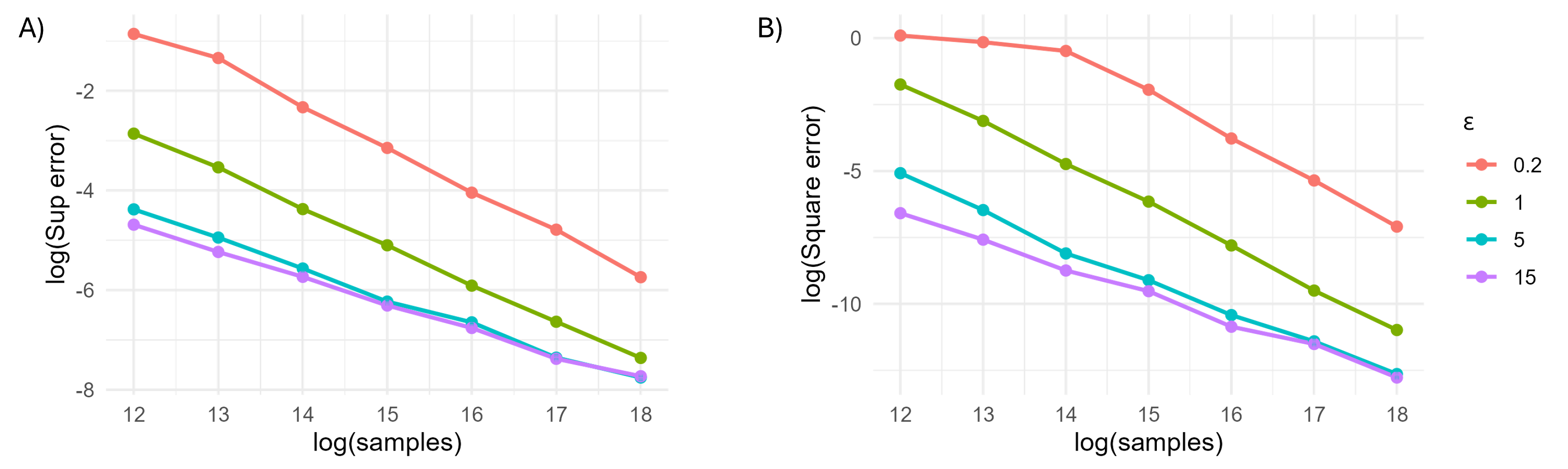}
  \caption{Plots of $\log_2(\|\widehat{\Lambda}_\mathrm{NA} - \Lambda\|^2)$ (panel A) and $\log_2(\|\widehat{\beta} - \beta_0\|^2)$ (panel B) against $\log_2(n)$, for $n \in \{2^{12}, 2^{13}, \dots, 2^{18}\}$ and $\epsilon \in \{0.2, 1, 5, 15\}.$}   
  \label{fig:CDP_log_samples}
\end{figure}

\begin{figure}[htbp]
  \centering
\includegraphics[width=0.95\textwidth]{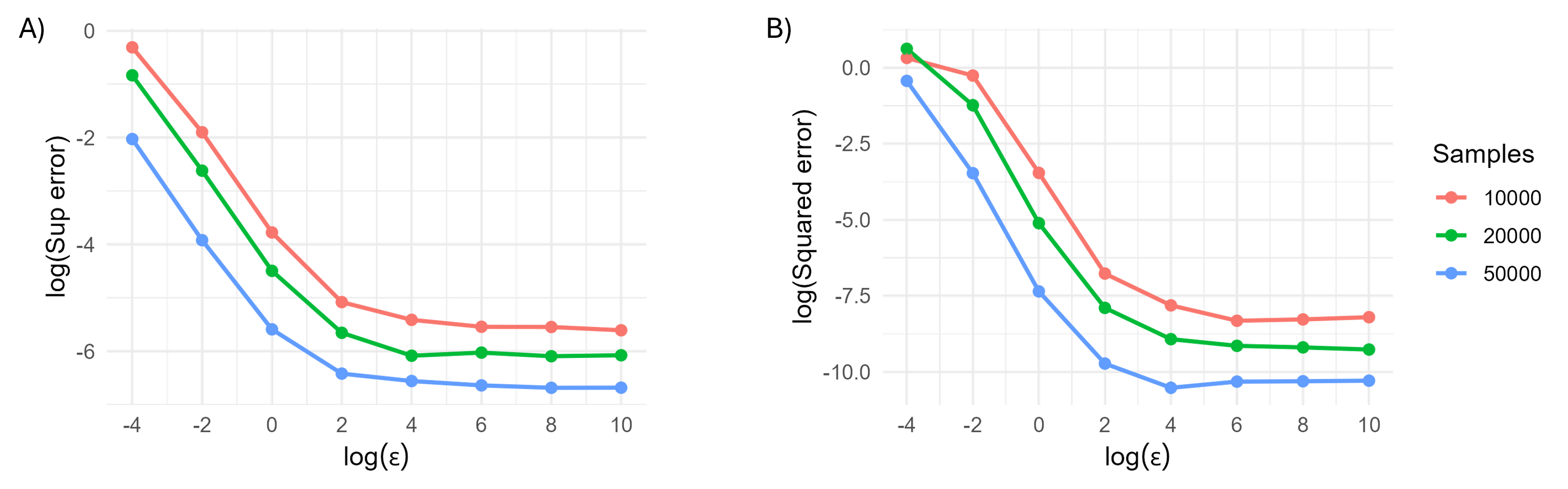}
    \caption{Plots of $\log_2(\|\widehat{\Lambda}_\mathrm{NA} - \Lambda\|^2)$ (panel A) and $\log_2(\|\widehat{\beta} - \beta_0\|^2)$ (panel B) against $\log_2(\epsilon)$, for $\epsilon \in \{2^{-4}, 2^{-3}, \dots, 2^{10}\}$ and $n \in \{10000, 20000, 50000\}$.} 
  \label{fig:CDP_log_epsilon}
\end{figure}

\subsection{Improvements from \Cref{alg:FDP-SGD-interactive}}
\label{app:interactive}
We consider experiments with the same setup as for $\beta_0$ estimation in \Cref{sec-simu-fdp}, but instead, we set $n_s=n=5000$ and obtain our $\widehat{\beta}$ estimate using \Cref{alg:FDP-SGD-interactive}. The results are shown in \Cref{fig:fdp_interactive}, where we see improved estimation results despite a 75\% reduction in the number of samples. 

\begin{figure}[htbp]
  \centering
\includegraphics[width=0.6\textwidth]{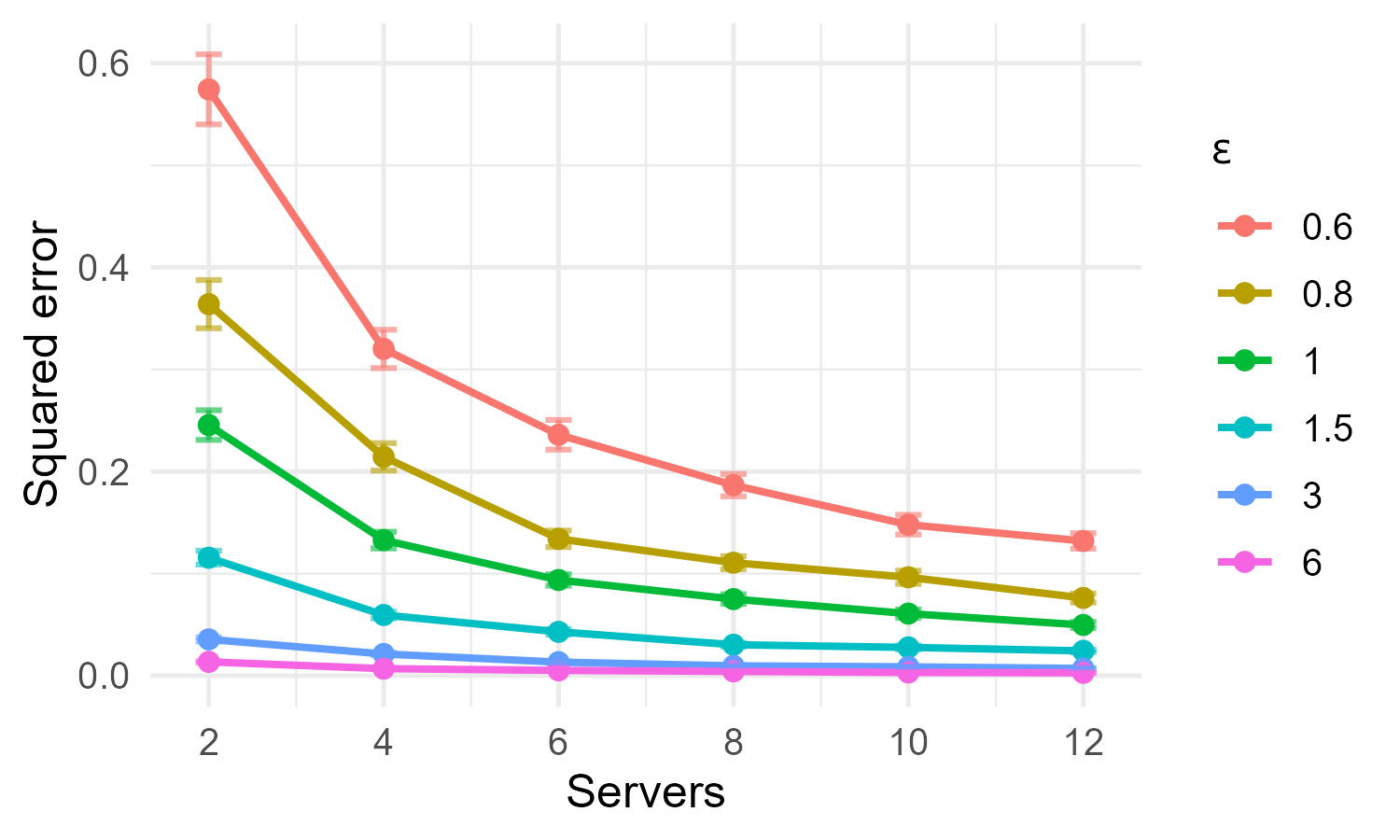}
    \caption{Numerical experiments using \Cref{alg:FDP-SGD-interactive}, with 5000 samples per server.}   \label{fig:fdp_interactive}
\end{figure}

\subsection{Additional details for the real data experiments}
\label{app:real_data}
For all server settings in our real data experiments, we set the tuning parameters of the interactive FDP-Cox algorithm (\Cref{alg:FDP-SGD-interactive}) as $C_Z=\sqrt{5}$, $C_\beta = 3$, step-size as 0.5 and the number of iterations as 300. In \Cref{alg:FDP-NA}, we set $C_\lambda$ to be three times the maximum ratio $\tilde{\Lambda}(t)/t$ observed in a non-private Nelson--Aalen estimate $\tilde{\Lambda}$ at the event times, and in \Cref{alg:FDP-SGD} we set $C_\beta=3$ because the coefficients $\beta^*$ from fitting a Cox regression model to the 5 categorical covariates has $\|\beta^*\|_2=2.63$; in practice these tuning parameters should be selected in a data-driven way under DP constraints. We fix $\delta = 0.001$ and compare the effects of both the privacy budget $\epsilon \in \{3, 6, 9, 12, 15\}$ and the number of servers on the estimation accuracy. For each configuration, we run 200 repetitions of randomly partitioning the data into the servers, and report the mean and standard error. The central DP experiments follow similar steps but randomises the order of the data rather than performing the splitting. 

\section{Further discussions of the Conclusions section} \label{sec-additional-discussions}

In this section, we provide additional discussions on the extensions mentioned in the conclusion.

\medskip
\noindent \textbf{Heterogeneous distributions.}  Our current theory is developed under the homogeneous distribution assumption, which allows us to isolate the statistical effect of privacy from the additional complications introduced by cross-server heterogeneity. Extending the framework to heterogeneous server-specific distributions is an important direction for future work.

For the Cox regression model, there are two components of the distribution - the cumulative baseline hazard function $\Lambda_0$ and the regression coefficients $\beta_0$.  If different servers possess the same~$\beta_0$ but different $\Lambda_0$, and if the goal is to estimate $\beta_0$, then the same procedure will lead to the same rates, up to constants, under the same regularity conditions on $\Lambda_0$.

If $\beta_0$ are different, then at a high level, we expect heterogeneity to introduce an additional bias term reflecting the discrepancy across the distributions.  Consider a transfer learning framework, where there is a target dataset along with multiple auxiliary datasets.  The goal is to estimate the regression coefficients $\beta_0$ in the target distribution, while the auxiliary datasets have regression coefficients $\beta_1, \ldots, \beta_S$, satisfying $\beta_0, \beta_1, \ldots, \beta_S \in B_{C_\beta}(0)$, $\max_{s \in [S]} \|\beta_s - \beta_0\|_2 \leq h$, and $\lambda_{\min}(G_s(\beta_s)) \asymp \lambda_{\max}(G_s(\beta_s)) \asymp 1/d$.

It follows from transfer learning literature, especially transfer learning under federated DP literature, including \cite{li2024federated} and \cite{auddy2024minimaxadaptivetransferlearning}, that the expected minimax rate would therefore be
\begin{equation}\label{eq-transfer-learning-conjecture}
    \inf_Q \inf_{\widehat{\beta}} \sup_P \mathbb{E}\left\{\big\|\widehat{\beta} - \beta_0\big\|_2^2\right\} \asymp_{\log} \min\left\{\frac{d^2}{n} + \frac{d^3}{n^2\epsilon^2},  \frac{d^2}{(S+1)n} + \frac{d^3}{(S+1)n^2\epsilon^2} + h^2\right\}.
\end{equation}
The rate in \eqref{eq-transfer-learning-conjecture} can be understood as the minimum of the target only rate and the transfer learning rate. The transfer learning rate further consists of the non-private rate, federated DP rate and the bias term.

The lower bound proof follows from our proof for the homogeneous distribution problem, along with the standard lower bound techniques for transfer learning used in \cite{li2024federated} and \cite{auddy2024minimaxadaptivetransferlearning}.  As for the upper bound, the additional technicality mainly roots in a more careful analysis on the smoothness and convexity of the partial log-likelihood function when deploying gradient descent algorithms.

\medskip
\noindent \textbf{Private inference.}  Uncertainty quantification under privacy constraints is a natural next step.  There has been a line of research on this topic in the literature, but we are not aware of any methods can be easily adapted to our setting.  

Firstly, the bootstrap asymptotic results in \cite{WangChengAwan2022DPBootstrap} rely on assumptions that do not directly match the Cox regression setting considered here, such as finite-support conditions and a parameter of interest that can be represented as a functional of a finite probability distribution.  For this reason, their framework does not appear to transfer directly to our problem.

More broadly, there are promising alternative directions, including covariance-based private inference methods \citep[e.g.][]{avella2023noisyopt} and subsampling or resampling approaches \citep[e.g.][]{chadha2024resampling, WangChengAwan2022DPBootstrap} developed under central DP.  Adapting these ideas to federated DP Cox regression, however, would likely require new techniques, particularly due to the use of cross-server information and the more complicated form of the Hessian matrix.

We therefore view private uncertainty quantification for federated Cox regression as an important and interesting direction for future work.

\section{Technical details on coefficient estimation}
\label{app-beta}
\Cref{app:beta_sensitivity} provides the bound on the sensitivity of the gradient, which is used within the gradient descent algorithms for $\beta_0$ estimation (Algorithms \ref{alg:FDP-SGD}, \ref{alg:cdp_sgd}, and \ref{alg:FDP-SGD-interactive}). The convergence guarantees of these algorithms can be found in \Cref{app:Cox_upper_proof}. The proofs of the lower bounds are in \Cref{app:Cox_lower_proof} -- these rely on bounding the traces of the Fisher information of the private transcripts in the Cox regression setting; the simpler central-DP setting can be found in \Cref{lemma:cdp_trace} and its extension to federated DP in \Cref{lemma:FDP_trace}.  \Cref{app:beta_auxiliary} contains the auxiliary results needed for these proofs (e.g.~convergence of $\dot{\ell}(\beta_0)$ and $\ddot{\ell}(\beta_0)$), which are not specific to DP and may be a useful compilation for other aspects of Cox regression. 

\subsection{Bounding the sensitivity of the gradient} \label{app:beta_sensitivity}

\begin{lemma} \label{lemma:grad_sensitivity2}
Let $\beta \in \mathbb{R}^d$ and suppose that the data $\{(T_i, \Delta_i, \{Z_i(t): \, t \in [0, 1]\})\}_{i \in [n]}$ satisfy \Cref{assump-1}\ref{assp:covariate_bound}. We can bound 
\begin{equation*}
   \mathrm{sens}_2(\dot{\ell}(\beta)) \leq \frac{4C_Z}{n} + \frac{2 \exp(2C_Z \|\beta\|_2)C_Z (1+\log(n))}{n}.
\end{equation*}
where $\dot{\ell}$ is the score defined in \eqref{eq:cox_gradient} and the sensitivity $\mathrm{sens}_2(\cdot)$ defined in \Cref{def-sensitivity}.

Further considering other forms of the sensitivity, the following hold for two neighbouring datasets $D = \{(T_i, \Delta_i, \{Z_i(t): \, t \in [0, 1]\})\}_{i \in [n]}$ and $D' = \{(T_i', \Delta_i', \{Z'_i(t): \, t \in [0, 1]\})\}_{i \in [n]}$. 
\begin{itemize}[leftmargin=2cm]
    \item[Case 1:] When only the censoring indicators $\Delta_i$'s are privatised, $\mathrm{sens}_2(\dot{\ell}(\beta)) \asymp 1/n$, where
    \[
        (T_i, Z_i(\cdot)) = (T_i', Z_i'(\cdot)), \quad i \in [n], \quad \mbox{and} \quad  \sum_{i=1}^n \mathbbm{1}\{\Delta_i \neq \Delta_i'\} \leq 1.
    \]
    \item [Case 2:] When only the covariate process $Z_i$'s are privatised, $\mathrm{sens}_2(\dot{\ell}(\beta)) \asymp \log(n)/n$, where
    \[
        (T_i, \Delta_i) = (T_i', \Delta_i') \quad i \in [n], \quad \mbox{and} \quad \sum_{i=1}^n \mathbbm{1}\{Z_i(\cdot) \neq Z_i'(\cdot) \} \leq 1.
    \]
    \item [Case 3:] When only the observed times $T_i$'s are privatised, $\mathrm{sens}_2(\dot{\ell}(\beta)) \asymp \log(n)/n$, where
    \[
        (\Delta_i, Z_i(\cdot)) = (\Delta_i', Z_i'(\cdot)), \quad i \in [n], \quad \mbox{and} \quad  \sum_{i=1}^n \mathbbm{1}\{T_i \neq T_i'\} \leq 1.
    \]    

    \item [Case 4:] When the triples $(T_i, \Delta_i, Z_i(\cdot))$'s are privatised, $\mathrm{sens}_2(\dot{\ell}(\beta)) \asymp \log(n)/n$, where
    \[
        \sum_{i=1}^n \mathbbm{1}\{(T_i, \Delta_i, Z_i(\cdot)) \neq (T_i', \Delta_i', Z_i'(\cdot)) \} \leq 1.
    \]
\end{itemize}
\end{lemma}

\begin{proof}
From \Cref{def-sensitivity} of the sensitivity and the definition of the score function $\dot{\ell}(\cdot)$ in \eqref{eq:cox_gradient}, we have that $\mathrm{sens}_2(\dot{\ell}) = \sup_{D \sim D'} \|\dot{\ell}(\beta; D) - \dot{\ell}(\beta; D')\|_2$.  Without loss of generality, we consider $D = \{(T_i, \Delta_i, Z_i(\cdot))\}_{i \in [n]}$ and $D' = (T_1', \Delta'_1, Z'_1(\cdot)) \cup \{(T_i, \Delta_i, Z_i(\cdot))\}_{i \in [n] \setminus \{1\}}$, and write $\bar{Z}_D(s, \beta)$ and $\bar{Z}_{D'}(s, \beta)$ be the weighted means in \eqref{eq:weighted-mean-def} computed on $D$ and $D'$ respectively.  We adopt the counting process notation for data in~$D$ and $D'$ without further definitions.  The rest of the proof proceeds by upper and lower bounding the sensitivity respectively.

\medskip
\noindent \textbf{Upper bounds.}  Due to the triangle inequality, we have 
\begin{align}
\|\dot{\ell}(\beta; D) - \dot{\ell}(\beta; D') \|_2 &= \|(I) - (II) - (III)\|_2 \nonumber \\
    &\leq \|(I)\|_2 + \|(II)\|_2 + \|(III)\|_2, 
    \label{eq-sens-lemma-proof-upper-decomp} 
\end{align}
    where
    \begin{align}
    (I) & := \frac{1}{n} \int_0^1 Z_1(s) \dint N_1(s) - \frac{1}{n} \int_0^1 Z_1'(s) \dint N_1'(s), \label{eq-sens-lemma-proof-upper-decomp-1}\\
    (II)&:=\frac1n\sum_{k=2}^n\int_0^1
        \{\bar Z_D(s,\beta)-\bar Z_{D'}(s,\beta)\}
        \dint N_k(s)
        \label{eq-sens-lemma-proof-upper-decomp-2}
\end{align}
and
\begin{align}
    (III) & := \frac{1}{n} \int_0^1 \frac{Y_1(s) Z_1(s) e^{\beta^\top Z_1(s)} + \sum_{j = 2}^n Y_j(s) Z_j(s) e^{\beta^\top Z_j(s)}}{Y_1(s) e^{Z_1(s)^\top \beta} + \sum_{j = 2}^n Y_j(s) e^{Z_j(s)^\top \beta}} \dint N_1(s) \nonumber \\
    & \hspace{3cm}- \frac{1}{n} \int_0^1 \frac{Y_1'(s) Z_1'(s) e^{\beta^\top Z_1'(s)} + \sum_{j = 2}^n Y_j(s) Z_j(s) e^{\beta^\top Z_j(s)}}{Y_1'(s) e^{\beta^\top Z'_1(s)} + \sum_{j =2}^n Y_j(s) e^{Z_j(s)^\top \beta}} \dint N_1'(s).  \label{eq-sens-lemma-proof-upper-decomp-4}
\end{align}

We now discuss these three terms separately.  For the terms $(I)$ and $(III)$, it directly follows from \Cref{assump-1}\ref{assp:covariate_bound} that 
\begin{equation} \label{eq-sens-lemma-proof-upper-decomp-1-1}
    \max\{\|(I)\|_2, \, \|(III)\|_2\} \leq 2C_Z/n.
\end{equation}

As for term $(II)$, fix a failure time $s$ of a subject that is common to both datasets, and assume that $Y_1(s)\exp(\beta^\top Z_1(s)) \geq Y_1'(s)\exp(\beta^\top Z_1'(s))$. Then 
\begin{align*}
       \|\bar Z_D(s,\beta)-\bar Z_{D'}(s,\beta)\|_2
    &=\Bigg\|
        \frac{
            \{Y_1(s)e^{\beta^\top Z_1(s)}
            -Y_1'(s)e^{\beta^\top Z_1'(s)}\}
            \{Z_1(s)-\bar Z_D(s,\beta)\}
        }{
            Y_1'(s)e^{\beta^\top Z_1'(s)}
            +\sum_{j=2}^nY_j(s)e^{\beta^\top Z_j(s)}
        }\\
    &\qquad\qquad+
        \frac{
            Y_1'(s)e^{\beta^\top Z_1'(s)}
            \{Z_1(s)-Z_1'(s)\}
        }{
            Y_1'(s)e^{\beta^\top Z_1'(s)}
            +\sum_{j=2}^nY_j(s)e^{\beta^\top Z_j(s)}
        }
    \Bigg\|_2\\
    &\leq
        \frac{2C_ZY_1(s)e^{\beta^\top Z_1(s)}}
        {Y_1'(s)e^{\beta^\top Z_1'(s)}
            +\sum_{j=2}^nY_j(s)e^{\beta^\top Z_j(s)}}\\
    &\leq
        \frac{2C_Ze^{2C_Z\|\beta\|_2}}
        {\sum_{j=2}^nY_j(s)},
\end{align*}
where the first inequality uses
$\|\bar Z_D(s,\beta)\|_2\leq C_Z$.
Interchanging the datasets gives the same bound when
$Y_1(s)e^{\beta^\top Z_1(s)}
<Y_1'(s)e^{\beta^\top Z_1'(s)}$.
It follows that
\begin{align}
    \|(II)\|_2
    &\leq\frac1n\sum_{k=2}^n\int_0^1
        \|\bar Z_D(s,\beta)-\bar Z_{D'}(s,\beta)\|_2
        \dint N_k(s)\notag\\
    &\leq\frac{2C_Ze^{2C_Z\|\beta\|_2}}n
        \sum_{k=2}^n\int_0^1
        \frac{\dint N_k(s)}{\sum_{j=2}^nY_j(s)}\notag\\
    &\leq\frac{2C_Ze^{2C_Z\|\beta\|_2}(1+\log n)}n,
    \label{eq-sens-lemma-proof-ii}
\end{align}
The final inequality follows by noting that, for $k\in\{2,\dots,n\}$
such that $\Delta_k=0$, i.e.~a censored subject, it holds that
\[
    \int_0^1\frac{1}{\sum_{j=2}^nY_j(s)}\dint N_k(s)=0.
\]
At a failure time with $r$ unchanged subjects at risk and $q$ failures
among them, the total contribution is $q/r$, and
\[
    \frac qr\leq\sum_{m=r-q+1}^r\frac1m.
\]
After these failures, the number of unchanged subjects at risk
decreases by at least $q$. We may therefore bound 
\begin{equation*}
    \sum_{k=2}^n\int_0^1
        \frac{1}{\sum_{j=2}^nY_j(s)}\dint N_k(s)
    \leq\sum_{m=1}^{n-1}\frac{1}{m}
    \leq1+\int_1^n\frac1x\dint x
    =1+\log(n).
\end{equation*}

Combining \eqref{eq-sens-lemma-proof-upper-decomp}, \eqref{eq-sens-lemma-proof-upper-decomp-1}, \eqref{eq-sens-lemma-proof-upper-decomp-2}, \eqref{eq-sens-lemma-proof-upper-decomp-4}, and their upper bounds from \eqref{eq-sens-lemma-proof-upper-decomp-1-1} and \eqref{eq-sens-lemma-proof-ii}, we are now ready to discuss upper bounds for different cases.

In Case 1, since only $\Delta$ is privatised, $(II) = 0$ and 
    \[
        \mathrm{sens}_2(\dot{\ell}(\beta)) \leq \|(I)\|_2 + \|(III)\|_2 \leq 4C_Z/n.
    \]
In Cases 2, 3 and 4, we have the upper bounds 
    \[
        \mathrm{sens}_2(\dot{\ell}(\beta)) \leq \frac{4C_Z}{n} + \frac{2 \exp(2C_Z \|\beta\|_2)C_Z (1+\log(n))}{n}.
    \]

\medskip
\noindent \textbf{Lower bounds.}  To show that the sensitivities are attainable up to constants, we construct specific datasets $D$ and $D'$ for each of the four cases, with covariates that are constant in time.  Take $n\geq2$, $\beta \in B_{C_\beta}(0)$, and $e_1=(1,0,\ldots,0)^\top$. Let $r:=\exp(2C_Z\beta^\top e_1)$, so $\exp(-2C_ZC_\beta)\leq r\leq \exp(2C_ZC_\beta)$.

For Case 1, set $T_i=i/(n+1)$, $Z_1=C_Ze_1$, and $Z_i=-C_Ze_1$ for $i\geq2$, and change only $\Delta_1=1$ to $\Delta_1'=0$. Since the risk sets and all the other failure indicators are unchanged,
\begin{align*}
\|\dot\ell_D(\beta)-\dot\ell_{D'}(\beta)\|_2
&=\left\|\frac{Z_1}{n}-\frac{\sum_{j=1}^n Y_j(T_1) \exp(\beta^\top Z_j)Z_j}{n\sum_{j=1}^n Y_j(T_1) \exp(\beta^\top Z_j)}\right\|_2\\
&=\frac{C_Z}{n}\left|1-\frac{r-(n-1)}{r+n-1}\right|\\
&=\frac{2C_Z(n-1)}{n(r+n-1)}\geq\frac{C_Z}{n(1+ \exp(2C_Z C_\beta))}.
\end{align*}

For Case 2, set $T_i=i/(n+1)$, $\Delta_i=1$, and $Z_i=-C_Ze_1$ for $i\in[n-1]$, while $T_n=1$, $\Delta_n=0$, and $Z_n=C_Ze_1$. Change only $Z_n$ to $Z_n'=-C_Ze_1$. At $T_i$, the risk set consists of subjects $i,\ldots,n$. All covariates of the observations in $D'$ are identical, so $\dot\ell_{D'}(\beta)=0$, and we have that 
\begin{align*}
\|\dot\ell_D(\beta)-\dot\ell_{D'}(\beta)\|_2
&=\left\|\frac1n\sum_{i=1}^{n-1}\left\{-C_Ze_1-\frac{-(n-i)C_Ze_1+ rC_Ze_1}{n-i+r}\right\}\right\|_2\\
&=\frac{2C_Z}{n}\sum_{i=1}^{n-1}\frac{r}{n-i+r} =\frac{2C_Z}{n}\sum_{j=2}^n\frac{1}{j-1+r}\\
&\geq\frac{C_Z \exp(-2C_Z C_\beta)\log n}{n},
\end{align*}
where the inequality is from $r/(j-1+r)\geq\min\{1,r\}/j\geq e^{-2C_ZC_\beta}/j$.

For Cases 3 and 4, use the same dataset $D$, but change only $T_n=1$ to $T_n'=T_1/2$, keeping $\Delta_n'=0$. Subject $n$ is then censored before every failure in $D'$, so each failure risk set has identical covariates and $\dot\ell_{D'}(\beta)=0$. Hence
\begin{align*}
\|\dot\ell_D(\beta)-\dot\ell_{D'}(\beta)\|_2
&=\left\|\frac1n\sum_{i=1}^{n-1}\left\{\frac{-(n-i)C_Ze_1}{n-i}-\frac{-(n-i)C_Ze_1+rC_Ze_1}{n-i+r}\right\}\right\|_2\\
&=\frac{2C_Z}{n}\sum_{i=1}^{n-1}\frac{r}{n-i+r}\geq\frac{C_Z \exp(-2C_ZC_\beta)\log n}{n}.
\end{align*}
\end{proof}

\subsection{Proofs of upper bounds}
\label{app:Cox_upper_proof}

We first give the proof of \Cref{prop:cdp_sgd} is on a central DP gradient descent method for estimating the Cox regression coefficients, which uses R\'enyi DP for privacy accounting. We then provide the proof of \Cref{prop:fdp_upper}. 

\begin{proof}[Proof of \Cref{prop:cdp_sgd}]
\ \\
\noindent \textbf{Privacy guarantees.} Plugging the sensitivity level obtained from \Cref{lemma:grad_sensitivity2} into the privacy guarantees in \Cref{lemma:Gaussian_composition}, it holds that \Cref{alg:cdp_sgd} satisfies $(\epsilon, \delta)$-CDP.

\medskip
\noindent \textbf{Estimation error.}
   Define the event
\begin{equation*}
    \mathcal E:=\left\{
        \|\ddot\ell(\beta_0)+G(\beta_0)\|_2
        \leq\frac{\rho_-}{2d}
    \right\}.
\end{equation*}
By \Cref{lemma:hessian_convergence}, under the stated sample size condition,
\begin{equation}
    \mathbb P(\mathcal E^c)\lesssim\frac1n.
    \label{eq:cdp-failure}
\end{equation}

On $\mathcal E$, \Cref{lemma:3-2_huang,lemma:hessian_upper} imply, uniformly on $B_{C_\beta}(0)$, that
\begin{equation*}
    \mu I_d\preceq-\ddot\ell(\beta)\preceq LI_d,
    \qquad \mu:=\frac{e^{-8C_ZC_\beta}\rho_-}{2d},
    \qquad L:=4C_Z^2.
\end{equation*}

Choose $\eta:=1/(2L)$. Due to the projection, on $\mathcal E$ we have
\begin{equation*}
    \|\beta^{(k)}-\beta_0\|_2^2
    \leq\left\|
        \beta^{(k-1)}-\beta_0
        +\eta\{\dot\ell(\beta^{(k-1)})-\dot\ell(\beta_0)\}
        +\eta\dot\ell(\beta_0)+\eta W^{(k)}
    \right\|_2^2.
\end{equation*}
Since $\mathbb E[W^{(k)}\mid D,\beta^{(k-1)}]=0$, the noise cross term vanishes when taking conditional expectation given $\mathcal E$. This implies that
\begin{align*}
    \mathbb E[\|\beta^{(k)}-\beta_0\|_2^2\mid\mathcal E]
    &\leq(1-\eta\mu)^2
        \mathbb E[\|\beta^{(k-1)}-\beta_0\|_2^2\mid\mathcal E]\\
    &\quad+2\eta(1-\eta\mu)
        \mathbb E\left[
            \|\beta^{(k-1)}-\beta_0\|_2\|\dot\ell(\beta_0)\|_2
            \mid\mathcal E
        \right]\\
    &\quad+\eta^2\mathbb E[\|\dot\ell(\beta_0)\|_2^2\mid\mathcal E]
        +\eta^2\mathbb E[\|W^{(k)}\|_2^2\mid\mathcal E]\\
    &\leq\gamma
        \mathbb E[\|\beta^{(k-1)}-\beta_0\|_2^2\mid\mathcal E]\\
    &\quad+\frac{2\eta^2}{1-(1-\eta\mu)^2}
        \mathbb E[\|\dot\ell(\beta_0)\|_2^2\mid\mathcal E]
        +\eta^2\mathbb E[\|W^{(k)}\|_2^2\mid\mathcal E],
\end{align*}
where the second inequality follows by Young's inequality with
$\gamma:=\{1+(1-\eta\mu)^2\}/2<1$.

By \Cref{lemma:true_grad_convergence}, since $\mathbb P(\mathcal E)\geq1/2$ for sufficiently large $n$, we have
\begin{equation*}
    \mathbb E[\|\dot\ell(\beta_0)\|_2^2\mid\mathcal E]
    \leq\frac{\mathbb E\|\dot\ell(\beta_0)\|_2^2}{\mathbb P(\mathcal E)}
    \lesssim\frac1n.
\end{equation*}
By the noise calibration in \Cref{alg:cdp_sgd} and the boundedness of the iterates,
\begin{equation*}
    \mathbb E[\|W^{(k)}\|_2^2\mid\mathcal E]
    \lesssim
    \frac{Kd(1+\log n)^2\{\log(1/\delta)+\epsilon\}}
        {n^2\epsilon^2}.
\end{equation*}
Using $\eta=1/(2L)$ gives the recursion
\begin{align*}
    \mathbb E[\|\beta^{(k)}-\beta_0\|_2^2\mid\mathcal E]
    &\leq\gamma
        \mathbb E[\|\beta^{(k-1)}-\beta_0\|_2^2\mid\mathcal E]
        +\frac{C}{L\mu n}\\
    &\quad+\frac{CKd(1+\log n)^2\{\log(1/\delta)+\epsilon\}}
        {L^2n^2\epsilon^2}
\end{align*}
for some absolute constant $C>0$. Iterating with $\beta^{(0)}=0$ gives
\begin{align*}
    \mathbb E[\|\beta^{(K)}-\beta_0\|_2^2\mid\mathcal E]
    &\leq\gamma^KC_\beta^2\\
    &\quad+\frac{C}{1-\gamma}
        \left\{\frac1{L\mu n}
        +\frac{Kd(1+\log n)^2\{\log(1/\delta)+\epsilon\}}
            {L^2n^2\epsilon^2}\right\}\\
    &\lesssim\gamma^K+\frac1{\mu^2n}
        +\frac{Kd(1+\log n)^2\{\log(1/\delta)+\epsilon\}}
            {L\mu n^2\epsilon^2},
\end{align*}
where we used $1-\gamma\asymp\mu/L$.

On $\mathcal E^c$, the projection still allows us to bound
$\|\beta^{(K)}-\beta_0\|_2^2\leq4C_\beta^2$. Hence,
\begin{equation}
    \mathbb E[\|\beta^{(K)}-\beta_0\|_2^2\mid\mathcal E^c]
        \mathbb P(\mathcal E^c)
    \leq4C_\beta^2\mathbb P(\mathcal E^c)
    \lesssim\frac1n.
    \label{pf:cdpsgd_lte}
\end{equation}
Combining the preceding bounds gives
\begin{align*}
    \mathbb E\|\beta^{(K)}-\beta_0\|_2^2
    &\lesssim\gamma^K+\frac1{\mu^2n}+\frac1n
        +\frac{Kd(1+\log n)^2\{\log(1/\delta)+\epsilon\}}
            {L\mu n^2\epsilon^2}.
\end{align*}

Choosing
\[
    K:=\max\left\{1,
        \left\lceil
            \frac{\log\{\max(e,n/d^2)\}}{-\log\gamma}
        \right\rceil
    \right\},
\]
we have, using $\mu\asymp1/d$ and
$K\asymp(L/\mu)\log\{\max(e,n/d^2)\}$,
\begin{align*}
    \mathbb E\|\beta^{(K)}-\beta_0\|_2^2
    &\lesssim\frac{d^2}{n}
        +\frac{d^3(1+\log n)^3\{\log(1/\delta)+\epsilon\}}
            {n^2\epsilon^2}\\
    &\lesssim\frac{d^2}{n}
        +\frac{d^3(1+\log n)^3\log(1/\delta)}
            {n^2\epsilon^2}.
\end{align*}
For $\epsilon\leq1$, the additional term $d^3(1+\log n)^3/(n^2\epsilon)$ is absorbed by the privacy term, since $\log(1/\delta)\geq1$. For $\epsilon>1$, it is absorbed by $d^2/n$ under $n\gtrsim d(1+\log n)^3$. This proves the final inequality.
\end{proof}

\begin{proof}[Proof of \Cref{prop:fdp_upper}]
\ \\
\noindent \textbf{Privacy guarantees.} Plugging the sensitivity level obtained from \Cref{lemma:grad_sensitivity2} into the privacy guarantees in \Cref{lemma:gaussian_mechanism}, it holds that \Cref{alg:FDP-SGD} satisfies $(\{\epsilon_s, \delta_s\}_{s \in [S]}, K)$-FDP.

\medskip
\noindent \textbf{Estimation error.}
For $k \in [K]$, define 
\begin{equation*}
    \ell^{(k)}(\beta ):= \sum_{s=1}^S v_s\ell(\beta;D_s^{(k)}),
\end{equation*}
and the event 
\begin{equation*}
    \mathcal{E}_k:=\left\{\left\|\ddot\ell^{(k)}(\beta_0)+G(\beta_0)\right\|\leq\frac{\rho_-}{2d}\right\}.
\end{equation*}
By \Cref{lemma:distributed_hessian}, under the sample size conditions in \eqref{eq:sample_size}, we have that 
\begin{equation}
\mathbb{P}(\mathcal{E}_k^c) \lesssim\frac{d^2}{\sum_{s=1}^S \min\{b_s,b_s^2\epsilon_s^2/d\}}.
\label{eq:fdp-round-failure}
\end{equation}
Since the current batches are independent of the previous rounds' data and privacy-preserving noise, this bound also holds conditional
on $\beta^{(k-1)}$.

On $\mathcal E_k$, Lemma~\ref{lemma:3-2_huang} implies, uniformly on $B_{C_\beta}(0)$, that
\begin{equation*}
\mu I_d\preceq-\ddot\ell^{(k)}(\beta)\preceq LI_d,
\qquad \mu:=\frac{e^{-8C_ZC_\beta}\rho_-}{2d},
\qquad L:=\frac{3e^{8C_ZC_\beta}\rho_+}{2d}.
\end{equation*}

Choose $\eta:=1/(2L)$, then on $\mathcal{E}_k$
\begin{equation*}
  \|\beta^{(k)}-\beta_0\|_2\leq(1-\eta\mu)\|\beta^{(k-1)}-\beta_0\|_2 +\eta\left\| \dot\ell^{(k)}(\beta_0)+\sum_{s=1}^S v_sW_s^{(k)}
\right\|_2.
\end{equation*}
This implies that 
\begin{align*}
\|\beta^{(k)}-\beta_0\|_2^2
&\leq (1-\eta\mu)^2\|\beta^{(k-1)}-\beta_0\|_2^2
+2\eta(1-\eta\mu)\|\beta^{(k-1)}-\beta_0\|_2
\left\|\dot\ell^{(k)}(\beta_0)+\sum_{s=1}^S v_sW_s^{(k)}\right\|_2\\
&\quad+\eta^2\left\|\dot\ell^{(k)}(\beta_0)+\sum_{s=1}^S v_sW_s^{(k)}\right\|_2^2\\
&\leq \gamma\|\beta^{(k-1)}-\beta_0\|_2^2
+\frac{2\eta^2}{1-(1-\eta\mu)^2}
\left\|\dot\ell^{(k)}(\beta_0)+\sum_{s=1}^S v_sW_s^{(k)}\right\|_2^2,
\end{align*}
where the second inequality follows by Young's inequality with $\gamma := \{1 + (1 - \eta \mu)^2\}/2 \leq 1$.

On $\mathcal{E}_k^c$, the projection still allows us to bound $\|\beta^{(k)} - \beta_0\|_2^2 \leq 4C_\beta^2$. Hence we have
\begin{align}
\|\beta^{(k)}-\beta_0\|_2^2
&\leq\gamma\|\beta^{(k-1)}-\beta_0\|_2^2\notag\\
&\quad+\frac{2\eta^2}{1-(1-\eta\mu)^2}\left\|\dot\ell^{(k)}(\beta_0)+\sum_{s=1}^S v_sW_s^{(k)}\right\|_2^2+4C_\beta^2\mathbbm{1}\{\mathcal E_k^c\}.
\label{eq:fdp-pathwise-round-bound}
\end{align}
By \Cref{lemma:true_grad_convergence} and the noise calibration,
\begin{align}
\mathbb E\left\|\dot\ell^{(k)}(\beta_0)+\sum_{s=1}^S v_sW_s^{(k)}\right\|_2^2 &\lesssim\sum_{s=1}^S\frac{v_s^2}{b_s} +\sum_{s=1}^S\frac{v_s^2d \{2\log(1/\delta_s)+\epsilon_s\}(1+\log b_s)^2}
{b_s^2\epsilon_s^2}\notag\\
&\lesssim\frac{\max_{s\in[S]} \{\log(1/\delta_s)(1+\log b_s)^2\}}{\sum_{s=1}^S\min\{b_s,b_s^2\epsilon_s^2/d\}}.
\label{eq:per-iter-noise}
\end{align}
For $\epsilon_s\leq1$, the extra $1/\epsilon_s$ term
is absorbed by the $\log(1/\delta_s)/\epsilon_s^2$ term.
For the remaining servers, the first condition in
\eqref{eq:Hessian_samplesize} gives
\[
\sum_{s:\epsilon_s>1}\frac{dv_s^2}{b_s^2\epsilon_s} \leq\sum_{s=1}^S\frac{dv_s^2}{b_s^2} \leq\frac{d\sum_{s=1}^S v_s/b_s}
{\sum_{s=1}^S\min\{b_s,b_s^2\epsilon_s^2/d\}}\lesssim\frac1{\sum_{s=1}^S\min\{b_s,b_s^2\epsilon_s^2/d\}}.
\]
Here we use the appropriately chosen logarithmic factors
in that condition so that $d\sum_s v_s/b_s\lesssim1$.
Taking expectations in \eqref{eq:fdp-pathwise-round-bound} and using that $\eta \asymp d$, \eqref{eq:fdp-round-failure} and \eqref{eq:per-iter-noise}, gives the recursion 
\begin{align*}
\mathbb E\|\beta^{(k)}-\beta_0\|_2^2
&\leq\gamma\mathbb E\|\beta^{(k-1)}-\beta_0\|_2^2 + C\frac{d^2\max_{s\in[S]} \{\log(1/\delta_s)(1+\log b_s)^2\}}
{\sum_{s=1}^S\min\{b_s,b_s^2\epsilon_s^2/d\}}
\end{align*}
for some absolute constant $C>0$. Iterating with the choice of $\beta^{(0)}=0$ gives 
\begin{align*}
\mathbb E\|\beta^{(K)}-\beta_0\|_2^2 &\leq\gamma^K C_\beta^2\\
&\quad+ \frac{C}{1-\gamma}\frac{d^2\max_{s\in[S]}\{\log(1/\delta_s)(1+\log b_s)^2\}}{\sum_{s=1}^S\min\{b_s,b_s^2\epsilon_s^2/d\}}.
\end{align*}
Choosing
\[
K:=\max\left\{1,\left\lceil\frac{\log\{\max(e,\sum_{s=1}^S n_s/d^2)\}}{-\log\gamma}\right\rceil\right\},
\]
we have that
\begin{align*}
\mathbb E\|\beta^{(K)}-\beta_0\|_2^2
&\lesssim
\frac{d^2K^2\max_{s\in[S]}
\{\log(1/\delta_s)(1+\log(n_s/K))^2\}}
{\sum_{s=1}^S\min\{n_s,n_s^2\epsilon_s^2/d\}}.
\end{align*}
\end{proof}

\subsection{Proofs of lower bounds}
\label{app:Cox_lower_proof}

In the proofs of the lower bounds results, our construction uses time-independent covariates, so we simplify the notation from $Z(\cdot)$ to $Z$. \Cref{prop:FDP_lower_bound} details the lower bound result for \Cref{thm-fdp-cox}. 

\begin{proposition} Let $\mathcal{P}$ denote the set of distribution satisfying \Cref{assump-1}. For any $K \in \mathbb{N}$, let $\mathcal{M}_{K, \epsilon,\delta}$ be the set of algorithms satisfying the $(\{(\epsilon_s, \delta_s)\}_{s=1}^S, K)$-FDP constraint. Then if $\delta_s \log (1/\delta_s) \lesssim \epsilon_s^2/d$ for all $s \in [S]$, 
    \begin{equation*}
        \inf_{K \in \mathbb{N}} \inf_{Q \in \mathcal{M}_{K, \epsilon, \delta}} \inf_{\widehat{\beta}}\sup_{P \in \mathcal{P}} \mathbb{E}_{Q, P_{\beta_0}}\|\widehat{\beta} - \beta_0\|_2^2 \gtrsim 1 \wedge  \frac{d^2}{\sum_{s=1}^S \min(n_s, n_s^2\epsilon_s^2/d)}.
    \end{equation*}
    \label{prop:FDP_lower_bound}
\end{proposition}

We prove the lower bound results for Cox coefficient estimation under central and federated DP together below. 

\begin{proof}[Proof of Proposition \ref{prop:FDP_lower_bound} and Corollary \ref{prop:CDP_lowerbound}]
Our proof is based on an application of the van Trees inequality. We first construct suitable parametric submodels, together with a prior distribution, to use within the van Trees inequality. This reduces the lower bound problem into bounding the trace of the Fisher information of the private transcripts. We present a bound on the trace of the Fisher information for a central DP transcript in \Cref{lemma:cdp_trace}, and then show how this can be extended to the federated DP setting in \Cref{lemma:FDP_trace}. 

We construct a class of distributions $\mathcal{P}' \subset \mathcal{P}$ as follows. We assume that $Z(t) = Z$, with co-ordinates $Z_j \overset{\mathrm{i.i.d.}}{\sim}\mathrm{Uniform}(-C_Z'/\sqrt{d}, C_Z'/\sqrt{d})$ for $j \in [d]$ for some small enough $0 < C_Z'\leq C_Z$; that each $\beta_j$ satisfies $|\beta_j| < C_\beta' / \sqrt{d}$ for some small enough $0 < C_\beta' \leq C_{\beta}$; that the baseline hazard is $\lambda_0(t) = \min\{C_\lambda,-\log C_p\}/2$
for all $t$; and that the censoring distribution has support contained within $(1, \infty)$. Choosing $C_\beta' C_Z' \leq \log2$, we have that 
\[
\mathbb{P}(T\geq1\mid Z) =\exp\{-\lambda_0(0)\exp(\beta^\top Z)\} \geq C_p,
\]
so the at-risk condition holds uniformly over $\mathcal{P}'$.

Due to the boundedness assumptions,  we may apply Fubini's theorem to interchange the order of integration:  
\begin{equation}
   G(\beta)  = \int_0^1 \mathbb{E}_{(T, \Delta, Z)}\left[Y(s) \exp(\beta^\top Z)(Z - \mu(s, \beta))^{\otimes 2} \right] \lambda_0(s)\dint s. 
    \label{pf:population_hessian}
\end{equation}

It then suffices to obtain bounds on the eigenvalues of 
\begin{equation}
    \mathbb{E}\left[h(Z, t) \left(Z - \frac{\mathbb{E}[h(Z, t)Z]}{\mathbb{E}[h(Z, t)]} \right)^{\otimes 2} \right] \succeq 0, \quad h(z, t) := \exp(\beta^\top z - t \lambda_0(0)\exp(\beta^\top z)), \, t \in [0, 1]. 
    \label{pf:weighted_cov}
\end{equation}

We have that 
\begin{equation}
    \mathbb{E}\left[h(Z,t) \left(Z - \frac{\mathbb{E}[h(Z, t)Z]}{\mathbb{E}[h(Z, t)]} \right)^{\otimes 2} \right] = \mathbb{E}[h(Z, t)Z^{\otimes 2}] - \frac{\mathbb{E}[h(Z, t)Z]^{\otimes 2}}{\mathbb{E}[h(Z, t)]},
    \label{pf:Weyl's_ineq}
\end{equation}
and since $c_- := \exp(-C_\beta' C_Z' - \lambda_0(0)\exp(C_\beta' C_Z')) \le h(z, t) \le \exp(C_\beta' C_Z') =: c_+$, that 
\begin{equation*}
    \frac{(C'_Z)^2 c_-}{3d}I_d\preceq \mathbb{E}[h(Z, t)Z^{\otimes 2}] \preceq \frac{(C_Z')^2c_+}{3d}I_d.
\end{equation*}

Since $\mathbb{E}[Z]=0$ and
$\mathbb{E}[h(Z,t)Z]=\mathbb {E}[h(Z,t)]\mu(t,\beta)$, we have
\begin{align*}
    \mathbb E\left[h(Z,t)\{Z-\mu(t,\beta)\}^{\otimes2}\right]
    &\succeq c_-\,\mathbb E\left[\{Z-\mu(t,\beta)\}^{\otimes2}\right]\\
    &=c_-\left\{\mathbb E[Z^{\otimes2}]+\mu(t,\beta)^{\otimes2}\right\}\\
    &\succeq c_-\,\mathbb E[Z^{\otimes2}].
\end{align*}
Moreover,
\begin{align*}
    \mathbb E\left[h(Z,t)\{Z-\mu(t,\beta)\}^{\otimes2}\right]
    &=\mathbb E[h(Z,t)Z^{\otimes2}]
        -\mathbb E[h(Z,t)]\mu(t,\beta)^{\otimes2}\\
    &\preceq\mathbb E[h(Z,t)Z^{\otimes2}]
    \preceq c_+\,\mathbb E[Z^{\otimes2}].
\end{align*}
Using $\mathbb E[Z^{\otimes2}]=(C_Z')^2I_d/(3d)$, it follows that
\begin{equation*}
    \frac{(C_Z')^2c_-}{3d}I_d
    \preceq\mathbb E\left[h(Z,t)\{Z-\mu(t,\beta)\}^{\otimes2}\right]
    \preceq\frac{(C_Z')^2c_+}{3d}I_d.
\end{equation*}
Integrating against $\lambda_0(t) \dint t$ over $t\in[0,1]$, we conclude that the eigenvalues of the population Hessian are of order $1/d$, which verifies \Cref{assump-1}\ref{assp:eigenvalues}. 

In addition, for any $P \in \mathcal{P'}$, the Fisher information for $\beta_0$ is given by
\begin{align*}
    I_{(T,\Delta,Z)}(\beta_0)
    &=\mathbb{E}_Z\left[
        Z^{\otimes2} \lambda_0(0)\exp(\beta_0^\top Z)
        \mathbb{E}_{\tilde T\mid Z}[\min\{\tilde T,1\}]
    \right]\\
    &=\mathbb{E}_Z\left[
        Z^{\otimes2} \lambda_0(0)\exp(\beta_0^\top Z)
        \int_0^1\exp\{-t \lambda_0(0)\exp(\beta_0^\top Z)\}\dint t
    \right]\\
    &=\mathbb{E}_Z\left[
        Z^{\otimes2}\{1-\exp(-\lambda_0(0)\exp(\beta_0^\top Z))\}
    \right].
\end{align*}
Since $|\beta_0^\top Z|\leq C_\beta'C_Z'$, we have
\begin{equation*}
    \{1-\exp(-\lambda_0(0)\exp(-C_\beta'C_Z'))\} \frac{(C_Z')^2}{3d}I_d
    \preceq I_{(T,\Delta,Z)}(\beta_0)
    \preceq\frac{(C_Z')^2}{3d}I_d.
\end{equation*}
 This implies that 
\begin{equation}
\mathrm{Tr}(I_{(T, \Delta, Z)}(\beta)) \lesssim 1. 
\label{pf:non_private_trace}
\end{equation}

From the van Trees inequality \citep{gill1995applications}, we have 
\begin{equation}
    \int \mathbb{E}_{P_\beta, Q} \|\widehat{\beta}(R) - \beta\|_2^2 \pi(\beta) \dint \beta \geq  \frac{d^2}{\int \mathrm{Tr}(I_R(\beta)) \pi(\beta) \dint \beta + J(\pi)} \geq \frac{d^2}{\sup_\beta \mathrm{Tr}(I_R(\beta)) + J(\pi)} 
    \label{eq:van_Trees}
\end{equation}
where $R$ is a random output from any $(\epsilon, \delta)$-DP mechanism,  $I_R(\beta)$ is the Fisher information for $\beta$ based on $R$, and $\pi(\beta)$ is a prior distribution on $\beta$ with support on a subset of $B_{C_\beta}(0)$ where the $\beta_j$'s are independently distributed and $J(\pi)$ is defined as
\begin{equation*}
J(\pi) := \sum_{j=1}^d \int\frac{( \pi_j'(\beta_j))^2}{\pi_j(\beta_j)} \dint \beta_j.
\end{equation*}
Let $\pi := \otimes_{j=1}^d \pi_j$, where 
\begin{equation*}
    \pi_j(\beta_j) := \frac{15\sqrt{d}}{16 C_\beta'} \left( 1 - \frac{d \beta_j^2}{(C_\beta')^2} \right)^2 \mathbbm{1} \left\{|\beta_j| \leq \frac{C_\beta'}{\sqrt{d}} \right\},
\end{equation*}
then we have that $\mathrm{supp}(\pi) \subset B_{C_\beta}(0)$ and $J(\pi) \lesssim d^2$. 

Using either \Cref{lemma:cdp_trace} (for CDP mechanisms) or \Cref{lemma:FDP_trace} (for FDP mechanisms) to bound 
\begin{equation*}
    \sup_\beta \mathrm{Tr}(I_R(\beta)) \leq \sum_{s=1}^S \min(n_s, n_s^2 \epsilon_s^2/d),
\end{equation*}
and plugging this into \eqref{eq:van_Trees} finishes the proof.
\end{proof}

\begin{lemma}
\label{lemma:cdp_trace}
Suppose that we have $n$ observations of $(T_i, \Delta_i, Z_i) \overset{\mathrm{i.i.d.}}{\sim}P \in \mathcal{P'}$, where $\mathcal{P'}$ is as constructed in the proof of \Cref{prop:FDP_lower_bound}. If $R$ is the random output of a mechanism that satisfies $(\epsilon, \delta)$-CDP, with $\delta\log(1/\delta)\lesssim\epsilon^2/d$, then
    \begin{equation*}
        \mathrm{Tr}(I_R(\beta)) \lesssim \min(n, n^2 \epsilon^2 / d).
    \end{equation*}    
\end{lemma}
\begin{proof} 
This proof is adapted from the proofs of Theorem 4.1 and Lemma 4.2 of \cite{cai2024optimal}. The first term in the minimum can be obtained by the data processing inequality for Fisher information and \eqref{pf:non_private_trace}. Let $D:=(D_1, \dots, D_n)$ denote the dataset.  The second term arises by considering
\begin{align}
    \mathbb{E}[S_\beta(D) \mid R=r] &= \int_{(\mathbb{R}^+, \{0, 1\}, \mathbb{R}^d)^n} \left( \nabla_\beta \log f_\beta(D) \right) f_\beta(D \mid R=r)\dint D \notag \\
    &= \int_{(\mathbb{R}^+, \{0, 1\}, \mathbb{R}^d)^n} \frac{\nabla_\beta f_\beta(D)}{f_\beta(D)} \frac{f_\beta(D, r)}{f_\beta(R=r)} \dint D=\nabla_\beta \log f_\beta(r) 
    \label{pf:transcipt_to_score}
\end{align}
which implies $I_R(\beta)=\mathbb E[\{\mathbb E[S_\beta(D)\mid R]\}^{\otimes2}]$. 

By the linearity of expectation, we have that 
\begin{align*}
    \mathrm{Tr}(\mathbb{E}[\mathbb{E}[S_\beta(D) \mid R]^{\otimes2}]) &= \mathbb{E}[\mathrm{Tr}(\mathbb{E}[S_\beta(D) \mid R]^{\otimes2})]
    = \mathbb{E}\left[\|\mathbb{E}[S_\beta(D) \mid R]\|_2^2\right] \\
    &= \sum_{i=1}^n \mathbb{E}[\langle S_\beta(D_i), \mathbb{E}[S_\beta(D) \mid R] \rangle] =: \sum_{i=1}^n \mathbb{E} G_i
\end{align*}
where the third equality is from the fact that for any random variables $A$ and $B$, the quantities $A - \mathbb{E}[A\mid B]$ and $\mathbb{E}[A \mid B]$ are uncorrelated. For $i \in [n]$, let $\breve{G}_i := \langle S_\beta(\breve{D}_i), \mathbb{E}[S_\beta(D) \mid R] \rangle$, where $\breve{D}_i$ is an independent copy of $D_i$ and denote $(\cdot)_+ = \max(0, \cdot)$ and $(\cdot)_- = -\min(0, \cdot)$. By the $(\epsilon, \delta)$-DP constraint, we have that for any $W > 0$, 
\begin{equation*}
    \mathbb{E}(G_i)_+ = \int_0^\infty\mathbb{P}((G_i)_+ \geq t) \dint t \leq \mathbb{E}(\breve{G}_i)_+ + C_\epsilon \epsilon \int_0^\infty \mathbb{P}((\breve{G}_i)_+ \geq w) \dint w + W\delta +  \int_W^\infty \mathbb{P}((G_i)_+ > w) \dint w
\end{equation*}
and similarly that 
\begin{equation*}
    \mathbb{E}(G_i)_- \geq \mathbb{E}(\breve{G}_i)_- - C_\epsilon \epsilon \int_0^\infty \mathbb{P}((\breve{G}_i)_- \geq w) \dint w - W\delta -  \int_W^\infty \mathbb{P}((\breve{G}_i)_- \geq w) \dint w 
\end{equation*}
Putting the last two equations together, we have 
\begin{align}
    \mathbb{E}[G_i] &= \mathbb{E}[(G_i)_+ - (G_i)_-] \notag \\
    &\lesssim  \mathbb{E}[\breve{G}_i] +  \epsilon \mathbb{E}|\breve{G}_i| + 2W\delta +  \int_W^\infty \{\mathbb{P}((G_i)_+ \geq w) + \mathbb{P}((\breve{G}_i)_- \geq w)\} \dint w.
    \label{pf:score_attack-2}
\end{align}

By independence, we have $\mathbb{E}[\breve{G}_i] = 0$. For the second term of \eqref{pf:score_attack-2}, by Jensen's and the Cauchy--Schwarz inequality, we have
\begin{align*}
\mathbb{E}|\breve{G}_i| &\leq \sqrt{ \mathbb{E} \left[ \mathbb{E}[S_\beta(D)\mid R]^\top \mathrm{Var}(S_\beta(\breve{D}_i)) \mathbb{E}[S_\beta(D)\mid R] \right] } \\
&\leq \sqrt{\lambda_{\max} \left(I_{\breve{D}_i}(\beta) \right)} \sqrt{ \mathbb{E} \left\|\mathbb{E}[S_\beta(D)\mid R] \right\|_2^2 }  \\
&\lesssim \sqrt{1/d} \sqrt{ \mathrm{Tr}\left(I_R(\beta)\right)}
\end{align*}
where the last line follows by the construction of $\mathcal{P}'$ and $\eqref{pf:transcipt_to_score}$. 

To obtain the tail bound for the final term on the right-hand side of \eqref{pf:score_attack-2}, we use the boundedness assumptions to obtain   
\begin{align}
    \|S_\beta (T_{i}, \Delta_{i}, Z_{ i}))\|_2 &= \left\|Z_{i} \left(\Delta_{i} - \int_0^{T_{i}} \lambda_0(u) \exp(\beta^\top Z_{i}) \dint u\right) \right\|_2 \notag\\
    &\leq C_Z \left(1 + e^{C_\beta C_Z}\Lambda_0(1) \right) =: C_S. 
    \label{eq:score_bound}
\end{align}

By \eqref{eq:score_bound}, we have that 
\begin{equation*}
    \mathbb{E}[\exp(\lambda |G_i|)] \leq \exp(\lambda C_S^2n)
\end{equation*}
and therefore by a Chernoff bound that 
\begin{equation*}
    \mathbb{P}(|G_i| \geq w) \leq \exp\left(1 - \frac{w}{C_S^2 n} \right).
    \label{eq:score_chernoff}
\end{equation*}
Applying similar arguments for a tail bound on $|(\breve{G}_i)_-|$ and taking $W = \log(1/\delta) C_S^2n$, we have 
\begin{equation}
    \int_W^\infty \left\{ \mathbb{P}((G_i)_+ \geq w) + \mathbb{P}((\breve{G}_i)_- \geq w) \right\} \dint w \lesssim C_S^2 \delta n
    \label{eq:lowerbound_prob_control}
\end{equation}
By summing over $\mathbb{E}[G_i]$ over $i \in [n]$, we have
\begin{equation*}
    \mathrm{Tr}(I_R(\beta)) \lesssim n \left( \epsilon \sqrt{1/d} \sqrt{\mathrm{Tr}(I_R(\beta))} + \log(1/\delta)\delta n \right),
\end{equation*}
which implies that $\mathrm{Tr}(I_R(\beta)) \lesssim (n\epsilon)^2/d$ as long as $\delta \log (1/\delta) \lesssim \epsilon^2 /d$. 
\end{proof}

\begin{lemma}
\label{lemma:FDP_trace}
Suppose that we have observations $\{(T_{s,i}, \Delta_{s,i}, Z_{s,i})_{i=1}^{n_s}\}_{s=1}^S \overset{\mathrm{i.i.d.}}{\sim}P \in \mathcal{P'}$, where $\mathcal{P'}$ is as constructed in the proof of \Cref{prop:FDP_lower_bound}. If $R$ is the random output of a mechanism that satisfies $\{(\epsilon_s, \delta_s)\}_{s=1}^S$-FDP, with $\delta_s\log(1/\delta_s)\lesssim\epsilon_s^2/d$ for all $s \in [S]$, we have that 
    \begin{equation*}
        \mathrm{Tr}(I_R(\beta)) \lesssim \sum_{s=1}^S \min(n_s, n_s^2 \epsilon_s^2/d).
    \end{equation*}    
\end{lemma}
\begin{proof}
    We use the arguments in the proof of Theorem 4 of  \cite{xue2024optimal}, where the assumption of independent batches of data allows us to follow similar steps to the proof of Lemma \ref{lemma:cdp_trace}. By the chain rule for Fisher information, we have 
\begin{equation}
    I_R(\beta) = \sum_{k=1}^K \sum_{s=1}^S I_{R_s^{(k)} \mid M^{(k-1)}}(\beta),
\end{equation}
where $R_s^{(k)}$ denotes the transcript from server $s$ on iteration $k$, and $M^{(k-1)} := \cup_{l=1}^{k-1} \cup_{s=1}^S R_s^{(l)}$. Let $D_s^{(k)}$ be the dataset used by server $s$ on iteration $k$. An analogous result to \eqref{pf:transcipt_to_score} may be obtained as  
\begin{align*}
    \mathbb{E}\left[S_\beta(D_{s}^{(k)} )\mid R^{(k)}_s = r, M^{(k-1)} \right] &= \int \frac{ \nabla_\beta f(x \mid M^{(k-1)})} {f(x  \mid M^{(k-1)})} \frac{f(x, R_s^{(k)} = r \mid M^{(k-1)})}{f(R_s^{(k)} = r \mid M^{(k-1)})} \dint x \\
    &= \frac{1}{f(R_s^{(k)}=r \mid M^{(k-1)})} \int f(r \mid x, M^{(k-1)}) \nabla_\beta f(x \mid M^{(k-1)}) \dint x\\
    &= \nabla_\beta \log f(R^{(k)}_s = r\mid M^{(k-1)}),
\end{align*}
where the first equality is due to the use of independent batches of data. Hence, we have 
\begin{equation*}
    I_{R^{(k)}_s \mid M^{(k-1)}}(\beta) = \mathbb{E} \left[\mathbb{E}_{R^{(k)}_s \mid M^{(k-1)}}  \left\{\mathbb{E}\left[ S_\beta(D_s^{(k)}) \mid R_s^{(k)}, M^{(k-1)}\right]^{\otimes 2} \right\} \right]. 
\end{equation*}
Defining $G_{s, i}^{(k)} := \left\langle S_\beta(D_{s, i}^{(k)}), \mathbb{E}[S_\beta(D_s^{(k)}) \mid R_s^{(k)}, M^{(k-1)}] \right\rangle$ and 
\[
\breve{G}_{s, i}^{(k)} := \left\langle S_\beta(\breve{D}_{s, i}^{(k)}), \mathbb{E}[S_\beta(D_s^{(k)}) \mid R_s^{(k)}, M^{(k-1)}]\right\rangle,
\]
where $\breve{D}_{s, i}^{(k)}$ is an independent copy of $D_{s, i}^{(k)}$, we obtain the conditional version of \eqref{pf:score_attack} as 
\begin{equation*}
    \mathbb{E}[G_{s, i}^{(k)}] \lesssim \mathbb{E}[\breve{G}_{s, i}^{(k)}] +  \epsilon_s \mathbb{E}|\breve{G}_{s, i}^{(k)}| + 2W\delta_s +  \int_W^\infty \mathbb{P}((G_{s, i}^{(k)})_+ \geq w) + \mathbb{P}((\breve{G}_{s, i}^{(k)})_- \geq w) \dint w,
\end{equation*}
for any $W>0$; this can be shown using that for random variables $A, B, C$, the variables $A - \mathbb{E}[A \mid B, C]$ and $\mathbb{E}[A \mid B, C]$ are uncorrelated. By similar arguments and using the bound on $S_\beta(\cdot)$ from \eqref{eq:score_bound}, we have
\begin{equation*}
    \mathrm{Tr}(I_{R_s^{(k)} \mid M^{(k-1)}}(\beta)) \lesssim (\epsilon_s b_s^{(k)})^2/d, 
\end{equation*}
provided that $\delta_s \log (1/\delta_s) \lesssim \epsilon_s^2/d$, where $b_s^{(k)}$ is the number of observations used by server $s$ at iteration $k$. Since $\{b_s^{(k)}\}_{k=1}^K$ such that $\sum_{k=1}^K b_s^{(k)} \leq n_s$, we have that $\sum_{k=1}^K (b_s^{(k)})^2 \leq n_s^2$, which implies for all $s \in [S]$,
\begin{equation*}
    \sum_{k=1}^K \mathrm{Tr}(I_{R_s^{(k)} \mid M^{(k-1)}}(\beta)) \lesssim \min(n_s, n_s^2 \epsilon_s^2/d),
\end{equation*}
where the first term in the minimum is from the the data processing inequality for Fisher information and \eqref{pf:non_private_trace}.
\end{proof}

\subsection{Auxiliary results}\label{app:beta_auxiliary}
We start by presenting a lemma on the sum of the distributed Hessian, which is used in the high probability event in the proof of the uppper bound of \Cref{thm-fdp-cox}. The remaining lemmas in this subsection are for a single-server setting and are not specific to differential privacy. For these lemmas, we assume that the observations $\{T_{i}, \Delta_{i}, \{Z_{i}(t): \, t\in [0, T_{ i}]\}\}_{i = 1}^{n}$ are i.i.d.~from a distribution satisfying \Cref{assump-1}, and write $\dot{\ell}$  and $\ddot{\ell}$ for \eqref{eq:cox_gradient} and \eqref{eq:Cox_hessian} based on a single server with $n$ samples.

\begin{lemma}[Distributed Hessian convergence]
Suppose that the data $\{T_{s, i}, \Delta_{s, i}, \{Z_{s, i}(t): \, t \in [0, T_{s, i}]\}_{}\}_{s, i=1}^{S, n_s}$ are independent and identically distributed (i.i.d.)~from a distribution satisfying \Cref{assump-1}. Let $b_s$ and $v_s$ be defined as in Algorithm \ref{alg:FDP-SGD} and let $\ell_s$ be the normalised log-likelihood \eqref{eq:cox_gradient} based on $b_s$. If we may assume that 
\begin{equation}
     \frac{\sum_{s=1}^S\min\{1, b_s \epsilon_s^2/d\}}{\sum_{s=1}^S \min\{b_s, b_s^2 \epsilon_s^2 / d\}} \lesssim_{\log} \frac{1}{d}, \quad \mathrm{and}  \quad \frac{
        d^2\sum_{s=1}^S \min\{b_s,b_s^3\epsilon_s^4/d^2\}
    }{
        \left(\sum_{s=1}^S \min\{b_s,b_s^2\epsilon_s^2/d\}\right)^2
    } \lesssim_{\log} 1
    \label{eq:Hessian_samplesize}
\end{equation}
for appropriately chosen log factors, then we may bound 
\begin{equation*}
    \mathbb{P} \left(  \left\| \sum_{s=1}^S v_s \ddot{\ell}_s(\beta_0) + G(\beta_0) \right\|\geq  \frac{\rho_-}{2d} \right) \lesssim  \frac{d^2}{\sum_{s=1}^S \min(b_s, b_s^2 \epsilon_s^2/d)}.
\end{equation*}
\label{lemma:distributed_hessian}
\end{lemma}
\begin{proof}
Recall that $G := \mathbb{E} [\int_0^1 G(s, \beta_0) \dint \Lambda_0(s)]$ is defined as a population version of $-\ddot{\ell}(\beta_0)$. We have from \Cref{lemma:hessian_convergence} that for some absolute constants $c_1, c_2 > 0$, 
\begin{equation}
\mathbb{P}\left(\|\ddot{\ell}_s(\beta_0) +G\|_2 \geq \frac{c_1 \log(b_sd)}{b_s^{1/2}} \right) \leq \frac{c_2}{b_s}.
\end{equation} 
Let $B_s := \mathbbm{1}\left\{\|\ddot{\ell}_s(\beta_0) +G\|_2 > c_1 \log(b_sd)/b_s^{1/2} \right\}$; we will use this to separately control
\begin{align*}
\mathbb{P} \left( \left\|\sum_{s=1}^S v_s \ddot{\ell}_s(\beta_0) + G \right\|_2 \geq \frac{\rho_-}{2d} \right)
    &\leq \mathbb{P} \left( \left\|\sum_{s=1}^S v_s (\ddot{\ell}_s(\beta_0) +G) B_s \right\|_2 \geq \frac{\rho_-}{4d} \right) 
    \\ &\quad \quad + \mathbb{P} \left( \left\|\sum_{s=1}^S v_s (\ddot{\ell}_s(\beta_0) +G)(1-B_s) \right\|_2 \geq \frac{\rho_-}{4d} \right)\\ 
    &=: (I) + (II),
\end{align*}
where $v_s$ are the weights defined in Algorithm \ref{alg:FDP-SGD}. Since $\|\ddot{\ell}_s(\beta_0) + G\|_2 \leq 8C_Z^2$ and each $B_s \preceq$ Bernoulli($c_2/b_s$) independently, for the first term we have 
\begin{equation*}
    (I) \leq \mathbb{P}\left( \sum_{s=1}^S v_s B_s \geq \frac{\rho_-}{32C_Z^2d}\right), 
\end{equation*}
where we have that 
\begin{equation*}
    \mathbb{E}\left[\sum_{s=1}^S v_s B_s\right] \leq \frac{\sum_{s=1}^S c_2\min(1, b_s \epsilon_s^2/d)} {\sum_{s=1}^S \min(b_s, b_s^2 \epsilon_s^2 / d)}
\end{equation*}
and 
\begin{equation*}
\mathrm{Var}\left( \sum_{s=1}^S v_sB_s \right) \leq \sum_{s=1}^S c \frac{v_s^2}{b_s} \leq \frac{c_2}{\sum_{s=1}^S \min(b_s, b_s^2\epsilon_s^2/d)}.
\end{equation*}
If we may assume 
\begin{equation*}
    \frac{\sum_{s=1}^S \min\{1, b_s \epsilon_s^2/d\}}{\sum_{s=1}^S \min\{b_s, b_s^2 \epsilon_s^2/d\}} \leq \frac{c \rho_-}{d}
\end{equation*}
for a sufficiently small absolute constant $c>0$,
then by Chebyshev's inequality we have that 
\begin{equation}
    (I) \leq \frac{(64C_Z^2 d)^2}{\rho_-^2 \sum_{s=1}^S \min(b_s, b_s^2 \epsilon_s^2/d)}
    \label{eq:Weyl1}
\end{equation}

For the second term, letting $A_s := (-\ddot{\ell}_s(\beta_0) - G)(1 - B_s)$, we may bound
\begin{align*}
    (II) = \mathbb{P}\left( \left \|\sum_{s=1}^S v_s A_s \right \|_2 \geq \frac{\rho_-}{4d} \right) &\leq \mathbb{P} \left(\left\|\sum_{s=1}^S v_s A_s
    - \mathbb{E} \left[ \sum_{s=1}^S v_s A_s \right] \right\|_2  + \left\|\mathbb{E} \left[ \sum_{s=1}^S v_s A_s \right] \right\|_2 \geq \frac{\rho_-}{4d} \right)\\
    &\leq \mathbb{P} \left(\left\|\sum_{s=1}^S v_s A_s
    - \mathbb{E} \left[ \sum_{s=1}^S v_s A_s \right] \right\|_2 \geq \frac{\rho_-}{8d} \right),
\end{align*}
provided that 
\begin{equation*}
    \left\|\mathbb{E} \left[\sum_{s=1}^S v_s A_s \right] \right\|_2 < \frac{\rho_-}{8d}.
\end{equation*}
By \Cref{lemma:hessian_bias}, the left-hand side above can be bound by
\begin{align*}
     \left\|\mathbb{E} \left[\sum_{s=1}^S v_s (\ddot{\ell}_s(\beta_0) + G)(1-B_s)  \right] \right\|_2 &\leq \left\|\mathbb{E} \left[\sum_{s=1}^S v_s (\ddot{\ell}_s(\beta_0) + G) \right] \right\|_2 + 8C_Z^2\sum_{s=1}^S v_s \mathbb{P} \left( B_s = 1\right) \\
     &\lesssim \sum_{s=1}^S \frac{v_s \log(1+b_s)^2}{b_s}.
\end{align*}
Since $A_s$ is self-adjoint and we have the bound $\|v_s\{A_s - \mathbb{E}[A_s]\}\|_2 \leq 2 c_1 v_s \log(b_sd)/b_s^{1/2}$ a.s., by a matrix Hoeffding inequality \citep[e.g.~Theorem 1.3 in][]{tropp2012user} , we have 
\begin{equation}
    \mathbb{P} \left( \left\|\sum_{s=1}^S v_sA_s - \mathbb{E}[v_sA_s] \right\|_2 \geq \frac{\rho_-}{8d}\right) \leq 2d \exp \left( \frac{-(\rho_-/8)^2}{32 d^2\sum_{s=1}^S v_s^2 \log^2 (b_sd)/b_s}\right).
    \label{eq:Weyl2}
\end{equation}
Therefore, using \eqref{eq:Weyl1} and \eqref{eq:Weyl2}, we have that 
\begin{align*}
 \mathbb{P} \left(  \left\| \sum_{s=1}^S v_s \ddot{\ell}_s(\beta_0) + G(\beta_0) \right\|\geq  \frac{\rho_-}{2d} \right)
&\lesssim
2d\exp\left(
-\frac{c\rho_-^2\left(\sum_{s=1}^S\min\{b_s,b_s^2\epsilon_s^2/d\}\right)^2}
{d^2\sum_{s=1}^S\min\{b_s,b_s^3\epsilon_s^4/d^2\}\log^2(b_sd)}
\right)\\
& \qquad+
\frac{d^2}{\sum_{s=1}^S\min\{b_s,b_s^2\epsilon_s^2/d\}}. 
\end{align*}
The second term dominates the first under the second sample size condition in \eqref{eq:Hessian_samplesize}. 
\end{proof}

\begin{lemma}[Lemma 3.2 in \citealp{huang2013oracle}]
\label{lemma:3-2_huang}
    Let $\dot{\ell}(\beta)$ and $\ddot{\ell}{(\beta)}$ be as defined in \eqref{eq:cox_gradient} and \eqref{eq:Cox_hessian} -- note that this is the negative of the notation in \cite{huang2013oracle}. For any $\beta, b \in \mathbb{R}^d$, it holds that 
    \begin{equation*}
    e^{-\eta_b}  b^\top \ddot{\ell}(\beta) b \geq b^\top \left[ \dot{\ell}(\beta + b) - \dot{\ell}(\beta) \right] \geq e^{\eta_b} b^\top \ddot{\ell}(\beta) b,
    \end{equation*}
where $\eta_b := \max_t \max_{i, j \in [n]} |b^\top \{Z_i(t) - Z_j(t)\}|$. Moreover,
\begin{equation*}
    -e^{-2\eta_b} \, \ddot{\ell}(\beta) \;\preceq\; -\ddot{\ell}(\beta + b) \;\preceq\; -e^{2\eta_b} \, \ddot{\ell}(\beta).
\end{equation*}
\end{lemma}

\begin{lemma}[Convergence of $\dot{\ell}(\beta_0)$] When $\dot{\ell}(\cdot)$ is evaluated at the true parameter value, we have that  
    \begin{equation*}
        \mathbb{P}\left( \|\dot{\ell}(\beta_0)\|_2 > u\right) \lesssim \exp ( -cnu^2)
    \end{equation*}
    for some absolute constant $c$. 
\label{lemma:true_grad_convergence}
\end{lemma}
\begin{proof}
Since the compensator of $N_i(t)$ is $\int_0^t Y_i(s)  \exp(\beta_0^\top Z_i(s)) \lambda_0(s)\dint s$,  we have that 
    \begin{equation*}
        \dot{\ell}(\beta_0) = \frac{1}{n} \sum_{i=1}^n\int_0^{1} \{Z_i(t) - \bar{Z}(t, \beta_0)\} \dint N_i(t) = \frac{1}{n}\sum_{i=1}^n \int_0^{1}\{Z_i(t \wedge T_i) - \bar{Z}(t)\} \dint M_i(t)
    \end{equation*}
    where each $M_i(t) := N_i(t) - \int_0^t Y_i(s) \exp(\beta_0^\top Z_i(s)) \lambda_0(s)\dint s$ is an independent martingale. Each process $\{Z_i(t\wedge T_i)-\bar Z(t)\}_{t\in[0,1]}$ is predictable w.r.t.~the filtration $(\mathcal{F}_t)$ that is generated by the observed histories up to time $t$, i.e 
  \[
\mathcal F_t:=\sigma\left( \{Z_i(u\wedge T_i),\mathbbm{1}\{T_i\leq u\},N_i(u): 0\leq u\leq t\}_{i=1}^n \right).
\]  
We therefore have that 
    \begin{equation}
        X(t) := \frac{1}{n}\sum_{i=1}^n \int_0^{t}\{Z_i(s \wedge T_i) - \bar{Z}(s)\} \dint M_i(s)
    \end{equation}
    is a martingale w.r.t.~$(\mathcal{F}_t)_{t \in [0, 1]}$.

Set $\tau_0:=0$ and let $\tau_k$ be the time of the $k$th observed failure, or $1$ if fewer than $k$ failures occur. Since these are bounded stopping times, $\{X(\tau_k)\}_{k=0}^n$ is a discrete martingale. Since 
\[
    \sum_{i=1}^nY_i(s) \exp(\beta_0^\top Z_i(s))\{Z_i(s)-\bar Z(s)\}=0,
\]
we have that 
\[
    X(t)=\frac{1}{n}\sum_{i=1}^n\int_0^t\{Z_i(s)-\bar Z(s)\}\dint N_i(s).
\]
Since failures do not occur at the same time almost surely, we can bound the martingale increments by $\|X(\tau_k)-X(\tau_{k-1})\|_2\leq2 C_Z/n$. Noting that $X(\tau_n) = \dot{\ell}(\beta_0)$, the result follows from \Cref{lemma:smooth_banach_martingale}. 
\end{proof}

\begin{lemma} If the Euclidean norm of the covariates are bounded by $C_Z$ a.s., we can bound 
$\lambda_{\max}(-\ddot{\ell}(\beta)) \leq 4C_Z^2$ for any $\beta$. 
\label{lemma:hessian_upper}
\end{lemma}
\begin{proof}
    From \eqref{eq:Cox_hessian}, it suffices to find a bound on $V_n(t, \beta)$ that holds for all values of $t, \beta$. Since $\bar{Z}$ is a weighted average, $\sup_{i, t} |Z_{i}(t)| \leq C_Z$ implies $\|\bar{Z}(t)\|_2 \leq C_Z$. The weights in the summand of $V_n(t, \beta)$ sum to 1, so by the triangle inequality, we have $0 \preceq V_n(t, \beta) \preceq  4C_Z^2 I_d$.
\end{proof}

\begin{lemma}[Hessian bias]
Let $G(\beta_0) := \mathbb{E} [\int_0^1 G(s, \beta_0) \dint \Lambda_0(s)]$, and suppose that \Cref{assump-1} holds. We have that 
\begin{equation*}
    \left\|\mathbb{E} \left[ \ddot{\ell}(\beta_0) + G(\beta_0) \right] \right\|_2 \lesssim_{\log} \frac{1}{n}.
\end{equation*}
    \label{lemma:hessian_bias}
\end{lemma}
\begin{proof}
    We can write 
\begin{align*}
    \mathbb{E}[\ddot{\ell}(\beta_0)] &= \frac{-1}{n}\mathbb{E} \left[\sum_{i=1}^{n}  \int_0^1 \frac{\sum_{j=1}^{n} Y_{j}(t) \exp\{\beta_0^{\top} Z_{j}(t)\} \left\{Z_{j}(t) - \bar{Z}(t, \beta_0) \right\}^{\otimes 2}}{n S^{(0)}(t, \beta_0)}\dint N_{i}(t)  \right] \\
    &= \frac{-1}{n} \sum_{i=1}^{n} \mathbb{E} \left[\int_0^1 Y_{i}(t) \exp(\beta_0^\top Z_{i}(t)) \left\{Z_{i}(t) - \bar{Z}(t, \beta_0) \right\}^{\otimes 2} \dint \Lambda_0(t)  \right] \\
    &= \frac{-1}{n} \sum_{i=1}^{n} \mathbb{E} \left[\int_0^1 Y_{i}(t)\exp(\beta_0^\top Z_{i}(t)) \left\{Z_{i}(t) - \mu(t, \beta_0) \right\}^{\otimes 2} \dint \Lambda_0(t) \right]  \\
    & \quad- \frac{1}{n} \sum_{i=1}^{n} \mathbb{E} \left[\int_0^1 Y_{i}(t)\exp(\beta_0^\top Z_{i}(t)) \left\{Z_{i}(t) - \mu(t, \beta_0) \right\} \left\{ \mu(t, \beta_0) - \bar{Z}(t, \beta_0) \right\}^\top \dint \Lambda_0(t) \right] \\
    & \quad- \frac{1}{n} \sum_{i=1}^{n} \mathbb{E} \left[\int_0^1 Y_{i}(t)\exp(\beta_0^\top Z_{i}(t)) \left\{ \mu(t, \beta_0) - \bar{Z}(t, \beta_0) \right\}\left\{Z_i(t) - \mu(t, \beta_0) \right\}^\top \dint \Lambda_0(t) \right] \\
    & \quad- \frac{1}{n} \sum_{i=1}^{n} \mathbb{E} \left[\int_0^1 Y_{i}(t)\exp(\beta_0^\top Z_{i}(t)) \left\{\mu(t, \beta_0) - \bar{Z}(t, \beta_0) \right\}^{\otimes 2} \dint \Lambda_0(t) \right] \\
    &= -G(\beta_0)+\mathbb{E}\left[\int_0^1 S^{(0)}(t,\beta_0)\{\bar Z(t,\beta_0)-\mu(t,\beta_0)\}^{\otimes2}\dint\Lambda_0(t)\right] \\
    &= -G(\beta_0) + B(\beta_0).
\end{align*}
The penultimate equality is due to 
\[
\frac1n\sum_{i=1}^nY_i(t)\exp(\beta_0^\top Z_i(t))\{Z_i(t)-\mu(t,\beta_0)\}
=S^{(0)}(t,\beta_0)\{\bar Z(t,\beta_0)-\mu(t,\beta_0)\}.
\]
Using \Cref{lemma:weighted_mean_convergence}, we may bound 
\begin{align*}
    \|B(\beta_0)\|_2 &\lesssim \mathbb{E} \left[\sup_{t \in [0, 1]} \|\mu(t) - \bar{Z}_s(t, \beta_0)\|_2^2 \right] \\
    &\leq  \int_0^{(2C_Z)^2} \mathbb{P} \left( \sup_{t \in [0, 1]} \|\bar{Z}(t) - \mu(t)\|_2^2 > x\right) \dint x   \\
    &\leq \left(\frac{c \log(n)}{\sqrt{n}} \right)^2 + 2 \int_{c \log(n)/\sqrt{n}}^{2C_Z} y \mathbb{P} \left( \sup_{t \in [0, 1]} \|\bar{Z}(t) - \mu(t)\|_2 > y \right) \dint y\\
    &\lesssim \frac{\log(n)^2}{n}.
\end{align*}
\end{proof}

\begin{lemma}[Convergence of weighted mean]
Under \Cref{assump-1}, there exists constants $c_1, c_2 > 0$ such that 
\begin{equation*}
    \mathbb{P}\left( \sup_{t \in [0, 1]}\|\bar{Z}(t, \beta_0) - \mu(t, \beta_0)\|_2 > \frac{c_1\log(n)}{\sqrt{n}} \right) \lesssim \sqrt{n} \exp(-c_2 \log(n)^2).
\end{equation*}
\label{lemma:weighted_mean_convergence}
\end{lemma}
\begin{proof}
We can decompose 
\begin{align*}
    \|\bar{Z}(t, \beta_0) - \mu(t, \beta_0)\|_2 &= \left\|\frac{S^{(1)}(t, \beta_0)}{S^{(0)}(t, \beta_0)} - \frac{\mathbb{E}[S^{(1)}(t, \beta_0)]}{\mathbb{E}[S^{(0)}(t, \beta_0)]} \right\|_2\\
    &= \left\|\frac{S^{(1)}(t, \beta_0) - \mathbb{E}[S^{(1)}(t, \beta_0)]}{S^{(0)}(t, \beta_0)} - \frac{\mathbb{E}[S^{(1)}(t, \beta_0)]\{S^{(0)}(t, \beta_0) - \mathbb{E}[S^{(0)}(t, \beta_0)]\}}{S^{(0)}(t, \beta_0) \mathbb{E}[S^{(0)}(t, \beta_0)]} \right\|_2\\
    &\leq \left\|\frac{S^{(1)}(t, \beta_0) - \mathbb{E}[S^{(1)}(t, \beta_0)]}{S^{(0)}(t, \beta_0)} \right\|_2+ \left\|\frac{\mathbb{E}[S^{(1)}(t, \beta_0)]\{S^{(0)}(t, \beta_0) - \mathbb{E}[S^{(0)}(t, \beta_0)]\}}{S^{(0)}(t, \beta_0) \mathbb{E}[S^{(0)}(t, \beta_0)]} \right\|_2.
\end{align*}
We first construct nets to control the two processes $\{S^{(1)}(t, \beta_0) - \mathbb{E}[S^{(1)}(t, \beta_0)]\}_{t \in [0, 1]}$ and $\{S^{(0)}(t, \beta_0) - \mathbb{E}[S^{(0)}(t, \beta_0)]\}_{t \in [0, 1]}$. Set $q_0=0$. For $q_k<1$, if $\mathbb{P} ( q_k < T < 1) \leq n^{-1/2}$, set $q_{k+1}=1$ and stop. Otherwise, define 
\begin{equation*}
    q_{k+1}:=\inf\left\{t\in(q_k,1]: \mathbb{P}(q_k<T\leq t)\geq \frac{1}{\sqrt n}\right\}.
\end{equation*} By construction, we have that 
\begin{equation*}
    \mathbb{P}(q_k<T<q_{k+1}) \leq \frac{1}{\sqrt{n}} \quad \mathrm{and} \quad \mathbb{P}(q_k < T \leq q_{k+1}) \geq \frac{1}{\sqrt{n}}
\end{equation*}
except that the second inequality may not hold for the final interval. Let $\{r_k\}_{k=1}^{K_r}$ be uniform grid over $[0, 1]$ with spacing $(n^{-1/2})$-cover. Write $0=\tau_0<\cdots<\tau_K=1$ for the distinct points in $\{q_k\}_{k=0}^{K_q}\cup \{r_k\}_{k=1}^{K_r}$, then $K\asymp\sqrt n$. Let $\mathcal{E}$ be the event 
\begin{equation*}
    \mathcal E:=\left\{\sum_{i=1}^n\mathbbm{1}\{\tau_{k-1}<T_i<\tau_k\}\leq 2\sqrt n,\quad k\in[K]\right\}.
\end{equation*}
For every fixed $k$, the sum above is binomial with mean at most $\sqrt n$.
By a union bound argument, we have that 
\begin{equation}
    \mathbb{P}(\mathcal{E}^c) \leq \sqrt{n} \exp(-c\sqrt{n}).
    \label{eq:emp-jumps-union}
\end{equation}
Fix $k\in[K]$ and $t\in(\tau_{k-1},\tau_k]$. Since $Y_i(t)=\mathbbm{1}\{T_i\geq t\}$,
\begin{equation*}
    \{Y_i(t)\neq Y_i(\tau_k)\}=\{t\leq T_i<\tau_k\}\subseteq\{\tau_{k-1}<T_i<\tau_k\}.
\end{equation*}
Therefore, on $\mathcal{E}$, we have 
\begin{equation*}
    \frac{1}{n}\sum_{i=1}^n\mathbbm{1}\{Y_i(t)\neq Y_i(\tau_k)\}\leq \frac{2}{\sqrt n},
\end{equation*}
which implies that 
\begin{align*}
|S^{(0)}(t,\beta_0)-S^{(0)}(\tau_k,\beta_0)|
&\leq \frac1n\sum_{i=1}^n
|Y_i(t)-Y_i(\tau_k)|\exp(\beta_0^\top Z_i(t))\\
&\quad+\left|\frac1n\sum_{i=1}^nY_i(\tau_k)
\{\exp(\beta_0^\top Z_i(t))-\exp(\beta_0^\top Z_i(\tau_k))\}\right|\\
&\leq \exp(C_ZC_\beta)\left\{
\frac1n\sum_{i=1}^n\mathbbm{1}\{Y_i(t)\neq Y_i(\tau_k)\}
+\|\beta_0\|_2L_Z|t-\tau_k|\right\}
\lesssim \frac1{\sqrt n}.
\end{align*}
where $C_Z, C_\beta, L_Z$ are constants bounding the norm of the covariates, coefficients, and the Lipschitz constant for the covariate process from \Cref{assump-1}. Similarly, we have
\begin{align*}
    &\|S^{(1)}(t,\beta_0)-S^{(1)}(\tau_k,\beta_0)\|_2\\
    & \quad \leq \frac{C_Z\exp(C_Z C_\beta)}{n}\sum_{i=1}^n\mathbbm{1}\{Y_i(t)\neq Y_i(\tau_k)\}\\
    &\qquad+\frac{1}{n}\sum_{i=1}^nY_i(\tau_k)
    \left\|Z_i(t) \exp(\beta_0^\top Z_i(t))-Z_i(\tau_k) \exp(\beta_0^\top Z_i(\tau_k))\right\|_2\\
    &\leq \frac{C_Z\exp(C_Z C_\beta)}{n}\sum_{i=1}^n\mathbbm{1}\{Y_i(t)\neq Y_i(\tau_k)\} +\exp(C_Z C_\beta)(1+C_ZC_\beta)L_Z(\tau_k - t)\\
    &\lesssim \frac{1}{\sqrt n}.
\end{align*}
For any fixed time $t \in [0, 1]$, we have by \Cref{assump-1}\ref{assp:coefficient_bound} and \Cref{lemma:smooth_banach_martingale} that  
\begin{equation*}
    \mathbb{P} \left( \left\|S^{(1)}(t, \beta_0) - \mathbb{E}[S^{(1)}(t, \beta_0)] \right\|_2> \frac{\log(n)}{\sqrt{n}} \right) \lesssim  \exp \left(-c (\log n)^2 \right)
\end{equation*}
and that 
\begin{equation*}
    \mathbb{P} \left( \left|S^{(0)}(t, \beta_0) - \mathbb{E}[S^{(0)}(t, \beta_0)] \right| > \frac{\log(n)}{\sqrt{n}} \right) \lesssim  \exp \left(-c (\log n)^2 \right).
\end{equation*}
By \Cref{assump-1}\ref{assp:mod_atrisk}, we have that $\mathbb{E}[S^{(0)}(t, \beta_0)] \geq \exp(-C_ZC_\beta)p_0$. Therefore, by a union bound argument, there exist absolute constants $c_1, c_2>0$ such that
\begin{equation}
    \mathbb{P} \left( \max_{0 \leq k \leq K} \|\bar{Z}(\tau_k, \beta_0) - \mu(\tau_{k}, \beta_0)\|_2 \geq \frac{c_1\log(n)}{\sqrt{n}} \right) \lesssim  \sqrt{n} \exp(-c_2\log(n)^2).
    \label{eq:emp-timepoints}
\end{equation}
The same bounds between adjacent grid points hold for the expectations as well. Moreover, with probability at least $1-e^{-cn}$,
\[
    \inf_{t\in[0,1]}S^{(0)}(t,\beta_0)\geq e^{-C_ZC_\beta}\frac1n\sum_{i=1}^nY_i(1)\geq e^{-C_ZC_\beta}p_0/2.
\]
On this event and $\mathcal{E}$, we have that 
\[
    \sup_{t\in[0,1]}\|\bar Z(t,\beta_0)-\mu(t,\beta_0)\|_2\leq\max_{0\leq k\leq K}\|\bar Z(\tau_k,\beta_0)-\mu(\tau_k,\beta_0)\|_2+\frac{C}{\sqrt n}.
\]
Combining this with \eqref{eq:emp-jumps-union} and \eqref{eq:emp-timepoints}, we have that
\begin{align*}
    \mathbb{P} \left( \sup_{t \in [0, 1]} \|\bar{Z}(t, \beta_0) - \mu(t, \beta_0)\|_2 \geq \frac{c\log(n)}{\sqrt{n}} \right) &\lesssim \sqrt{n} \exp(-c_2 \log(n)^2) \lesssim \frac{1}{n^2}. 
\end{align*}
\end{proof}

\begin{lemma}[Hessian convergence] Under \Cref{assump-1}, for the Hessian of the log partial-likelihood evaluated at the true parameter value $\ddot{\ell}(\beta_0)$, it holds that
\begin{equation*}
    \mathbb{P} \left(\left\|\ddot{\ell}(\beta_0) +\mathbb{E} \int_0^{1} G_n(s, \beta_0)\,  \dint \Lambda_0(s) \right\|_2 > \frac{c \log(dn)}{n^{1/2}}\right) \lesssim \frac{1}{n},
\end{equation*}
where $c>0$ is an absolute constant that does not depend on $n$ or $d$, and
\begin{equation*}
G_n(s, \beta) := \frac{1}{n}\sum_{i=1}^n Y_i(s) \exp(Z_i(s)^\top\beta)[Z_i(s) - \mu(s, \beta)]^{\otimes 2} \,\, \mbox{and} \,\, \mu(s, \beta) := \frac{\mathbb{E}\{Z(s) Y(s) \exp(\beta^\top Z(s))\}}
{\mathbb{E}\{Y(s) \exp(\beta^\top Z(s))\}}.
\end{equation*}
\label{lemma:hessian_convergence}
\end{lemma}

\begin{proof}
We have, by the triangle inequality,
\begin{align}   &\left\|\ddot\ell(\beta_0)+\mathbb{E}\int_0^1G_n(s,\beta_0)\dint\Lambda_0(s)\right\|_2 \notag\\
    &\leq \left\|\frac{1}{n}\sum_{i=1}^n\int_0^1V_n(s,\beta_0)\dint M_i(s)\right\|_2 \notag\\
    &\quad+\left\|\int_0^1\{V_n(s,\beta_0)S^{(0)}(s,\beta_0)-G_n(s,\beta_0)\}\dint\Lambda_0(s)\right\|_2 \notag\\
    &\quad+\left\|\int_0^1G_n(s,\beta_0)\dint\Lambda_0(s)-\mathbb{E}\int_0^1G_n(s,\beta_0)\dint\Lambda_0(s)\right\|_2 \notag\\
    &=:(I)+(II)+(III).
    \label{eq:hessian_triangle}
\end{align}
In the remainder of the proof, we will obtain high probability bounds on each of these terms. 

\medskip
\noindent\textbf{Term (I).} Define
\begin{equation*}
    M(t)
    :=
    \frac{1}{n}\sum_{i=1}^nM_i(t)
    =
    \frac{1}{n}\sum_{i=1}^nN_i(t)
    -
    \int_0^tS^{(0)}(s,\beta_0)\,\dint\Lambda_0(s).
\end{equation*}

On the event in \Cref{lemma:risk-set-covariance-variation}, which occurs with probability at least $1-\exp(-cn)$, we have that 
\begin{align}
    (I)\leq
    \left\{
        \sup_{t\in[0,1]}\|V_n(t,\beta_0)\|_2
        + \mathcal{V} \left(V_n(\cdot,\beta_0)\right)
    \right\}
    \sup_{t\in[0,1]}|M(t)| \lesssim 
    \sup_{t\in[0,1]}|M(t)|.
    \label{eq:hessian-I-reduction}
\end{align}
By orthogonality of the individual martingales, we have 
\begin{equation*}
    \langle M\rangle(1)=\frac{1}{n}
    \int_0^1 S^{(0)}(t,\beta_0)\dint\Lambda_0(t)
    \leq \frac{e^{C_\beta C_Z}C_\lambda}{n},
\end{equation*}
so applying Freedman's inequality gives 
    \begin{equation*}
    \mathbb{P}\left(
        \sup_{t\in[0,1]}|M(t)|>x
    \right)
    \leq
    2\exp\left\{
        -\frac{x^2}
        {2\left(e^{C_\beta C_Z}C_\lambda/n+x/(3n)\right)}
    \right\}.
\end{equation*}
Taking $x=C\sqrt{\log(n)/n}$ for a sufficiently large constant $C>0$, and combining with
\Cref{lemma:risk-set-covariance-variation}, gives
\begin{equation}
    \mathbb{P}\left((I)>C\sqrt{\frac{\log n}{n}}\right)
    \lesssim\frac{1}{n}.
    \label{eq:hessian-1}
\end{equation}

\medskip
\noindent\textbf{Term (II).} The second term on the right-hand side of \eqref{eq:hessian_triangle} can be rewritten as 
\begin{align*}
& \left\|\int_0^{1} V_n(s, \beta_0)\, S^{(0)}(s, \beta_0) \dint\Lambda_0(s) - \int_0^{1} G_n(s, \beta_0) \dint\Lambda_0(s)\right\|_2 \\
&\quad=\left\|\int_0^{1} \{\bar{Z}(s, \beta_0) - \mu(s, \beta_0)\}^{\otimes 2} S^{(0)}(s, \beta_0) \dint \Lambda_0(s) \right\|_2\\
&\quad \leq \exp(C_Z C_\beta) C_\lambda \sup_{s \in [0, 1]}  \|\bar Z(s,\beta_0)-\mu(s,\beta_0)\|_2^2 \notag\\
    &\quad \leq \exp(C_Z C_\beta)C_\lambda
    \sup_{s\in[0,1]}\|\bar Z(s,\beta_0)-\mu(s,\beta_0)\|_2^2.
\end{align*}

Therefore, by \Cref{lemma:weighted_mean_convergence}, we have that 
\begin{equation}
    \mathbb{P} \left( (II) \geq \frac{C\log(n)^2}{n} \right) \lesssim \frac{1}{n}.
    \label{eq:hessian-2}
\end{equation}

\medskip
\noindent\textbf{Term (III).} We can write
\begin{equation*}
    \int_0^1G_n(s,\beta_0)\dint\Lambda_0(s)
    =\frac{1}{n}\sum_{i=1}^n\int_0^1Y_i(s)e^{\beta_0^\top Z_i(s)}[Z_i(s)-\mu(s,\beta_0)]^{\otimes2}\dint\Lambda_0(s).
\end{equation*}
The matrices in the sum are independent, symmetric, and have operator norm bounded by
\begin{equation*}
    4C_Z^2\exp(C_Z C_\beta)\Lambda_0(1)\leq 4C_Z^2\exp(C_Z C_\beta)C_\lambda.
\end{equation*}
Since $(III) = \|G_n - \mathbb{E}[G_n]\|_2$, it therefore follows from a matrix Bernstein inequality, \cite[e.g.][Theorem 6.1.1]{Tropp2015}, that 
\begin{equation}
    \mathbb{P} \left( (III) \geq \frac{c\log(dn)}{\sqrt{n}} \right) \lesssim \frac{1}{n}. 
    \label{eq:hessian-3}
\end{equation}

Hence, a union bound argument using \eqref{eq:hessian-1}, \eqref{eq:hessian-2}, and \eqref{eq:hessian-3} concludes the proof. 
\end{proof}

\begin{lemma}[Regularity of the risk-set covariance]
\label{lemma:risk-set-covariance-variation}
Under \Cref{assump-1}, there exists absolute constants $c,C>0$ such
that, with probability at least $1-e^{-cn}$,
\begin{equation}
    \inf_{t\in[0,1]}S^{(0)}(t,\beta_0)\geq c,
    \qquad
    \sup_{t\in[0,1]}\|V_n(t,\beta_0)\|_2
    +
    \mathcal V\!\left(V_n(\cdot,\beta_0)\right)
    \leq C,
    \label{eq:risk-set-covariance-regularity}
\end{equation}
where we write $\mathcal{V}(f)$ to denote the pathwise variation of a vector or matrix-valued function $f$ on $[0, 1]$. 
\end{lemma}
\begin{proof}
    Since $Y_i(t)\geq Y_i(1)$ for every $t\in[0,1]$, and $\mathbb{E} [Y_i(1)]\geq p_0$, we have by Hoeffding's inequality that
\begin{equation}
    \mathbb{P}\left\{ \inf_{t\in[0,1]}S^{(0)}(t,\beta_0)  \geq\frac{p_0}{2}e^{-C_\beta C_Z}\right\} \geq
    1 - \exp\left(-\frac{np_0^2}{2}\right).
    \label{eq:risk_set_denominator-bound}
\end{equation}
Furthermore, since $Y_i(t)$ is bounded by one and has at most one jump, by \Cref{assump-1}\ref{assp:covariate_bound} and \ref{assp:covariate_Lipschitz}, we have that 
\begin{equation}
    \sup_{t\in[0,1]} \|S^{(r)}(t,\beta_0)\|_2 + \mathcal{V} \left(S^{(r)}(\cdot,\beta_0)\right) \leq c,
    \qquad r\in\{0,1,2\},
    \label{eq:risk-set-moment-variation}
\end{equation}
which implies 
\begin{equation*}
    \mathcal V\left(\bar Z(\cdot,\beta_0)\right)
    + \mathcal V\left(\frac{S^{(2)}(\cdot,\beta_0)}{S^{(0)}(\cdot,\beta_0)}\right)  \leq C_1, 
    \qquad \mathcal {V} (\bar{Z}(\cdot, \beta_0)^{\otimes2}) \leq C_2,
\end{equation*}
where for the second variation bound we used that $\sup_t \|\bar{Z}(t, \beta_0)\|_2 \leq C_Z$. Combining with the fact that $V_n(t,\beta_0)$ is a weighted covariance matrix and
$\|Z_i(t)-\bar Z(t,\beta_0)\|_2\leq2C_Z$, so that
\begin{equation*}
    \sup_{t\in[0,1]}\|V_n(t,\beta_0)\|_2\leq4C_Z^2,
\end{equation*}
gives the second claim in the lemma. 
\end{proof}

\section{Technical details of cumulative hazard estimation}
\label{app-cumulative-hazard}
We write out the FDP Breslow estimator used in \Cref{prop:hazard_upper} in \Cref{app:FDP_breslow_alg}, then provide an estimator for the at-risk probability $\hat{p}$ that can be used within \Cref{prop:hazard_upper} in \Cref{app-hazard-atrisk}. \Cref{subsec-lambda-no-covs} contains our lower bound proofs: for the cumulative hazard in \Cref{prop:hazard_lower}, followed by \Cref{prop:cumulative-baseline-lower} for the cumulative baseline hazard. This is followed by our upper bound proofs -- the proof for \Cref{prop:hazard_upper} is contained within Appendix \ref{subsec-proof-hzard-upper}; in \Cref{subsec-proof-na-upper} we show how the result of \Cref{thm:na_upper} can be obtained by adapting those arguments. The proof of \Cref{prop-optimality-survival} is in \Cref{app-proof-prop-optimality-survival}. \Cref{app:lambda_equiv_optimality} contains additional lemmas on other loss functions and lower bounds for sequentially interactive mechanisms.   

\subsection{FDP Breslow estimator} \label{app:FDP_breslow_alg}
The FDP Breslow estimator described in \Cref{subsec:fdp_breslow} can be found in \Cref{alg:FDP-Breslow}. 
\begin{algorithm}
\caption{FDP-Breslow estimator}
\label{alg:FDP-Breslow}
\textbf{Input}: Datasets $\{(T_{s, i}, \Delta_{s, i}, \{Z_{s, i}(t):\, t \in [0, T_{s, i}]\})\}_{s, i = 1}^{S, n_s}$, privacy parameters $\{(\epsilon_s, \delta_s)\}_{s=1}^S$, regression coefficient estimate $\widehat{\beta}$, at-risk probability estimate $\hat{p}$, covariate $\ell_2$ bound $C_Z>0$, baseline hazard upper bound $C_\lambda>0$.

\begin{algorithmic}[1]
\State Set 
\begin{equation*}
    \quad c := 0.9\exp(-C_Z \|\widehat{\beta}\|_2) \hat{p}, \quad\text{and }  h:=\left \lceil \frac{1}{2} \log_2 \left(\sum_{s=1}^S \min\{n_s, n_s^2 \epsilon_s^2\}\right) \right \rceil + 1.
\end{equation*}
\For{$s = 1, \ldots, S$}
    \For{$m = 1, \ldots, 2^h$}
        \State Set 
            \begin{equation*}
                x_{s, h, m} := \int_{(m-1)/2^h}^{m/2^h} \sum_{i=1}^{n_s} \frac{\dint  N_{s, i}(t)}{n_s\max \left\{c, S^{(0)}\left(t, \widehat{\beta}; \{(T_{s, i}, \Delta_{s, i}, (Z_{s, i}(t))\}_{i=1}^{n_s} \right) \right\}}.
            \end{equation*}
        \EndFor
        \For{$l = h-1, h-2, \ldots, 1$}
            \For{$m = 1, \ldots, 2^l$}
                \State Set $x_{s, l, m} := x_{s, l+1, 2m-1} + x_{s, l+1, 2m}$.
            \EndFor

    \EndFor
    \For{$l = 1, \dots, h$}
        \For{$m = 1, \dots, 2^l$}
            \State Generate independently \begin{equation*}
                W_{s, l, m} \sim \mathcal{N}\left(0, \left[\frac{\exp(2C_Z\|\widehat{\beta}\|_2)}{c^4} + \frac{2\exp(C_Z\|\widehat{\beta}\|_2)}{c^3}+ \frac{2}{c^2}\right]\frac{2 \log(1/\delta_s)/\epsilon_s + 1}{n_s^2 \epsilon_s/h} \right).
            \end{equation*}
            \State Set $x_{s, l, m} := x_{s, l, m} + W_{s, l, m}$.
        \EndFor
    \EndFor
\EndFor
\State Set weights $\{v_s\}_{s \in [S]}$ as $v_s := \min(n_s, n_s^2 \epsilon_s^2)/\left\{\sum_{u=1}^S \min(n_u, n_u^2 \epsilon_u^2)\right\}$. 
\State For any $t \in [0, 1)$, let $(b_1, \dots b_h)$ be the binary representation of $\lfloor 2^h t \rfloor$, and set 
    \begin{equation*}
        \widehat{\Lambda}(t) :=  \max \left\{0, \sum_{s=1}^S v_s \sum_{l: b_l = 1} x_{s, l, \, \sum_{k=1}^l 2^{l-k} b_k}\right\} \wedge C_\lambda.  
    \end{equation*} 
\State Set $\widehat{\Lambda}(1) := \max\{0, \sum_{s=1}^S v_s (x_{s, 1, 1} + x_{s, 1, 2})\} \wedge C_\lambda$. 
\State \textbf{Output}: $\widehat{\Lambda}(t)_{t \in [0, 1]}$.
\end{algorithmic}
\end{algorithm}

\subsection{Estimating the at-risk probability}
\label{app-hazard-atrisk}
\Cref{alg:FDP-probabilities} can be used to estimate the at-risk probability $\hat{p}$ in order to select the level of truncation in \Cref{alg:FDP-Breslow}. Its convergence rate is provided in \Cref{lemma:at_risk_convergence}.

\begin{algorithm}
\caption{FDP-probabilities}
\label{alg:FDP-probabilities}
\textbf{Input}: Datasets $\{T_{s, i}\}_{s, i = 1}^{S, n_s}$, privacy parameters $\{\epsilon_s, \delta_s\}_{s=1}^S$.
\begin{algorithmic}[1]
\For{$s=1, \dots, S$}
\State Set
    \begin{equation*}
        \hat{p}_s = \frac{1}{n_s} \sum_{i=1}^{n_s} Y_{s,i}(1) + \frac{\sqrt{2 \log(1/\delta_s) + \epsilon_s}}{n_s \epsilon_s}W_s, \quad W_s \overset{\mathrm{i.i.d.}}{\sim} \mathcal{N}(0, 1),
    \end{equation*}
\State Set $v_s := \min\{n_s, n_s^2 \epsilon_s^2\}/\sum_{u=1}^S \min\{n_u, n_u^2 \epsilon_u^2\}$. 
\EndFor
\end{algorithmic}
\textbf{Output: } $\hat{p} := \max \left\{\sum_{s=1}^S v_s \hat{p}_s, C_p/2 \right\}$. 
\end{algorithm}

\begin{lemma}
\label{lemma:at_risk_convergence}
    Suppose that $\{T_{s, i}\}_{s, i=1}^{S, n_s}$ are i.i.d.~and $\mathbb{P}(T_{1, 1} \geq 1) = p_0 >C_p$. Let $\hat{p}$ be the output of \Cref{alg:FDP-probabilities}. Then $\hat{p}$ is $(\{\epsilon_s, \delta_s\}_{s \in [S]}, 1)$-FDP, and, for any $\gamma > 0$, 
    \begin{equation*}
\mathbb{P}\left(
\hat p>(1+\gamma)p_0
\right)
\leq
2 \exp\left(
-\frac{\gamma^2p_0^2 \sum_{s=1}^S \min\{n_s, n_s^2 \epsilon_s^2\}}
{24\max_{s\in[S]}\log(1/\delta_s)}
\right).
\end{equation*}
\end{lemma}
\begin{proof}
Since each server's sample mean  $\sum_{i=1}^{n_s} Y_{s, i}(1)/n_s$ has sensitivity $1/n_s$, the FDP guarantee follows from \Cref{lemma:gaussian_mechanism}. 

To show the probability bound, we start with the union bound
\begin{align}
\mathbb{P}\left(\hat p>(1+\gamma)p_0\right)
&\leq
\mathbb{P}\left(
\sum_{s=1}^S
\frac{v_s}{n_s}\sum_{i=1}^{n_s}Y_{s,i}(1)
>
\left(1+\frac{\gamma}{2}\right)p_0
\right)
\nonumber +
\mathbb{P}\left(
\sum_{s=1}^S \frac{v_s\sqrt{2\log(1/\delta_s)+\epsilon_s}}{n_s\epsilon_s}W_s
>\frac{\gamma p_0}{2}
\right).
\label{eq:at_risk_weighted_decomposition}
\end{align}

Since the $Y_{s,i}(1)$ are independent Bernoulli$(p_0)$ random variables,
Hoeffding's inequality gives
\begin{equation*}
\mathbb{P}\left(\sum_{s=1}^S \frac{v_s}{n_s}\sum_{i=1}^{n_s}Y_{s,i}(1)
-p_0 >\frac{\gamma p_0}{2}\right)\leq \exp\left(
-\frac{\gamma^2p_0^2}
{2\sum_{s=1}^S v_s^2/n_s}\right).
\end{equation*}
Since
\[
\frac{v_s^2}{n_s}
=\frac{\min\{n_s,n_s^2\epsilon_s^2\}^2}{(\sum_{s=1}^S \min\{n_s,n_s^2\epsilon_s^2\})^2n_s}
\leq\frac{\min\{n_s,n_s^2\epsilon_s^2\}}{(\sum_{s=1}^S \min\{n_s,n_s^2\epsilon_s^2\})^2},
\]
we have $\sum_s v_s^2/n_s\leq (\sum_{s=1}^S \min\{n_s,n_s^2\epsilon_s^2\})^{-1}$, which gives
\begin{equation*}
\mathbb{P}\left(
\sum_{s=1}^S\frac{v_s}{n_s}\sum_{i=1}^{n_s}Y_{s,i}(1) >
\left(1+\frac{\gamma}{2}\right)p_0\right) \leq \exp\left( -\frac{\gamma^2p_0^2\sum_{s=1}^S\min\{n_s, n_s^2 \epsilon_s^2\}}{2} \right).
\end{equation*}
The Gaussian noise term is centred with variance
\begin{align*}
\sigma^2&=\sum_{s=1}^S\frac{v_s^2\{2\log(1/\delta_s)+\epsilon_s\}}
{n_s^2\epsilon_s^2}\\
&\leq\frac{2\max_{s\in[S]}\log(1/\delta_s)+1}{\sum_{s=1}^S \min\{n_s, n_s^2 \epsilon_s^2\}}
\leq\frac{3\max_{s\in[S]}\log(1/\delta_s)}{\sum_{s=1}^S \min\{n_s, n_s^2 \epsilon_s^2\}}.
\end{align*}
The final inequality follows from the definition of $v_s$ and Cauchy--Schwarz, since
\begin{equation}
\sum_{s=1}^S\frac{v_s^2}{n_s^2\epsilon_s}
\leq\left(\sum_{s=1}^S\frac{v_s^2}{n_s^2} \right)^{1/2} \left(\sum_{s=1}^S\frac{v_s^2}{n_s^2\epsilon_s^2}\right)^{1/2}
\leq \frac{1}{\sum_{s=1}^S \min\{n_s, n_s^2 \epsilon_s^2\}}.
\label{eq:noise-Cauchy--Schwarz}
\end{equation}
A Gaussian tail bound gives 
\begin{equation*}
\mathbb{P}\left(
\sum_{s=1}^S \frac{v_s\sqrt{2\log(1/\delta_s)+\epsilon_s}}{n_s\epsilon_s}W_s
>
\frac{\gamma p_0}{2}
\right)
\leq
\exp\left(
-\frac{\gamma^2p_0^2 \sum_{s=1}^S \min\{n_s, n_s^2 \epsilon_s^2\}}
{24\max_{s\in[S]}\log(1/\delta_s)}
\right).
\end{equation*}

Combining the above bounds on the two probabilities leads to the result of the lemma. 
\end{proof}
\Cref{lemma:at_risk_convergence} implies that the output of \Cref{alg:FDP-probabilities} satisfies the condition for $\hat{p}$ needed in \Cref{prop:hazard_upper}, provided that
\begin{equation*}
\frac{C_p^2\sum_{s=1}^S\min\{m_s,m_s^2\epsilon_s^2\}}
{24\times18^2\times\max_{s\in[S]}\log(1/\delta_s)}
\geq \log\!\left(2\max\left\{1,\sum_{s=1}^S\min\{n_s,n_s^2\epsilon_s^2\}\right\}\right),
\end{equation*}
where $m_s$ and $n_s$ denote the independent probability-estimation and Breslow sample sizes at server $s$, respectively.

\subsection{Proof of Proposition~\ref{prop:hazard_lower}}
\label{subsec-lambda-no-covs}
\begin{proof}
 Let $P_1$ and $P_2$ be models with deterministic censoring at some time $t_0 > 1$, and suppose the event times are distributed as $\mathrm{Exp}(\lambda^*)$ and $\mathrm{Exp}(\lambda^* - h)$ respectively, where 
 $\lambda^* := \min\{C_\lambda, -\log(C_p)\}/2$, and $0<h<\lambda^*/2$ will be specified later. 
 
Since the censoring distributions are identical, we have that
\begin{equation}
    D_{\mathrm{KL}}(P_1, P_2) \leq D_\mathrm{KL}(\mathrm{Exp}(\lambda^*), \mathrm{Exp}(\lambda^*-h)) = \log(\lambda^*) - \log(\lambda^*-h) + \frac{\lambda^*-h}{\lambda^*} - 1 \lesssim h^2 
    \label{eq:exp_dist_KL}
\end{equation}
and
\begin{align}
    D_{\mathrm{TV}}(P_1, P_2) & \leq D_\mathrm{TV}(\mathrm{Exp}(\lambda^*), \mathrm{Exp}(\lambda^*-h)) = \frac{1}{2}\int_0^\infty \left| \lambda^* e^{-\lambda^* x} - (\lambda^*-h) e^{-(\lambda^*-h)x} \right| \dint x \notag \\
    &\leq \frac{\lambda^*}{2}\int_0^\infty |e^{-\lambda^* x}(1 - e^{hx})| \dint x + h \int_0^\infty e^{-(\lambda^*-h)x} \dint x \lesssim h.
    \label{eq:exp_dist_TV}
\end{align}
For $s \in [S]$, let $Q_s$ be an $(\epsilon_s, \delta_s)$-CDP mechanism acting on the $n_s$ observations held by server $s$. For $j \in \{1, 2\}$, define $R_j$ to be the joint law of the transcripts generated by $\{Q_s\}_{s=1}^S$ if the observations were i.i.d.~from $P_j$. 

By Lemma C.7 of \cite{cai2024optimal},  for any $\mathcal{S} \subseteq [S]$, it holds that  
\begin{equation}
    D_\textrm{TV}(R_1, R_2) \lesssim  \sqrt{\sum_{s \in \mathcal{S}}(n_s \epsilon_s D_\mathrm{TV}(P_1, P_2) )^2 + \sum_{s \notin \mathcal{S}} n_s D_\mathrm{KL}(P_1, P_2) } + \sum_{s \in \mathcal{S}} n_s \delta_s D_\mathrm{TV}(P_1, P_2),
    \label{eq:TV_bound_ref}
\end{equation}
provided that $n_s \epsilon_s D_\mathrm{TV}(P_1, P_2) \lesssim 1$ for all $s \in \mathcal{S}$. We can choose
\begin{equation*}
    \mathcal{S} := \{s \in [S] : \epsilon_s^2 < 1/n_s\} \quad \text{ and } \quad h^2 := c_0 \min \left\{ \lambda^{*2},  \frac{1}{\sum_{s=1}^S \min(n_s, n_s^2 \epsilon_s^2)} \right\}.
\end{equation*}
Therefore under the assumption of 
    $\sum_{s=1}^S n_s \delta_s \mathbbm{1} \left\{\epsilon_s^2 \leq 1/n_s\right\} \lesssim \sqrt{\sum_{s=1}^S \min(n_s, n_s^2 \epsilon_s^2)}$,
we have that $D_\mathrm{TV}(R_1, R_2) \leq 1/2$ with appropriately chosen scaling parameter $c_0$.

Since $\sup_{t \in [0, 1]} |\Lambda_1(t) -\Lambda_2(t)| = h$, applying Le Cam's two point method leads to the proposition:
\begin{align*}
    \inf_{Q \in \mathcal{Q}} \inf_{\widehat{\Lambda}} \sup_{P \in \mathcal{P}} \mathbb{E} \left[ \sup_{t \in [0, 1]} | \widehat{\Lambda}(t) - \Lambda_P(t)| \right] & \geq \frac{1}{4} \sup_{t \in [0, 1]} \left| \Lambda_1(t) - \Lambda_2(t) \right|  \left(1 - D_{\mathrm{TV}}(R_1, R_2) \right) \gtrsim h.
\end{align*}
\end{proof}

\begin{proposition}
\label{prop}
Let $\mathcal{P}$ be the class of distributions satisfying Assumptions~\ref{assp:baseline} and~\ref{assump-1}. Fix privacy parameters $\{\epsilon_s, \delta_s\}_{s=1}^S$ such that $\sum_{s=1}^S n_s \delta_s \mathbbm{1} \left\{\epsilon_s^2 \leq 1/n_s \right\} \lesssim \sqrt{\sum_{s=1}^S \min(n_s, n_s^2 \epsilon_s^2)}$ for all $s \in [S]$. Let $\mathcal{Q}$ be the class of $(\{(\epsilon_s, \delta_s)\}_{s \in [S]}, 1)$-FDP estimators. Then, writing $\Lambda_{0, P}$ for the cumulative baseline hazard under $P$, we have 
\begin{equation}
    \inf_{Q\in\mathcal Q}\inf_{\widehat\Lambda}\sup_{P\in\mathcal P}
    \mathbb E_{P,Q}\left[\sup_{t\in[0,1]}|\widehat\Lambda(t)-\Lambda_{0, P}(t)|\right]
    \gtrsim 1 \wedge \sqrt{\frac{d}{\sum_{s=1}^S\min\{n_s,n_s^2\epsilon_s^2\}}}.
    \label{eq:cumulative_hazard_lower}
\end{equation}

To consider interactive mechanisms, fix privacy parameters $\{(\epsilon_s,\delta_s)\}_{s=1}^S$ satisfying
$\delta_s\leq c_\delta \epsilon_s/\sqrt{S}$ for all $s\in[S]$, for some universal constant $c_\delta>0$. Let $\tilde{\mathcal{Q}}$ be the set of privacy mechanisms defined as $Q \in \tilde{\mathcal{Q}}$ if and only if there is some $K \in \mathbb{N}_+$ such that $Q$ satisfies $(\{(\epsilon_s, \delta_s)\}_{s \in [S]}, K)$-FDP and at each iteration $k \in [K]$, server $s \in [S]$ uses $n_s/K$ observations. Then we have that 
\begin{equation}
    \inf_{Q\in \tilde{\mathcal{Q}}} \inf_{\widehat\Lambda}\sup_{P\in\mathcal P}
    \mathbb E_{P,Q}\left[\sup_{t\in[0,1]}|\widehat\Lambda(t)-\Lambda_{0, P}(t)|\right]
    \gtrsim 1 \wedge \sqrt{\frac{d}{\min\left\{\sum_{s=1}^S n_s, \sum_{s=1}^S n_s^2 \epsilon_s^2 \right\}}}.
    \label{eq:cumulative_hazard_lower_int}
\end{equation}
\label{prop:cumulative-baseline-lower}
\end{proposition}

\begin{proof} The proof of this proposition follows the same general idea for proving the lower bounds for cumulative hazard estimation (see \Cref{prop:hazard_lower}), but uses a different distribution construction corresponding to a joint perturbation of the regression parameters and baseline hazard to account of the effect of covariates. These distributions are then used in Le Cam's two point method. We therefore bound the necessary divergences, and then give pointers to where they can be substituted into the arguments of \Cref{prop:hazard_lower} and \Cref{prop:hazard_lower-int} for the non-interactive FDP lower bounds and the restricted sequentially interactive FDP lower bounds respectively. 

We construct a distribution $P_0$ as follows. Take $\beta_0=0$, $\lambda_0(t)= \min\{C_\lambda,-\log C_p\}/2$ for $t\in[0,1]$, no censoring on $[0,1]$, and let the covariate distribution be $ Z(t)\equiv Z=ae_1+\frac{\sigma}{\sqrt d}\xi$, 
where $\xi = (\xi_1,\ldots,\xi_d)^{\top}$ with $\{\xi_i\}_{i \in [d]}$ being independent Rademacher random variables and $a,\sigma>0$ are sufficiently small absolute constants. Then $\|Z(t)\|_2\leq a+\sigma$, $t \in [0, 1]$, and the covariate process is Lipschitz. Since $\beta_0=0$, the event time is independent of $Z$, so we have that $\mathbb{E}[Y(t)]=\exp(-\lambda_0(0)t)$ and $\mu(t,\beta_0)=ae_1$. 
Moreover,
\[    \mathbb{E}\left[\int_0^1G(t,\beta_0) \dint\Lambda_0(t)\right] = \int_0^1 e^{-\lambda_0(0)t} \mathrm{Cov}(Z) \lambda_0(t)\dint t
    =\frac{\sigma^2(1-\exp(-\lambda_0(0))}{d}I_d,
\]
which shows that $P_0 \in \mathcal{P}$. 

We now introduce some notation to construct a perturbation of $P_0$, which we will denote by $P_h$. For $t \in [0, 1]$, let
\[
    V(t) := \int_0^t \frac{\dint \Lambda_0(u)}{s^{(0)}(u)}, \qquad s^{(0)}(u) := \mathbb{E} \left[Y(u) \exp(\beta_0^\top Z(u)) \right],
\]
and
\[
   \widetilde d(t) := Q(t)^\top I_{\beta_0}^{-1}Q(t),
    \qquad
    Q(t):=
    \int_0^t \mu(u)\,\mathrm{d}\Lambda_0(u), \quad   
    I_{\beta_0} :=\mathbb{E}_{P_0} \left[\int_0^1 G(u,\beta_0)\,\mathrm{d}\Lambda_0(u)
    \right].
\]
In particular, we have that 
\begin{align*}
   V(1)+\widetilde d(1) &=\int_0^1\frac{\mathrm d\Lambda_0(u)}{\mathbb E Y(u)} +\|Q(1)\|_2^2\frac d{\sigma^2\{1-\exp(-\lambda_0(0))\}}\\
&= \exp(\lambda_0(0))-1+ \frac{a^2\lambda_0(0)^2d}{\sigma^2\{1-\exp(-\lambda_0(0))\}}\asymp d.
\end{align*}
Let 
\[      v:=\frac{I_{\beta_0}^{-1}Q(1)}{\sqrt{V(1)+\widetilde d(1)}},
    \qquad
    b(u):=\frac{1}{s^{(0)}(u)\sqrt{V(1)+\widetilde d(1)}}+v^\top\mu(u,\beta_0).
\]
For $h>0$, define
\[
    \beta_h=\beta_0-hv, \qquad \mathrm d\Lambda_h(u)=e^{hb(u)}\dint \Lambda_0(u),
\]
and let the covariate distribution and censoring mechanisms be the same as in $P_0$; denote this distribution by $P_h$. Since $\|v\|_2,\|b\|_\infty\lesssim\sqrt d$, taking $h\sqrt d$ sufficiently small ensures $\|\beta_h\|_2\leq C_\beta$ and $h\{\|b\|_\infty+C_Z\|v\|_2\}\leq\log2$. Both the baseline and conditional hazards are then bounded by $2\lambda_0(0)\leq C_\lambda$, and $\mathbb P_h(T\geq1\mid Z)\geq\exp\{-2\lambda_0(0)\}\geq C_p$.

The perturbed risk-set weights are also bounded above and below by fixed multiples of their counterparts under $P_0$. To transfer these bounds to the population Hessian, note that, for any positive weight $w$,
\[
u^\top\mathbb E[w(Z)\{Z-\mu_w\}^{\otimes2}]u
=\inf_{c\in\mathbb R}\mathbb E[w(Z)(u^\top Z-c)^2],
\qquad \mu_w:=\frac{\mathbb E[w(Z)Z]}{\mathbb E[w(Z)]}.
\]
Integrating against the comparable baseline hazards shows that $cI_{\beta_0}\preceq\mathcal I_h\preceq CI_{\beta_0}$ for fixed constants $c,C>0$ independent of $d$. Since $I_{\beta_0}$ has eigenvalues of order $1/d$ and the covariate and censoring conditions are unchanged, $P_h\in\mathcal P$.

By a Taylor expansion when $h \sqrt{d}$ is small,
\[
    \Lambda_h(1)-\Lambda_0(1)
    =h\int_0^1b(u)\,\mathrm d\Lambda_0(u)+R_\Lambda(h),
\]
with
\[
    |R_\Lambda(h)|
    \leq \frac{h^2}{2}\int_0^1b(u)^2e^{|h||b(u)|}\,\mathrm d\Lambda_0(u)
    \lesssim h^2d.
\]
Furthermore,
\[
    \int_0^1b(u)\,\mathrm d\Lambda_0(u)
    =\frac{V(1)+Q(1)^\top I_{\beta_0}^{-1}Q(1)}
    {\sqrt{V(1)+\widetilde d(1)}}
    =\sqrt{V(1)+\widetilde d(1)}
    \asymp\sqrt d.
\]
This implies we can lower bound the separation as  $|\Lambda_h(1)-\Lambda_0(1)|\gtrsim h\sqrt d$. 

Let 
\begin{equation*}
    h \zeta(u, Z) := \log \frac{\dint \Lambda_h(u \mid Z)}{ \dint \Lambda_0(u \mid Z)} =  \frac{h}{s^{(0)}(u)\sqrt{V(1)+\widetilde d(1)}}-hv^\top\{Z-\mu(u,\beta_0)\},
\end{equation*}
then for $h$ small, we can write 
\begin{align*}
    D_{\mathrm{KL}}(P_0 \parallel P_h)
    &=\mathbb E_{P_0}\left[\int_0^1\{e^{h\zeta(u,Z)}-1-h\zeta(u,Z)\}Y(u) \dint\Lambda_0(u)\right]\\
    &\lesssim h^2 \mathbb{E}_{P_0}\left[ \int_0^1 \zeta(u, Z)^2 Y(u) \dint \Lambda_0(u) \right] \\
    &= h^2 \left( \frac{V(1)}{V(1)+\tilde{d}(1)} + v^\top I_{\beta_0} v \right) = h^2. 
\end{align*}
For $s \in [S]$, let $Q_s$ be an $(\epsilon_s, \delta_s)$-CDP mechanism acting on the $n_s$ observations held by server $s$. Define $R_0$ and $R_h$ to be the joint law of the transcripts generated by $\{Q_s\}_{s=1}^S$ if the observations were i.i.d.~from $P_0$ and $P_h$ respectively. Under the conditions of \Cref{prop:cumulative-baseline-lower}, by similar arguments to the proof of \Cref{prop:hazard_lower}, in particular using the bound in \eqref{eq:TV_bound_ref}, if we choose
\begin{equation}
    h^2  = c_0 \min \left\{\frac{1}{d},  \frac{1}{\sum_{s=1}^S \min(n_s, n_s^2 \epsilon_s^2)} \right\}
    \label{eq:h_choice}
\end{equation}
with appropriate absolute constant scaling $c_0$, we can bound $D_{\mathrm{TV}}(R_0, R_h) \leq 1/2$. The lower bound for the non-interactive mechanisms can then be seen by applying Le Cam's two-point lemma to obtain 
\begin{align*}
    \inf_{Q \in \mathcal{Q}} \inf_{\widehat{\Lambda}} \sup_{P \in \mathcal{P}} \mathbb{E}\left[\sup_{t \in [0, 1]} \big|\widehat{\Lambda}(t) - \Lambda_{0, P}(t)\big|\right] &\geq \inf_{Q \in \mathcal{Q}}  \inf_{\widehat{\Lambda}} \max_{P \in \{P_0, \, P_h\}} \mathbb{E}_P\left[\big|\widehat{\Lambda}(1) - \Lambda_{0, P}(1)\big|\right]\\
    &\gtrsim |\Lambda_0(1) - \Lambda_h(1)| \left\{1 - D_{\mathrm{TV}}(R_0, R_h)\right\} \gtrsim h\sqrt{d},
\end{align*}
where $\Lambda_{0,P_0}=\Lambda_0$ and $\Lambda_{0,P_h}=\Lambda_h$. Substituting the choice of $h$ in \eqref{eq:h_choice} into the preceding display concludes the proof of the claim in \eqref{eq:cumulative_hazard_lower}. 

To show the lower bound over the restricted sequentially interactive FDP mechanisms in \eqref{eq:cumulative_hazard_lower_int}, we can use the same construction as above, and then follow the arguments of \Cref{prop:hazard_lower-int}. Define $R_0$ and $R_h$ analogously as above to be the joint laws of the sequentially interactive transcripts. By Pinsker's inequality, the total variation distance between the single-observation distributions is $O(h)$. Choose
\begin{equation*}
h^2 = c_0 \min \left\{\frac{1}{d},
\frac{1}{\min \{\sum_{s=1}^S n_s, \sum_{s=1}^S n_s^2 \epsilon_s^2\}} \right\}.
\end{equation*}
Following the argument in the proof of \Cref{prop:hazard_lower-int}, we use the private-transcript bound when
$\sum_{s=1}^S n_s^2\epsilon_s^2 \leq \sum_{s=1}^S n_s$
and the non-private bound otherwise. Thus, for any $Q \in \tilde{\mathcal{Q}}$, this choice gives
\begin{equation*}
D_{\mathrm{TV}}(R_0, R_h)
\lesssim h \min \left\{\sum_{s=1}^S n_s,
\sum_{s=1}^S n_s^2 \epsilon_s^2\right\}^{1/2}
\leq \sqrt{c_0}.
\end{equation*}
Taking the absolute constant $c_0>0$ sufficiently small therefore ensures
$D_{\mathrm{TV}}(R_0, R_h) \leq 1/2$ and applying Le Cam's two-point lemma, we have 
\begin{align*}
    \inf_{Q \in \tilde{\mathcal{Q}}} \inf_{\widehat{\Lambda}} \sup_{P \in \mathcal{P}} \mathbb{E}\left[\sup_{t \in [0, 1]} \big|\widehat{\Lambda}(t) - \Lambda_{0, P}(t)\big|\right] &\geq \inf_{Q \in \mathcal{\tilde Q}}  \inf_{\widehat{\Lambda}} \max_{P \in \{P_0, \, P_h\}} \mathbb{E}_P\left[\big|\widehat{\Lambda}(1) - \Lambda_{0, P}(1)\big|\right]\\
    &\gtrsim |\Lambda_0(1) - \Lambda_h(1)| \left\{1 - D_{\mathrm{TV}}(R_0, R_h)\right\} \\ 
    &\gtrsim 1 \wedge \frac{d^{1/2}}{\left(\min\{\sum_{s=1}^S n_s, \sum_{s=1}^S n_s^2 \epsilon_s^2\} \right)^{1/2}}.
\end{align*} 
\end{proof}

\subsection{\texorpdfstring{Proof of \Cref{prop:hazard_upper}}{Proof of Theorem 5}} \label{subsec-proof-hzard-upper}
\begin{proof}\noindent \textbf{Privacy guarantees.} We first show that the output of \Cref{alg:FDP-Breslow} satisfies $(\{\epsilon_s, \delta_s\}_{s \in [S]}, 1)$-FDP. Each (non-privatised) level $l \in [h]$ of a binary tree from a single server's data is given by 
\begin{equation*}
    b_l(\{T_i, \Delta_i, Z_i\}_{i=1}^n ) := \left(\int_{0}^{1/2^l} \frac{\sum_{i=1}^{n}\dint  N_{i}(t)}{n \max(c, S^{(0)}(t, \widehat{\beta}))}, \ldots, \int_{(2^l-1)/2^l}^{1} \frac{\sum_{i=1}^{n}\dint  N_{i}(t)}{n \max(c, S^{(0)}(t, \widehat{\beta}))} \right). 
\end{equation*}
Without loss of generality, let $D := \{(T_i, \Delta_i, Z_i(\cdot))\}_{i \in [n]}$ and 
\[
    D' := (T_1', \Delta'_1, Z'_1(\cdot)) \cup \{(T_i, \Delta_i, Z_i(\cdot))\}_{i \in [n] \setminus \{1\}},
\]
with corresponding $S^{(0)}(t, \widehat{\beta})$ and $S'^{(0)}(t, \widehat{\beta})$. We have that 
\begin{align}
\label{eq:breslow_sensitivity}
\|b_l(D)-b_l(D')\|_2^2
&=\sum_{m=1}^{2^l}\Bigg(
\int_{(m-1)/2^l}^{m/2^l}
\left\{\frac1{n\max(c,S^{(0)}(t,\widehat\beta))}-\frac1{n\max(c,S'^{(0)}(t,\widehat\beta))}\right\}
\sum_{i=2}^n\dint N_i(t)\notag\\
&\hspace{3cm}
+\int_{(m-1)/2^l}^{m/2^l}
\left\{\frac{\dint N_1(t)}{n\max(c,S^{(0)}(t,\widehat\beta))}
-\frac{\dint N_1'(t)}{n\max(c,S'^{(0)}(t,\widehat\beta))}\right\}
\Bigg)^2\notag\\
&\leq\sum_{m=1}^{2^l}
\left(\frac{e^{C_Z \|\widehat{\beta}\|_2}\sum_{i=2}^n\mathbbm{1}\{(m-1)/2^l<T_i\leq m/2^l,\ \Delta_i=1\}}{(cn)^2}\right)^2\notag\\
&\quad+\sum_{m=1}^{2^l}
\left(
\frac{\Delta_1\mathbbm{1}\{(m-1)/2^l<T_1\leq m/2^l\}}{n\max(c,S^{(0)}(T_1,\widehat\beta))}
-\frac{\Delta_1'\mathbbm{1}\{(m-1)/2^l<T_1'\leq m/2^l\}}{n\max(c,S'^{(0)}(T_1',\widehat\beta))}
\right)^2\notag\\
&\quad+2\sum_{m=1}^{2^l}
\frac{e^{C_Z \|\widehat{\beta}\|_2}\sum_{i=2}^n\mathbbm{1}\{(m-1)/2^l<T_i\leq m/2^l,\ \Delta_i=1\}}{(cn)^2} \notag\\
&\qquad \qquad \times 
\left|
\frac{\Delta_1\mathbbm{1}\{(m-1)/2^l<T_1\leq m/2^l\}}{n\max(c,S^{(0)}(T_1,\widehat\beta))}
-\frac{\Delta_1'\mathbbm{1}\{(m-1)/2^l<T_1'\leq m/2^l\}}{n\max(c,S'^{(0)}(T_1',\widehat\beta))}
\right|\notag\\
& \leq \frac{e^{2C_Z \|\widehat{\beta}\|_2}(n-1)^2}{(cn)^4} +\frac{2}{(cn)^2} +\frac{2e^{C_Z \|\widehat{\beta}\|_2}(n-1)}{(cn)^3}
\end{align}

Since \eqref{eq:breslow_sensitivity} holds for all $D$ and $D'$ that differ in at most one entry, \eqref{eq:breslow_sensitivity} is an upper bound on $\{\mathrm{sens_2}(b_l)\}^2$ for all $l \in [h]$, where $c$ is lower bounded by a fixed constant by construction. By \Cref{lemma:Gaussian_composition}, releasing $\{(x_{s, l, m})_{m=1}^{2^l}\}_{l=1}^h$ is therefore $(\epsilon_s, \delta_s)$-CDP for $\{T_{s, i}, \Delta_{s, i}, Z_{s, i}\}_{i=1}^{n_s}$. 

\medskip
\noindent \textbf{Estimation error.}
We now consider the $\ell_\infty$ convergence of $\widehat{\Lambda}$ on  $[0, 1]$. Using that $\Lambda_0(t)$ is non-negative for all $t \in [0, 1]$ and the triangle inequality, we have 
\begin{align}
 \mathbb{E}\left[\sup_{\tau \in [0, 1]} \left| \widehat{\Lambda}(\tau) - \Lambda_0(\tau) \right| \right] &\leq \sum_{s=1}^S v_s\mathbb{E} \left[ \sup_{\tau \in [0, 1]} \left| \sum_{i=1}^{n_s} \int_0^\tau \frac{\dint N_{s, i}(t)}{n_s \max(c, S^{(0)}_s(t, \widehat{\beta}))} - \frac{\dint N_{s, i}(t)}{n_s S^{(0)}_s(t, \widehat{\beta})} \right|\right] \notag\\
 & \quad + \sum_{s=1}^S v_s \mathbb{E} \left[ \sup_{\tau \in [0, 1]} \left| \sum_{i=1}^{n_s} \int_0^\tau \frac{\dint N_{s, i}(t)}{n_s S_s^{(0)}(t, \widehat{\beta})} - \frac{\dint N_{s, i}(t)}{n_s S_s^{(0)}(t, \beta_0)}\right| \right]\notag \\
   & \quad + \mathbb{E} \left[ \sup_{\tau \in [0, 1]} \left| \sum_{s=1}^S v_s  \sum_{i=1}^{n_s} \int_0^\tau \frac{\dint N_{s, i}(t)}{n_s S^{(0)}_s(t, \beta_0)} -  \Lambda_0(\tau) \right| \right] \notag \\
 & \quad + h \,\mathbb{E} \left[\max_{l \in [h], \,m \in [2^l]} \left| \sum_{s=1}^S v_s W_{s, l,m} \right| \right] \notag \\
 & \quad+ \sup_{\tau, \tau' \in [0, 1]:\, |\tau - \tau'| \leq 1/2^h} \left| \Lambda_0(\tau) - \Lambda_0(\tau') \right| \notag \\
  &=: (I) + (II) + (III) + (IV) + (V).  
\label{pf:na_main}
  \end{align}
The terms in \eqref{pf:na_main} correspond to errors due to the truncation, $\widehat{\beta}$ estimation, process variance (non-private rate), noise added for privacy, and the discretisation respectively -- we will obtain bounds on each of these terms. 

\noindent \textbf{Term (I).} We have that
\begin{align}
    \mathbb{E} &\left[ \sup_{\tau \in [0, 1]} \left| \sum_{i=1}^{n_s} \int_0^\tau \frac{\dint N_{s, i}(t)}{n_s \max(c, S^{(0)}_s(t, \widehat{\beta}))} - \frac{\dint N_{s, i}(t)}{n_s S^{(0)}_s(t, \widehat{\beta})} \right|\right] \notag\\
    &\lesssim \log(n_s) \,\mathbb{P}\left( \frac{1}{n_s}\sum_{i=1}^{n_s} Y_{s, i}(1) \leq 0.9 \hat{p} \right) \notag\\
    &= \log(n_s) \mathbb{E}_{\hat{p}} \left[ \mathbb{P}_{Y_{s, 1}, \dots, Y_{s, n_s}}\left( \frac{1}{n_s}\sum_{i=1}^{n_s} Y_{s, i}(1) \leq 0.9 \hat{p} \,\Big\vert\, \hat{p} \right)\right] \notag\\ 
    &\leq \log(n_s) \left[\mathbb{P} \left( \hat{p} > \frac{19 p_0}{18} \right) + \mathbb{P} \left(\frac{1}{n_s}\sum_{i=1}^{n_s} Y_{s, i}(1) \leq \frac{19 p_0}{20} \right) \right] \notag\\
     &\leq \log(n_s) \left[\mathbb{P} \left( \hat{p} > \frac{19 p_0}{18} \right) + \exp\left(-\frac{p_0 n_s}{200}\right) \right] 
     \label{pf:breslow_truncation}
\end{align}
where the first inequality is from bounding the magnitude of the non-private Breslow estimator. Therefore, we obtain 
\begin{equation*}
    (I)\lesssim \mathbb{P}(\hat{p}>19p_0/18)\sum_{s=1}^S v_s(1+\log n_s) +\sum_{s=1}^S v_s(1+\log n_s)\exp(-cn_s),
\end{equation*}
where $c>0$ depends only on the fixed lower bound $C_p$. Under the assumptions of \Cref{prop:hazard_upper}, we have $\mathbb{P}(\hat{p}>19p_0/18)\lesssim N_S^{-1}$ and $\exp(-cn_s)\lesssim N_S^{-1}$ for every $s\in[S]$, provided that the constant in the sample size condition is sufficiently large. Hence, we can bound
\begin{equation*}
    (I)\lesssim \frac{\sum_{s=1}^S v_s(1+\log n_s)}{\sum_{s=1}^S \min\{n_s, n_s^2 \epsilon_s^2\}}
    \leq \frac{1+\log(\sum_{s=1}^S n_s)}{\sum_{s=1}^S \min\{n_s, n_s^2 \epsilon_s^2\}},
\end{equation*}
which is absorbed by the first term in the upper bound \eqref{eq:breslow_upper}.

\medskip
\noindent \textbf{Term (II).} For all $t \in [0, 1]$, it holds that  
\begin{align*}
    \left|\frac{1}{n_sS^{(0)}_s(t, \widehat{\beta})} - \frac{1}{n_sS^{(0)}_s(t, \beta_0)} \right| & = \frac{|\sum_{i=1}^{n_s} Y_{s, i}(t) \exp(\widehat{\beta}^\top Z_{s, i}(t))\{ \exp((\beta_0-\widehat{\beta})^\top Z_{s, i}(t))-1\}|}{n_s^2 S_s^{(0)}(t, \widehat{\beta}) \,S_s^{(0)}(t, \beta_0) } \\
    & \leq \frac{C_Z \exp(2C_ZC_\beta) \|\widehat{\beta} - \beta_0\|}{n_sS_s^{(0)}(t, \beta_0)}
\end{align*}
where the inequality follows from $|e^x-1|\leq e^{|x|}|x|$ and $\|\widehat{\beta}\|_2,\|\beta_0\|_2 \leq C_\beta$.

Since $\widehat{\beta}$ is an independent estimator, we have 
\begin{equation*}
    (II) \lesssim \mathbb{E}[\|\widehat{\beta} - \beta_0\|_2] \sum_{s=1}^S v_s \mathbb{E}\left[\int_0^1 \sum_{i=1}^{n_s} \frac{\dint N_{s, i}(t)}{n_s S_s^{(0)}(t, \beta_0)}    \right] \lesssim \mathbb{E}[\|\widehat{\beta} - \beta_0\|_2] \sum_{s=1}^S v_s \log(n_s). 
\end{equation*}
i.e.~$(II)$ converges at the same rate as $\widehat{\beta}$ (up to a log factor). 
\medskip

\noindent \textbf{Term (III).}  Define 
\begin{equation*}
    \dint M_{s, i}(t) := \dint N_{s, i}(t) - Y_{s, i}(t) \exp(\beta_0^\top Z_{s, i}(t)) \dint \Lambda_0(t), 
\end{equation*}
and for any martingale $M(t)$, let $[M](t)$ denote its quadratic variation. We have that
\begin{align*}
    (III)
    &\leq \mathbb{E}\left[\sup_{\tau\in[0,1]}
    \left|\sum_{s=1}^S v_s\sum_{i=1}^{n_s}\int_0^\tau
    \frac{\mathbbm{1}\{S_s^{(0)}(t,\beta_0)>0\}}{n_sS_s^{(0)}(t,\beta_0)}\,\dint M_{s,i}(t)\right|\right] +C_\lambda\sum_{s=1}^S v_s(1-p_0)^{n_s},
\end{align*}
since an empty risk set at any time implies $\sum_i Y_{s,i}(1)=0$. Under the sample size condition, the second term is $O(N_S^{-1})$. Applying the Burkholder--Davis--Gundy inequality to the first term, we have
\begin{align*}
    &\mathbb{E} \left\{ \sup_{\tau \in [0, 1]} \left| \sum_{s=1}^S v_s \sum_{i=1}^{n_s} \int_0^\tau \frac{\mathbbm{1}\{S_s^{(0)}(t,\beta_0)>0\}\dint M_{s, i}(t)}{n_sS^{(0)}_s(t, \beta_0)} \right| \right\}\\
    &\quad\lesssim \mathbb{E} \left\{ \left[\sum_{s=1}^S v_s \sum_{i=1}^{n_s} \int \frac{\mathbbm{1}\{\sum_{i=1}^n Y_{s, i}(t) > 0\} \dint M_{s, i}(t)}{n_sS^{(0)}_s(t, \beta_0)} \right](1) \right\}^{1/2} \\
    & \quad=  \left(\mathbb{E} \left\{ \sum_{s=1}^S v_s^2 \sum_{i=1}^{n_s} \int_0^1 \frac{\dint N_{s, i}(t)}{(n_sS^{(0)}_s(t, \beta_0))^2} \right\}\right)^{1/2}.
\end{align*} 
where the inequality is due to the Burkholder--Davis--Gundy inequality \citep[e.g. Theorem 1.1 of][]{marinelli2013BDG} and Jensen's inequality, and the equality is due to $[M_{s, i}] = N_{s, i}$ and that the integrands are predictable. We may further bound
\begin{equation*}
    \mathbb{E} \left\{ \sum_{i=1}^{n_s}  \int_0^1 \frac{\dint N_{s, i}(t)}{(n_sS^{(0)}_s(t, \beta_0))^2}\right\} \lesssim \frac{1}{n_s} + \log(n_s) \mathbb{P}\left( \frac{1}{n_s}\sum_{i=1}^{n_s} Y_{s, i}(1) \leq p_0/2\right) \
\end{equation*}
by considering the events of whether $(n_s)^{-1}\sum_{i=1}^{n_s} Y_{s, i}(1) > p_0/2$ or not. Following the arguments leading to \eqref{pf:breslow_truncation}, we therefore have  
\begin{equation*}
    (III) \lesssim \left\{\sum_{s=1}^S v_s^2 \left(\frac{1}{n_s} + \log(n_s)\exp(-cn_s) \right) \right\}^{1/2} \lesssim \frac{1}{(\sum_{s=1}^S \min(n_s, n_s^2 \epsilon_s^2))^{1/2}}. 
\end{equation*}

\noindent \textbf{Term (IV).}  For term $(IV)$, write $V_{l,m}:=\sum_{s=1}^S v_sW_{s,l,m}$. Conditional on $(\widehat\beta,\hat p)$, these variables are centred Gaussian with common variance
\begin{equation*}
    \sigma^2:=\mathrm{Var}(V_{1,1}\mid\widehat\beta,\hat p)
    \lesssim h\sum_{s=1}^S
    \frac{v_s^2\{2\log(1/\delta_s)+\epsilon_s\}}{n_s^2\epsilon_s^2}
    \lesssim \frac{h\{1+\max_s\log(1/\delta_s)\}}{\sum_{s=1}^S \min\{n_s, n_s^2 \epsilon_s^2\}}.
\end{equation*}
where the final inequality follows from \eqref{eq:noise-Cauchy--Schwarz}. 
Since there are fewer than $2^{h+1}$ nodes, a Gaussian tail bound and a union bound give
\begin{equation*}
    \mathbb{P}\left(\max_{l,m}|V_{l,m}|^2>t
    \,\middle|\,\widehat\beta,\hat p\right)
    \leq\min\{1,2^{h+2} \exp(-t/(2\sigma^2))\}.
\end{equation*}
Integrating this bound and taking expectations over the auxiliary inputs yields
\begin{equation*}
    \mathbb{E}\left[\max_{l,m}|V_{l,m}|^2\right]
    \leq 2\{1+\log(2^{h+2})\}\mathbb{E}[\sigma^2]
    \lesssim \frac{h^2\{1+\max_s\log(1/\delta_s)\}}{\sum_{s=1}^S \min\{n_s, n_s^2 \epsilon_s^2\}}.
\end{equation*}
Every reconstruction sums at most $h$ node's noises, so by Cauchy--Schwarz, we have 
\begin{align*}
    (IV)
    &\leq h\left\{\mathbb{E}\left[\max_{l,m}|V_{l,m}|^2\right]\right\}^{1/2}\\
    &\lesssim h^2\left\{\frac{1+\max_s\log(1/\delta_s)}{\sum_{s=1}^S \min\{n_s, n_s^2 \epsilon_s^2\}}\right\}^{1/2}.
\end{align*}

 \medskip
\noindent \textbf{Term (V).}  Under the assumption of $\lambda_0(t) < C_\lambda$, we can bound $(V) \leq C_\lambda /2^h$. 
\end{proof}

\subsection{Proof of Theorem 1}
\label{subsec-proof-na-upper}
\begin{proof}
We adapt the proof for the Breslow estimator in \Cref{prop:hazard_upper} by setting $C_Z=0$ and $\widehat\beta=\beta_0=0$, and using the local probability estimates from \Cref{alg:FDP-NA}. 

Let $\mathcal{A}$ be the event that no truncation occurs at any server. By \Cref{lemma:local_prob_estimate} and a union bound,
\[
\mathbb P(\mathcal A^c)
\leq2\sum_{s=1}^S\exp\left[-cC_p^2\min\left\{n_s,\frac{n_s^2\epsilon_s^2}{\log(1/\delta_s)}\right\}\right]
\leq \frac{1}{\sqrt{\sum_{s=1}^S \min\{n_s, n_s^2 \epsilon_s^2\}}},
\]
where the last inequality follows from the stated sample-size condition with a sufficiently large implied constant. Hence the truncation error satisfies
\[
(I)\lesssim\frac{1 + \log(\sum_{s=1}^S n_s)}{\sqrt{\sum_{s=1}^S \min\{n_s, n_s^2 \epsilon_s^2\}}}.
\]
The remaining terms can be bounded as in the proof of \Cref{prop:hazard_upper}; in particular the privacy noise bound holds since $c_s\geq C_p/2$.
\end{proof}

\begin{lemma}
 Let $n\geq10$, $\epsilon>0$ and $0<\delta<e^{-1}$. Let $\{(T_i,\Delta_i)\}_{i=1}^n$ be i.i.d. from a distribution satisfying \Cref{assp:baseline}, and write $p_0:=\mathbb P(T\geq1)$. Define an estimator of the at-risk probability by
    \[
\hat{p} = \frac{1}{\lfloor n/10 \rfloor} \sum_{i=1}^{\lfloor n/10 \rfloor} Y_{i}(1) + \frac{\sqrt{2 \log(1/\delta) + \epsilon}}{\lfloor n/10 \rfloor \epsilon}W, \quad W \sim \mathcal{N}(0,1).
\]

We then have that 
\begin{equation*}
    \mathbb{P} \left( \frac{1}{n-\lfloor n/10 \rfloor } \sum_{i=\lfloor n/10\rfloor + 1}^n Y_i(1) \leq \max \left\{\frac{C_p}{2}, 0.9 \hat{p} \right\} \right) \leq \exp \left(  -cp_0^2 \min \left\{n, \frac{n^2\epsilon^2}{\log(1/\delta)} \right\}\right)
\end{equation*}
Moreover, if $M$ is a mechanism such that for any fixed $p$, $M(\{(T_i, \Delta_i)\}_{i=\lfloor n/10\rfloor + 1}^n; p)$ is $(\epsilon, \delta)$-DP, then $M(\{(T_i, \Delta_i)\}_{i=\lfloor n/10\rfloor + 1}^n; \hat{p})$ is $(\epsilon, \delta)$-DP with respect to the whole dataset $\{(T_i, \Delta_i)\}_{i=1}^n$. 
\label{lemma:local_prob_estimate}
\end{lemma}

\begin{proof}
Write
\begin{equation*}
    A := \frac{1}{\lfloor n/10 \rfloor }\sum_{i=1}^{\lfloor n/10 \rfloor} Y_{i}(1), \qquad B:= \frac{1}{n-\lfloor n/10 \rfloor } \sum_{i=\lfloor n/10\rfloor + 1}^n Y_i(1), \qquad \sigma := \frac{\sqrt{2 \log(1/\delta) + \epsilon}}{\lfloor n/10 \rfloor \epsilon}.
\end{equation*}
We will first consider bounding 
    \begin{align}
        \mathbb{P} \left( \frac{1}{n-\lfloor n/10 \rfloor } \sum_{i=\lfloor n/10\rfloor + 1}^n Y_i(1) \leq 0.9 \hat{p} \right) &= \mathbb{P} \left( (B - p_0) - 0.9(A-p_0) - 0.9\sigma W \leq -0.1 p_0  \right) \notag\\
        &\leq \exp \left( \frac{-(0.1p_0)^2}{2v} \right)
        \label{eq:truncation-estimator}
    \end{align}
where we use $v$ to denote the sub-Gaussian variance proxy of $(B - p_0) - 0.9(A-p_0) - 0.9\sigma W$. Since $A$ and $B$ are the average of Bernoulli random variables, we have 
\begin{equation*}
    v \leq \frac{1}{4(n-\lfloor n/10 \rfloor)} + \frac{0.81}{4 \lfloor n/10 \rfloor} + 0.81 \sigma^2.
\end{equation*}
For $n \geq 10$, this implies that
\[
v\lesssim\frac{1}{n}+ \frac{2\log(1/\delta)+\epsilon}{n^2\epsilon^2} \lesssim\frac1n+\frac{\log(1/\delta)}{n^2\epsilon^2}.
\]
To see the final inequality, if $\epsilon\leq1$ the extra term can be absorbed into the privacy variance, and otherwise it can be absorbed into the $1/n$ term. 
Since $C_p\leq p_0$, Hoeffding's inequality gives
\begin{equation*}
    \mathbb{P}(B\leq C_p/2) \leq\mathbb{P}(B\leq p_0/2) \leq\exp\{-p_0^2(n-\lfloor n/10\rfloor)/2\},
\end{equation*}
so a union bound argument with \eqref{eq:truncation-estimator} proves the probability bound in the lemma. 
    
Since each $Y_i(1) \in \{0, 1\}$, the sensitivity of $A$ is $(\lfloor n/10 \rfloor)^{-1}$. Therefore, by \Cref{lemma:gaussian_mechanism}, $\hat{p}$ is $(\epsilon, \delta)$-DP for $\{(T_i, \Delta_i)\}_{i=1}^{\lfloor n/10 \rfloor}$. The privacy guarantee for $\{(T_i, \Delta_i)\}_{i=1}^n$ follows from the use of disjoint data and the post-processing property of DP.
\end{proof}

\subsection{Proof of Theorem~\ref{prop-optimality-survival}} \label{app-proof-prop-optimality-survival}

\begin{proof}[Proof of \Cref{prop-optimality-survival}] 
\medskip
\textbf{Lower bound:} we first show a lower bound over estimators with a plug-in form. It suffices to show that, for any $\beta_1 \neq \beta_2$, we have
\[
    \sup_{z: \|z\|_2 \leq C_Z} |S_{\beta_1}(1; z) - S_{\beta_2}(1; z)| \gtrsim \|\beta_1 - \beta_2\|_2.
\]

Construct $\Lambda_0(t) = \lambda t$ and define the function $f(x) = \exp(-\lambda e^x)$.  We then have that $S_{\beta} (1; z)= f(\beta^{\top} z)$ and the derivative of $f$ is given by 
\[
    f'(x) = \exp(-\lambda e^x) (-\lambda e^x).
\]
Since $\|\beta\|_2 \leq C_{\beta}$ and $\|z\|_2 \leq C_Z$, we consider the range $x \in [-C_{\beta}C_Z, C_{\beta} C_Z]$. For $x$ in this range we can bound 
\[
    |f'(x)| \geq \lambda e^{-C_{\beta}C_Z} \exp(- \lambda e^{C_{\beta}C_Z})>0.
\]

It follows from the mean value theorem that there exists $\tilde{x} \in [\beta_1^\top z, \beta_2^\top z]$ such that
\[
    f(\beta_1^{\top}z) - f(\beta_2^{\top}z) = f'(\tilde{x}) (\beta_1 - \beta_2)^{\top} z.
\]
Taking $z^* = C_Z(\beta_1 - \beta_2)/\|\beta_1 - \beta_2\|_2$ leads to
\begin{align}
   \sup_{z: \|z\|_2\leq C_Z} |S_{\beta_1}(1; z) - S_{\beta_2}(1; z)| 
   &\geq |S_{\beta_1}(1; z^*) - S_{\beta_2}(1; z^*)| \notag\\
   &\geq \lambda C_Z \exp(-C_{\beta} C_Z) \exp(-\lambda e^{C_{\beta}C_Z}) \|\beta_1 - \beta_2\|_2.
   \label{eq:survival-to-beta}
\end{align}

Now let $\widehat{S}$ be an arbitrary survival function estimator, and fix some $a>0$. Let $\widetilde{\beta}_a$ satisfy 
\begin{equation*}
    \sup_{\|z\|_2\leq C_Z}|\widehat S(1;z)-S_{\widetilde{\beta}_a}(1;z)|
    \leq \inf_{\beta\in B_{C_\beta}(0)}
    \sup_{\|z\|_2\leq C_Z}|\widehat S(1;z)-S_{\beta}(1;z)|+a.
\end{equation*}
Since $\beta_0\in B_{C_\beta}(0)$, the triangle inequality gives
\begin{align*}
    \sup_{\|z\|_2\leq C_Z}
    |S_{\widetilde\beta_a}(1;z)-S_{\beta_0}(1;z)|
    &\leq \sup_{\|z\|_2\leq C_Z}
    |S_{\widetilde\beta_a}(1;z)-\widehat S(1;z)|
    +\sup_{\|z\|_2\leq C_Z}
    |\widehat S(1;z)-S_{\beta_0}(1;z)|\\
    &\leq 2\sup_{\|z\|_2\leq C_Z}
    |\widehat S(1;z)-S_{\beta_0}(1;z)|+a,
\end{align*}
so \eqref{eq:survival-to-beta} implies that 
\begin{equation*}
    \|\widetilde\beta_a-\beta_0\|_2
    \lesssim 2\sup_{\|z\|_2\leq C_Z}
    |\widehat S(1;z)-S_{\beta_0}(1;z)|+a.
\end{equation*}
Note that $\widetilde{\beta}_a$ is a post-processing of $\widehat{S}$, and therefore satisfies the same DP constraints. Squaring this inequality and taking $a \downarrow 0$, we can thus transfer the lower bound to the coefficient estimation lower bound in \Cref{prop:FDP_lower_bound}. 

\medskip
\textbf{Upper bound:} fix any covariate path $z \in \mathcal{Z}$ that satisfies the boundedness and Lipschitz conditions from \Cref{assump-1}\ref{assp:covariate_bound} and \ref{assp:covariate_Lipschitz}. We first bound the estimated integrated hazard from below. Since $0\leq\widehat\Lambda_B\leq C_\lambda$ and $\widehat\Lambda_B(0)=0$, integration by parts gives
\[
    \int_0^t e^{\widehat\beta^\top z(s)}\dint\widehat\Lambda_B(s)
    =e^{\widehat\beta^\top z(t)}\widehat\Lambda_B(t)-\int_0^t\widehat\Lambda_B(s)\partial_s (e^{\widehat\beta^\top z(s)})\dint s
    \geq-C_\lambda C_\beta L_Ze^{C_ZC_\beta}=:-C_S.
\]
Both integrated hazards therefore lie in $[-C_S,\infty)$, where $x\mapsto e^{-x}$ is $e^{C_S}$-Lipschitz. Applying the mean value theorem, we have that  
\begin{align*}
    |\widehat{S}(t; z) - S(t; z)| &\leq e^{C_S} \left| \int_0^t  \exp(\widehat{\beta}^\top z(s))  \dint \widehat{\Lambda}_B(s) -  \int_0^t \exp(\beta_0^\top z(s)) \dint \Lambda_0(s)\right|\\
    &\leq e^{C_S} \left| \int_0^t  \exp(\widehat{\beta}^\top z(s))  \dint \widehat{\Lambda}_B(s) -  \int_0^t \exp(\widehat{\beta}^\top z(s)) \dint \Lambda_0(s)\right|\\ 
    &\quad+e^{C_S} \left| \int_0^t  \left\{\exp(\widehat{\beta}^\top z(s)) -  \exp(\beta_0^\top z(s)) \right\} \dint \Lambda_0(s)\right|\\
   & =: e^{C_S} \left\{(I) + (II) \right\}. 
\end{align*}

Using integration by parts, we can bound 
\begin{align*}
    (I) &\leq \left| \exp\{\widehat{\beta}^\top z(t)\}\{\widehat{\Lambda}_B(t) - \Lambda_0(t)\} \right| + \left| \int_0^t \partial_s \left(\exp\{\widehat{\beta}^\top z(s)\} \right)\{\widehat{\Lambda}_B(s) - \Lambda_0(s)\} \dint s \right|\\
    &\leq \exp(C_Z C_\beta) \{1+ C_\beta L_Z\} \sup_{t \in [0, 1]} \left| \widehat{\Lambda}_B(t) - \Lambda_0(t)\right|.
\end{align*}
We also have the bound 
\begin{align*}
    (II) &\leq \int_0^t  \left|\exp(\widehat{\beta}^\top z(s)) \left[ 1 - \exp \left\{(\beta_0 - \widehat{\beta})^\top z(s) \right\}\right] \dint \Lambda_0(s)\right|\\
    &\leq C_\lambda  \exp(C_{\beta} C_Z) C_Z \|\beta_0 - \hat{\beta}\|_2.
\end{align*}
Using these bounds and taking supremum over covariate paths and $t \in [0, 1]$, we therefore have that 
\begin{align*}
    \mathbb{E}\left[ \sup_{t \in [0, 1]} \sup_{z \in \mathcal{Z}} |\widehat{S}(t; z) - S(t; z)|^2\right] \lesssim \mathbb{E}\left[ \sup_{t \in [0, 1]}\left| \widehat{\Lambda}_B(t) - \Lambda_0(t)\right|^2\right] + \mathbb{E}[\|\widehat{\beta}-\beta_0\|_2^2].
\end{align*}
The second term can be bound using \Cref{prop:fdp_upper}, so it remains to bound the first term. The proof proceeds very similarly to the proof of \Cref{prop:hazard_upper}. By the Cauchy--Schwarz inequality, we may bound  
    \begin{align}
 \mathbb{E}\left[\sup_{\tau \in [0, 1]} \left| \widehat{\Lambda}(\tau) - \Lambda_0(\tau) \right|^2 \right] &\leq 5\,\mathbb{E} \left[ \sup_{\tau \in [0, 1]} \left| \sum_{s=1}^S v_s\sum_{i=1}^{n_s} \int_0^\tau \frac{\dint N_{s, i}(t)}{n_s \max(c, S^{(0)}_s(t, \widehat{\beta}))} - \frac{\dint N_{s, i}(t)}{n_s S^{(0)}_s(t, \widehat{\beta})} \right|^2\right] \notag\\
 & \quad + 5\,\mathbb{E} \left[ \sup_{\tau \in [0, 1]} \left|\sum_{s=1}^S v_s \sum_{i=1}^{n_s} \int_0^\tau \frac{\dint N_{s, i}(t)}{n_s S_s^{(0)}(t, \widehat{\beta})} - \frac{\dint N_{s, i}(t)}{n_s S_s^{(0)}(t, \beta_0)}\right|^2 \right]\notag \\
   & \quad + 5\,\mathbb{E} \left[ \sup_{\tau \in [0, 1]} \left| \sum_{s=1}^S v_s  \sum_{i=1}^{n_s} \int_0^\tau \frac{\dint N_{s, i}(t)}{n_s S^{(0)}_s(t, \beta_0)} -  \Lambda_0(\tau) \right|^2 \right] \notag \\
 & \quad + 5h^2\,\mathbb{E}\left[\max_{1\leq l\leq h}\max_{1\leq m\leq 2^l}
\left|\sum_{s=1}^S v_s W_{s,l,m}\right|^2\right] \notag\\
 & \quad+ \sup_{\tau, \tau' \in [0, 1]:\, |\tau - \tau'| \leq 1/2^h} 5\,\left| \Lambda_0(\tau) - \Lambda_0(\tau') \right|^2 \notag \\
  &=: 5 \left\{(A) + (B) + (C) + (D) + (E) \right\}. 
\label{pf:na_main_sq}
  \end{align}

We start with bounding $(A)$. Define $\mathcal{E}:=\{\hat p>19p_0/18\}$ and $\mathcal{F}_s:=\{n_s^{-1}\sum_{i=1}^{n_s}Y_{s,i}(1)\leq19p_0/20\}$. A truncation at server $s$ can occur only on $\mathcal{E}\cup \mathcal{F}_s$. We can bound
\begin{align*}
    (A)
    &\lesssim \mathbb{E}\left[\left\{
        \mathbbm{1}\{\mathcal{E}\}\sum_{s=1}^S v_s(1+\log n_s)
        +\sum_{s=1}^S v_s(1+\log n_s)\mathbbm{1}\{\mathcal{F}_s\}
    \right\}^2\right]\\
    &\lesssim \mathbb{P}(\mathcal{E})\left(\sum_{s=1}^S v_s(1+\log n_s)\right)^2
    +\sum_{s=1}^S v_s^2(1+\log n_s)^2\exp(-cn_s)\\
    &\quad+\left(\sum_{s=1}^S v_s(1+\log n_s)\exp(-cn_s)\right)^2,
\end{align*}
where we used that $\mathbb{P}(\mathcal{F}_s)\leq\exp(-cn_s)$. Under the assumptions of \Cref{prop:hazard_upper}, $\mathbb{P}(\mathcal{E})\lesssim N_S^{-1}$ and $\exp(-cn_s)\lesssim N_S^{-1}$ for every $s\in[S]$, with a sufficiently large constant in the sample size condition. Since $\sum_s v_s=1$, it follows that
\begin{equation*}
    (A)\lesssim \frac{\{1+\log(\sum_{s=1}^S n_s)\}^2}{\sum_{s=1}^S \min\{n_s, n_s^2 \epsilon_s^2\}}.
\end{equation*}

For Term $(B)$, we have by Cauchy--Schwarz that 
\begin{align*}
    (B) &\leq \mathbb{E} \left[ \sup_{\tau \in [0, 1]} \left( \sum_{s=1}^S v_s\right) \sum_{s=1}^S v_s \left(\sum_{i=1}^{n_s} \int_0^\tau \frac{\dint N_{s, i}(t)}{n_s S_s^{(0)}(t, \widehat{\beta})} - \frac{\dint N_{s, i}(t)}{n_s S_s^{(0)}(t, \beta_0)}  \right)^2 \right]\\
    &\leq \sum_{s=1}^S v_s \mathbb{E}\left[ \sup_{\tau \in [0, 1]}\left(\sum_{i=1}^{n_s} N_{s, i}(\tau) \right) \sum_{i=1}^{n_s} \int_0^\tau \left(\frac{1}{n_s S_s^{(0)}(t, \widehat{\beta})} - \frac{1}{n_s S_s^{(0)}(t, \beta_0)} \right)^2 \dint N_{s, i}(t) \right] \\
    &\leq \sum_{s=1}^S v_s \mathbb{E}\left[n_s \sum_{i=1}^{n_s} \int_0^1 \left(\frac{C_Z \exp(2C_Z C_\beta) \|\widehat{\beta} - \beta_0\|}{n_sS_s^{(0)}(t, \beta_0)}\right)^2 \dint N_{s, i}(t) \right]\\
    &=: \mathbb{E}[\|\widehat{\beta} - \beta_0\|^2]  \sum_{s=1}^S v_s \mathbb{E}\left[A_s\right] 
\end{align*}
Defining the event $\mathcal{E}_s := \{S_s^{(0)}(1, \beta_0) \geq \exp(-C_ZC_\beta) p_0/2\}$, we have that 
\begin{align*}
    \mathbb{E}[A_s]&=\mathbb{E}[A_s | \mathcal{E}_s] \mathbb{P}(\mathcal{E}_s) + \mathbb{E}[A_s \mid \mathcal{E}_s^c] \mathbb{P}(\mathcal{E}^c_s)\\
    &\lesssim \log(n_s) + n_s \log(n_s) \exp(-cn_s). 
\end{align*}

For Term $(C)$ of \eqref{pf:na_main_sq}, we can use similar arguments leading to the bound for term $(III)$ in \eqref{pf:na_main} to obtain 
\begin{align*}
    (C)
    &\lesssim \mathbb{E}\left[\sup_{\tau\in[0,1]}
    \left|\sum_{s=1}^S v_s\sum_{i=1}^{n_s}\int_0^\tau
    \frac{\mathbbm{1}\{S_s^{(0)}(t,\beta_0)>0\}}{n_sS_s^{(0)}(t,\beta_0)}\,\dint M_{s,i}(t)\right|^2\right]\\
    &\quad+C_\lambda^2\left\{
    \sum_{s=1}^S v_s^2(1-p_0)^{n_s}
    +\left(\sum_{s=1}^S v_s(1-p_0)^{n_s}\right)^2\right\}.
\end{align*}
The second term is $O(N_S^{-1})$ by the sample size condition. The martingale term can also be bound to be $O(N_S^{-1})$ by the Burkholder--Davis--Gundy inequality. 

For term $(D)$, every reconstruction sums at most $h$ node noises. Using the Gaussian maximum bound established for term $(IV)$ in the proof of \Cref{prop:hazard_upper}, we obtain
\begin{equation*}
    (D)\leq h^2\mathbb{E}\left[\max_{l,m}|V_{l,m}|^2\right]
    \lesssim \frac{h^4\{1+\max_s\log(1/\delta_s)\}}{\sum_{s=1}^S \min\{n_s, n_s^2 \epsilon_s^2\}}
\end{equation*}
where $V_{l,m}:=\sum_{s=1}^S v_sW_{s,l,m}$.

For the final term in \eqref{pf:na_main_sq}, we use that $\lambda_0(t) < C_\lambda$ to bound 
\begin{equation*}
(E) \leq C_\lambda^2/2^{2h} \lesssim \frac{1}{\sum_{s=1}^S \min(n_s, n_s^2\epsilon_s^2)}.
\end{equation*}

Summing the bounds for $(A),\dots,(E)$ in \eqref{pf:na_main_sq}, we obtain
\begin{equation*}
    \mathbb{E}\left[\sup_{\tau\in[0,1]}
    |\widehat{\Lambda}_B(\tau)-\Lambda_0(\tau)|^2\right]
    \lesssim_{\log} \mathbb{E}[\|\widehat{\beta}-\beta_0\|_2^2]+ \frac{1}{\sum_{s=1}^S \min\{n_s, n_s^2 \epsilon_s^2\}}.
\end{equation*}
\end{proof}

\subsection{Interactive mechanisms}
\label{app:lambda_equiv_optimality}
While \Cref{prop:hazard_lower} considered non-interactive mechanisms $(K=1)$, in \Cref{prop:hazard_lower-int}, we show a lower bound of $\left(\min \left\{\sum_{s=1}^S n_s^2 \epsilon_s^2, \sum_{s=1}^S n_s \right\} \right)^{-1/2}$, for $(\{(\epsilon_s, \delta_s)\}_{s\in [S]}, K)$-FDP mechanisms where server $s \in [S]$ uses $n_s/K$ observations at each iteration, over all choices of $K \in \mathbb{N}_+$. For the homogeneous setting, this would be the same rate as in \Cref{prop:hazard_lower}.

\begin{lemma}
\label{prop:hazard_lower-int}
Fix privacy parameters $\{(\epsilon_s,\delta_s)\}_{s=1}^S$ satisfying
$\delta_s\leq c_\delta \epsilon_s/\sqrt{S},  s\in[S]$ for some universal constant $c_\delta>0$. Let $\mathcal Q$ be the class of mechanisms such that $Q \in \mathcal{Q}$ iff there exists $K\in\mathbb N_+$ such that $Q$ satisfies $(\{(\epsilon_s,\delta_s)\}_{s=1}^S,K)$-FDP, $|D_s^{(k)}|=n_s/K$, the partitions $\{D_s^{(k)}\}_{k=1}^K$ are fixed and data-independent, and, conditional on $M^{(k-1)}$, the server outputs use independent local randomisation and depend on the data only through their respective $D_s^{(k)}$. Let $\mathcal{P}$ be the set of distributions satisfying \Cref{assp:indep} and \Cref{assp:baseline}. Then 
\begin{equation*}
    \inf_{Q\in\mathcal Q}\inf_{\widehat\Lambda}\sup_{P\in\mathcal P}
    \mathbb E_P\left[ \sup_{t\in[0,1]}|\widehat\Lambda(t)-\Lambda_P(t)|\right]
    \gtrsim 1 \wedge \frac{1}{\sqrt{\min\left\{\sum_{s=1}^S n_s,\,\sum_{s=1}^S n_s^2\epsilon_s^2 \right\}}},
\end{equation*}
where we write $\Lambda_P$ for the cumulative hazard of the event time distribution under model $P$
\end{lemma}

\begin{proof}
Let $P_1$ and $P_2$ be models with deterministic censoring at some time $t_0 > 1$, and suppose the event times are distributed as $\mathrm{Exp}(\lambda^*)$ and $\mathrm{Exp}(\lambda^* - h)$ respectively, for $\lambda^*$ as defined in \Cref{prop:hazard_lower} and some $0 < h < \lambda^*$ to be specified. We have that $\|\Lambda_1-\Lambda_2\|_\infty=h$ and by \eqref{eq:exp_dist_KL} and \eqref{eq:exp_dist_TV} that $D_{\mathrm{KL}}(P_1,P_2) \lesssim h^2$ and $D_{\mathrm{TV}}(P_1, P_2) \lesssim h$. 

Fix any $Q \in \mathcal{Q}$, and for $j \in \{1, 2\}$, let $P_j^{M^{(K)}}$ denote the push-forward of the collection of private transcripts when the data is generated from i.i.d.~draws from $P_j$, and similarly $P_j^{R^{(k)} \mid M^{(k-1)}}$ for the $k$-th round mechanism conditional on the past transcripts $M^{(k-1)}$. 

Since the censoring distributions are identical, by the data processing inequality and Pinsker's inequality, we have that
\begin{equation}
    D_{\mathrm{TV}}(P_1^{M^{(K)}},P_2^{M^{(K)}})
    \leq \left\{\frac{1}{2}\left(\sum_{s=1}^S n_s\right)
        D_{\mathrm{KL}}(P_1,P_2)\right\}^{1/2}.
    \label{eq:nonprivate-expdist}
\end{equation}

To control the total variation distance between the private transcripts' distributions, we set up the notation 
\begin{equation*}
    R^{(k)}:=(R_1^{(k)},\ldots,R_S^{(k)}),\qquad
    M^{(k)}:=(M^{(k-1)},R^{(k)}), \qquad 
    H_{s-1}^{(k)}
    :=
    (M^{(k-1)},R_1^{(k)},\ldots,R_{s-1}^{(k)}).
\end{equation*}
The chain rule for total variation gives
\begin{align}
    D_{\mathrm{TV}}(P_1^{M^{(k)}},P_2^{M^{(k)}})
    &\leq
    D_{\mathrm{TV}}(P_1^{M^{(k-1)}},P_2^{M^{(k-1)}}) \notag\\
    &\quad+
    \mathbb E_{P_1}\left[
        D_{\mathrm{TV}}
        \left(
            P_1^{R^{(k)}\mid M^{(k-1)}},
            P_2^{R^{(k)}\mid M^{(k-1)}}
        \right)
    \right].
    \label{eq:TV_chainrule}
\end{align}

For each $k\in[K]$, $s\in[S]$ and fixed history $H_{s-1}^{(k)}=u$, the batch
$D_s^{(k)}$ remains distributed as the $n_s/K$-fold product distribution $P_j^{\otimes n_s/K}$ under $P_j$ due to the independence assumptions,
and the conditional mechanism $D_s^{(k)}
    \mapsto
    Q_s^{(k)}(\cdot \mid  u,D_s^{(k)})$
is $(\epsilon_s,\delta_s)$-differentially private. We therefore have that 
\begin{equation}
    D_\infty^{\delta_s'}
    \left(
        P_j^{R_s^{(k)}\mid H_{s-1}^{(k)}=u},
        P_{j'}^{R_s^{(k)}\mid H_{s-1}^{(k)}=u}
    \right)
    \leq\epsilon_s', \qquad j,j' \in \{1, 2\}
    \label{eq:Vadhan_generalisation}
\end{equation}
where
\begin{equation*}
    \epsilon_s':=6n_s\epsilon_s D_{\mathrm{TV}}(P_1, P_2)/K,
    \qquad
    \delta_s':= 6e^{\epsilon_s'}n_s\delta_s D_{\mathrm{TV}}(P_1, P_2)/K.
\end{equation*}
The above claim can be shown by applying Lemma 6.1 of \cite{karwa2017finitesampledifferentiallyprivate}, applied conditionally on $M^{(k-1)}=m$ with sample size $n_s/K$, which would give that for every measurable $E$,
\begin{equation*}
    \mathbb P_1(R_s^{(k)}\in E\mid M^{(k-1)}=m)
    \leq
    e^{\epsilon_s'}
    \mathbb P_2(R_s^{(k)}\in E\mid M^{(k-1)}=m)
    +\delta_s'.
\end{equation*}
Interchanging $P_1$ and $P_2$ gives the reverse inequality, proving \eqref{eq:Vadhan_generalisation}.

Due to \eqref{eq:Vadhan_generalisation}, by Lemma C.4 of \cite{cai2024optimal}, for every $u$ there exist
conditional kernels $\widetilde K_{j,s}(\cdot\mid u)$, $j\in\{1,2\}$,
such that
\begin{equation*}
    D_{\mathrm{TV}}
    \left(
        P_j^{R_s^{(k)}\mid H_{s-1}^{(k)}=u},
        \widetilde K_{j,s}(\cdot\mid u)
    \right)
    \leq\delta_s'
\end{equation*}
and
\begin{equation*}
    D_{\mathrm{KL}}
    \left(
        \widetilde K_{1,s}(\cdot\mid u),
        \widetilde K_{2,s}(\cdot\mid u)
    \right)
    \leq
    \epsilon_s'(e^{\epsilon_s'}-1).
\end{equation*}

Let $\widetilde P_j^{(k)}$ be the joint conditional law generated
sequentially by the kernels $\{\widetilde K_{j,s}\}_{s=1}^S$. We can bound 
\begin{align*}
    D_{\mathrm{TV}}\left(
        P_j^{R^{(k)}\mid M^{(k-1)}=m},
        \widetilde P_j^{(k)}
    \right)
    &\leq
    \sum_{s=1}^S
    \mathbb E\left[
        D_{\mathrm{TV}}\left(
            P_j^{R_s^{(k)}\mid H_{s-1}^{(k)}},
            \widetilde K_{j,s}(\cdot\mid H_{s-1}^{(k)})
        \right)
    \right] \leq
    \sum_{s=1}^S \delta_s',
\end{align*}
and we also have that 
\begin{align*}
    D_{\mathrm{KL}}(\widetilde P_1^{(k)},\widetilde P_2^{(k)})
    &=
    \sum_{s=1}^S
    \mathbb E_{\widetilde P_1^{(k)}}
    \left[
        D_{\mathrm{KL}}
        \left(
            \widetilde K_{1,s}(\cdot\mid H_{s-1}^{(k)}),
            \widetilde K_{2,s}(\cdot\mid H_{s-1}^{(k)})
        \right)
    \right] \\
    &\leq
    \sum_{s=1}^S
    \epsilon_s'(e^{\epsilon_s'}-1)
\end{align*}
Therefore, by the triangle inequality and Pinsker's inequality,
\begin{align*}
    D_{\mathrm{TV}}
    \left(
        P_1^{R^{(k)}\mid M^{(k-1)}=m},
        P_2^{R^{(k)}\mid M^{(k-1)}=m}
    \right)
    &\leq
    \left\{
        \frac12\sum_{s=1}^S
        \epsilon_s'(e^{\epsilon_s'}-1)
    \right\}^{1/2}
    +2\sum_{s=1}^S\delta_s'\\
    &\lesssim
    \frac{D_{\mathrm{TV}}(P_1, P_2)}{K} \left\{\left(\sum_{s=1}^S n_s^2\epsilon_s^2\right)^{1/2}
    + \sum_{s=1}^S n_s\delta_s \right\}.
\end{align*}
Here, the final inequality holds provided we choose perturbation $h$ such that $\epsilon_s'$ is small enough for all $s \in [S]$. Iteratively substituting the above bound into \eqref{eq:TV_chainrule} allows us to bound
\begin{equation*}
    D_{\mathrm{TV}}(P_1^{M^{(K)}},P_2^{M^{(K)}})
    \lesssim D_{\mathrm{TV}}(P_1, P_2)   \left(\sum_{s=1}^S n_s^2 \epsilon_s^2 \right)^{1/2}
\end{equation*}
under the conditions on $\{\delta_s\}_{s=1}^S$. Choose
\[
h^2=c_0\min \left\{1, \frac{1}{\min\{\sum_{s=1}^S n_s, \sum_{s=1}^S n_s^2\epsilon_s^2\}} \right\}.
\]
If $\sum_{s=1}^S n_s^2\epsilon_s^2\leq \sum_{s=1}^S n_s$, then, since $D_{\mathrm{TV}}(P_1,P_2)\lesssim h$,
\[
\max_{s\in[S]}\epsilon_s'
=\max_{s\in[S]}\frac{6n_s\epsilon_s}{K}D_{\mathrm{TV}}(P_1,P_2)
\lesssim h\sqrt{\sum_{s=1}^S n_s^2\epsilon_s^2}\leq\sqrt{c_0}.
\]
Thus, for sufficiently small $c_0>0$, the preceding private-transcript bound applies and gives
\[
D_{\mathrm{TV}}(P_1^{M^{(K)}},P_2^{M^{(K)}})
\lesssim h\sqrt{\sum_{s=1}^S n_s^2\epsilon_s^2}\leq\sqrt{c_0}.
\]
If $\sum_{s=1}^S n_s^2\epsilon_s^2>\sum_{s=1}^S n_s$, \eqref{eq:nonprivate-expdist} and
$D_{\mathrm{KL}}(P_1,P_2)\lesssim h^2$ instead give
\[
D_{\mathrm{TV}}(P_1^{M^{(K)}},P_2^{M^{(K)}})
\lesssim h\sqrt{\sum_{s=1}^S n_s}\leq\sqrt{c_0}.
\]
Choosing $c_0>0$ sufficiently small ensures
$D_{\mathrm{TV}}(P_1^{M^{(K)}},P_2^{M^{(K)}})\leq1/2$ in both cases and applying Le Cam's two-point lemma, we have 
\begin{align*}
    \inf_{Q \in \mathcal{Q}} \inf_{\widehat{\Lambda}} \sup_{P \in \mathcal{P}} \mathbb{E}\left[\sup_{t \in [0, 1]} \big|\widehat{\Lambda}(t) - \Lambda_P(t)\big|\right] &\geq \inf_{Q \in \mathcal{Q}}  \inf_{\widehat{\Lambda}} \max_{P \in \{P_1, \, P_2\}} \mathbb{E}_P\left[\big|\widehat{\Lambda}(1) - \Lambda_P(1)\big|\right] \gtrsim h.
\end{align*} 
\end{proof}

\section{Additional details on label differential privacy}\label{app-sec-labelDP}
To consider the extension to label-FDP, we start with a formal definition. 
\begin{definition}[Label federated differential privacy, label-FDP] \label{def:label-FDP} 
For $S \in \mathbb{N}_+$, let $\epsilon_s > 0$ and $\delta_s \geq 0$, $s \in [S]$, be privacy parameters.  For $K \in \mathbb{N}_+$, we say that a privacy mechanism $Q := \{Q_s^{(k)}\}_{s, k=1}^{S, K}$ satisfies $(\{\epsilon_s, \delta_s\}_{s=1}^S, K)$-label-FDP, if for any $s \in [S]$ and $k \in [K]$, the data $R_s^{(k)}$ shared by the server $s$ satisfies $(\epsilon_s, \delta_s)$-label CDP with local independent randomisation, i.e.
\[
    Q_s^{(k)} \left(R_s^{(k)} \in A_s^{(k)} \mid M^{(k-1)}, (X_s^{(k)}, Y_s^{(k)}) \right) \leq e^{\epsilon_s}Q_s^{(k)} \left(R_s^{(k)} \in A_s^{(k)} \mid M^{(k-1)}, (X_s^{(k)}, (Y_s^{(k)})') \right) + \delta_s,
\]
for all possible $X_s^{(k)}$ and any measurable set $A_s^{(k)}$, $M^{(k-1)} := \cup_{l=1}^{k-1} \cup_{s=1}^S R_s^{(l)}$, and pairs of dataset labels $Y_s^{(k)} \sim (Y_s^{(k)})'$, where for each $s \in [S]$, $\cup_{k=1}^K (X_s^{(k)}, Y_s^{(k)})$ forms a data-independent partition of the dataset at server $s$.  
\end{definition}

\Cref{subsec-labelDP-sims} shows the numerical improvement of using a label central-DP constraint in \Cref{alg:cdp_sgd}. \Cref{app-label-FDP-proof} contains the proof of \Cref{prop-label-FDP}. 

\subsection{Numerical results}
\label{subsec-labelDP-sims}
In \Cref{fig:labelDP_sims} we provide a comparison of numerical results from using a tighter sensitivity bound under label DP, compared to the full DP constraint.  We follow the same settings as in \Cref{subsec-sim-cox}, except for using the smaller magnitude of noise required for label CDP. 

\begin{figure}[htbp]
  \centering
\includegraphics[width=0.6\textwidth]{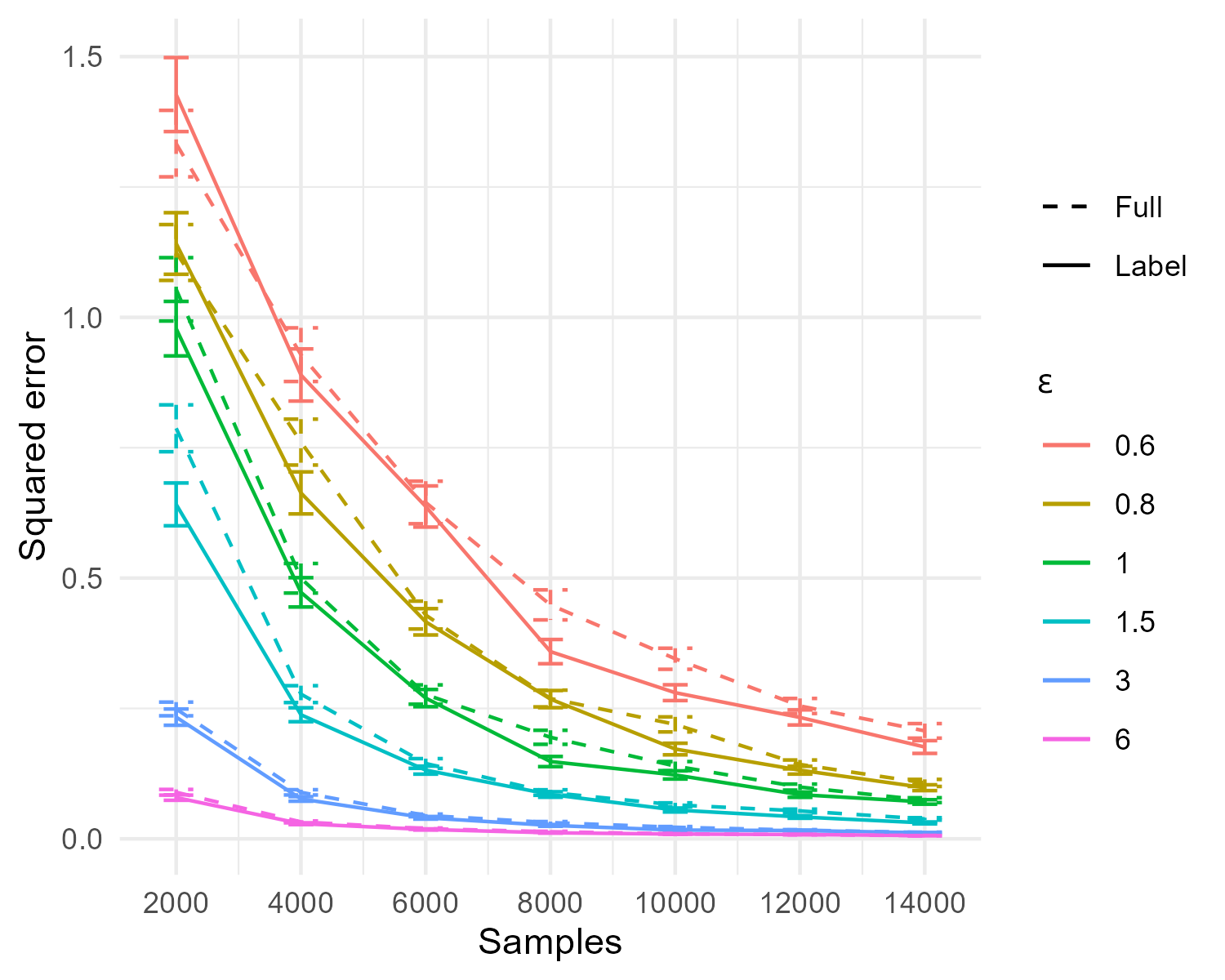}
  \caption{Comparisons of label DP (solid lines) and full DP (dashed lines) for $\beta_0$ estimation, with bars showing standard error from 200 repetitions. \label{fig:labelDP_sims}}
\end{figure}

With known, time-invariant covariates $\{Z_i\}_{i=1}^n$, we may replace $C_Z$ in the proof of \Cref{lemma:grad_sensitivity2} by the observed covariate norms. The three terms satisfy
\[
\|(I)\|_2\leq\frac{\max_{i\in[n]}\|Z_i\|_2}{n},
\qquad
\|(III)\|_2\leq\frac{2\max_{i\in[n]}\|Z_i\|_2}{n},
\]
and
\[
\|(II)\|_2\leq
\frac{2(1+\log n)}{n}\max_{i\in[n]}\|Z_i\|_2
\exp\left\{\max_{i\in[n]}\beta^\top Z_i-\min_{i\in[n]}\beta^\top Z_i\right\}.
\]
These bounds are computable in linear time in $n$. In particular, their sum is bounded by
\[
\frac{3C_Z}{n}+\frac{2C_Z(1+\log n)\exp(2\|\beta\|_2C_Z)}{n}.
\]
This gives a reduction of less than 25\% compared with the full-DP sensitivity bound in \Cref{lemma:grad_sensitivity2}, with the relative improvement decreasing as $n$ increases.

\subsection{\texorpdfstring{Proof of \Cref{prop-label-FDP}}{Proof of Proposition 7}}
\label{app-label-FDP-proof}
\begin{proof} 
Our upper bound follows from \Cref{prop:fdp_upper}, since any algorithm satisfying the FDP definition in \Cref{def:FDP} also satisfies the label-FDP definition in \Cref{def:label-FDP}. 

To prove the lower bound, we will show how the steps of the proof of \Cref{lemma:FDP_trace} can be adapted to the label-DP constraint. The notation used in the proof is for general public information $X$ and labels $Y$ such that the marginal distribution of $X$ does not depend on the parameter of interest; the survival setting corresponds to $X=\{Z_i\}_{i=1}^n$ and $Y=\{(T_i, \Delta_i)\}_{i=1}^n$. We write $(X_s^{(k)}, Y_s^{(k)})$ for the batch of data used by server $s$ on iteration $k$, which are assumed to be independent, and $L=\cup_{k=1, s=1}^{K,S} X_s^{(k)}$. Further define the individual and batch conditional scores by
\[
    S_{s,i}^{(k)}
    :=\nabla_\beta\log f_\beta(Y_{s,i}^{(k)}\mid X_{s,i}^{(k)}),
    \qquad
    S_s^{(k)}:=\sum_{i=1}^{b_s^{(k)}}S_{s,i}^{(k)}.
\]

Since $I_L(\beta)=0$, by the chain rule, the quantity we now need to bound for denominator of the van Trees inequality in \eqref{eq:van_Trees} is
\begin{equation}
    I_{R, L}(\beta) = \sum_{k=1}^K \sum_{s=1}^S I_{R_s^{(k)} \mid M^{(k-1)}, L}(\beta).
    \label{eq:trace-chain-rule}
\end{equation}

For every fixed $r, m, l$ such that  $f_\beta(r \mid m, l)$ is positive, we have 
\begin{align*}
    &\mathbb{E}_\beta\bigl[S_s^{(k)}\mid
      R_s^{(k)}=r,M^{(k-1)}=m,L=l\bigr]\\
    &\quad=\int \nabla_\beta\log f_\beta(y\mid x)
      \frac{q_s^{(k)}(r\mid x,y,m)f_\beta(y\mid x)}
      {f_\beta(r\mid m,l)}\dint y\\
    &\quad=\frac{1}{f_\beta(r\mid m,l)}
      \int q_s^{(k)}(r\mid l,y,m)\nabla_\beta f_\beta(y\mid x)\dint y\\
    &\quad=\frac{1}{f_\beta(r\mid m,l)}
      \nabla_\beta \int q_s^{(k)}(r\mid l,y,m)f_\beta(y\mid x)\dint y
      =\nabla_\beta\log f_\beta(r\mid m,l),
\end{align*}
which implies that 
\begin{align*}
    I_{R_s^{(k)}\mid M^{(k-1)},L}(\beta)
    :=\mathbb{E}_\beta\left[
      \left\{\nabla_\beta\log
      f_\beta(R_s^{(k)}\mid M^{(k-1)},L)\right\}^{\otimes2}
      \right]
      =\mathbb{E}_\beta\left[
      \left\{{\mathbb{E}}_\beta[S_s^{(k)}\mid
      R_s^{(k)},M^{(k-1)},L]\right\}^{\otimes2}
      \right].
\end{align*}

For each $i\in[b_s^{(k)}]$, define
\[
    G_i:=\langle S_{s,i}^{(k)}, \mathbb{E}_\beta[S_s^{(k)}\mid R_s^{(k)},M^{(k-1)},L]\rangle.
\]
By the tower property we have that 
\begin{align}
    \mathbb{E}_\beta\|\mathbb{E}_\beta[S_s^{(k)}\mid R_s^{(k)},M^{(k-1)},L]\|_2^2
      =\mathbb{E}_\beta\langle S_s^{(k)}, \mathbb{E}_\beta[S_s^{(k)}\mid R_s^{(k)},M^{(k-1)},L]\rangle
      =\sum_{i=1}^{b_s^{(k)}}\mathbb{E}_\beta[G_i].
    \label{eq:trace-attack}
\end{align}

Defining $G_{s, i}^{(k)} := \left\langle S_\beta(X_{s, i}^{(k)}, Y_{s, i}^{(k)}), \mathbb{E}[S_\beta(D_s^{(k)}) \mid R_s^{(k)}, M^{(k-1)}, L] \right\rangle$ and 
\[
\breve{G}_{s, i}^{(k)} := \left\langle S_\beta(X_{s, i}^{(k)},\breve{Y}_{s, i}^{(k)}), \mathbb{E}[S_\beta(X_s^{(k)}, Y_s^{(k)}) \mid R_s^{(k)}, M^{(k-1)}, L]\right\rangle,
\]
where \(\breve{Y}_{s,i}^{(k)}\) is conditionally independent of all labels and mechanism outputs given \(X_{s,i}^{(k)}\), and has the same conditional distribution as \(Y_{s,i}^{(k)}\mid X_{s,i}^{(k)}\). We now write $G_i:=G_{s,i}^{(k)}$, $\breve G_i :=\breve G_{s,i}^{(k)}$, and write $\breve S_i:=S_\beta(X_{s,i}^{(k)},\breve Y_{s,i}^{(k)})$. 

 By the $(\epsilon_s, \delta_s)$-label FDP constraint, we have that for any $W > 0$, 
\begin{equation*}
    \mathbb{E}(G_i)_+ = \int_0^\infty\mathbb{P}((G_i)_+ \geq t) \dint t \leq \mathbb{E}(\breve{G}_i)_+ + C_\epsilon \epsilon_s \int_0^\infty \mathbb{P}((\breve{G}_i)_+ \geq w) \dint w + W\delta_s +  \int_W^\infty \mathbb{P}((G_i)_+ > w) \dint w
\end{equation*}
and similarly that 
\begin{equation*}
    \mathbb{E}(G_i)_- \geq \mathbb{E}(\breve{G}_i)_- - C_\epsilon \epsilon_s \int_0^\infty \mathbb{P}((\breve{G}_i)_- \geq w) \dint w - W\delta_s -  \int_W^\infty \mathbb{P}((\breve{G}_i)_- \geq w) \dint w. 
\end{equation*}
Putting the last two equations together, we have 
\begin{align}
    \mathbb{E}[G_i] &= \mathbb{E}[(G_i)_+ - (G_i)_-] \notag \\
    &\lesssim  \mathbb{E}[\breve{G}_i] +  \epsilon_s \mathbb{E}|\breve{G}_i| + 2W\delta_s +  \int_W^\infty \{\mathbb{P}((G_i)_+ \geq w) + \mathbb{P}((\breve{G}_i)_- \geq w)\} \dint w.
    \label{pf:score_attack}
\end{align}
It remains to control each of these terms.

Since $\breve{Y}_{s,i}^{(k)}$ is conditionally independent of
$(R_s^{(k)},M^{(k-1)},Y)$ given $L$, we have
\begin{align*}
    \mathbb{E}[\breve{G}_i]
    &=\mathbb{E}\left[
        \mathbb{E}\left[
            \left\langle \breve{S}_i,\mathbb{E}_\beta[S_s^{(k)}\mid R_s^{(k)},M^{(k-1)},L]\right\rangle
            \,\middle|\,R_s^{(k)},M^{(k-1)},L\right]\right]\\
    &=\mathbb{E}\left[
        \left\langle
            \mathbb{E}\left[
                \breve{S}_i
                \,\middle|\,R_s^{(k)},M^{(k-1)},L\right],
            \mathbb{E}[S_s^{(k)}\mid R_s^{(k)},M^{(k-1)},L]
        \right\rangle\right]\\
    &=\mathbb{E}\left[
        \left\langle
            \mathbb{E}[\breve{S}_i\mid L],
            \mathbb{E}[S_s^{(k)}\mid R_s^{(k)},M^{(k-1)},L]
        \right\rangle\right].
\end{align*}
Because $\breve{S_i} \mid X_{s,i}^{(k)} \overset{d}{=} \breve{S}_i \mid L$, we have that 
\begin{align*}
    \mathbb{E}[\breve{S}_i\mid L]
    &=\int_{\mathcal Y}
        \nabla_\beta\log f_\beta
        \bigl(y\mid X_{s,i}^{(k)}\bigr)
        f_\beta\bigl(y\mid X_{s,i}^{(k)}\bigr)
        \,\mathrm{d}y\\
    &=\nabla_\beta
        \int_{\mathcal Y}
        f_\beta\bigl(y\mid X_{s,i}^{(k)}\bigr)
        \,\mathrm{d}y
      =0,
\end{align*}
and therefore that $\mathbb{E}[\breve{G}_i]=0$.

We can also bound 
\begin{align*}
    \sum_{i=1}^{b_s^{(k)}} \mathbb{E}|\breve{G}_i| &\leq \sqrt{{b_s^{(k)}} \sum_{i=1}^{b_s^{(k)}} \mathbb{E}[\breve{G}_i^2]} \\
    & =\sqrt{{b_s^{(k)}} \sum_{i=1}^{b_s^{(k)}}  \mathbb{E} \left[ \mathbb{E}[S_s^{(k)}\mid R_s^{(k)}, M^{(k-1)},  L]^\top \mathrm{Var}(\breve{S}_i | R_s^{(k)}, M^{(k-1)}, L) \mathbb{E}[S_s^{(k)}\mid R_s^{(k)}, M^{(k-1)},  L] \right]}.
\end{align*}
Define
\begin{equation*}
    V_i:=\mathrm{Var}_\beta(\breve S_i\mid L)
    =\mathrm{Var}_\beta(S_{s,i}^{(k)}\mid X_{s,i}^{(k)}),
\end{equation*}
and the event
\begin{equation*}
    \mathcal{E} := \left\{ \left\|\sum_{i=1}^{b_s^{(k)}} V_i  \right\|_2 \leq \frac{2c b_s^{(k)}}{d} \right\}.
\end{equation*}

The full Fisher-information calculation in the proof of
\Cref{prop:FDP_lower_bound}, applied conditional on $Z_i$, gives
\[
    V_i=Z_iZ_i^\top\{1-\exp(-\lambda_0(0)\exp(\beta^\top Z_i))\}.
\]
Hence $\mathbb E[V_i]=I_{(T,\Delta,Z)}(\beta)$, whose eigenvalues were
shown to be of order $1/d$. Moreover,
\[
    \|V_i\|_2\leq\|Z_i\|_2^2\leq C_Z^2.
\]
Since the $V_i$'s are independent, by a matrix Chernoff inequality \cite[e.g.][Theorem 5.1.1]{Tropp2015}, we have that for every $\gamma \geq 0$, 
\begin{equation*}
\mathbb{P}\left( \lambda_{\max}\left(\sum_{i=1}^{b_s^{(k)}} V_i\right)
\geq(1+\gamma)\mu_{\max}\right)\leq
d\left(\frac{e^\gamma}{(1+\gamma)^{1+\gamma}}\right)^{\mu_{\max}/C_Z^2}, \quad \mu_{\max} := \lambda_{\max} \left( \sum_{i=1}^{b_s^{(k)}} \mathbb{E}_\beta[V_i]\right).
\end{equation*}
Choosing a fixed constant $c>0$ such that $\mu_{\max}\leq cb_s^{(k)}/d$, and taking $\gamma=2cb_s^{(k)}/(d\mu_{\max})-1$, gives
\begin{equation*}
\mathbb{P}\left(
\lambda_{\max}\left(\sum_{i=1}^{b_s^{(k)}} V_i\right)
\geq
2 \frac{c b_s^{(k)}}{d}
\right) \leq d\exp\left\{-\frac{c b_s^{(k)}}{d C_Z^2} \left[2\log(2)-1\right]\right\}. 
\end{equation*}

We can therefore bound 
\begin{align*}
    \sum_{i=1}^{b_s^{(k)}} \mathbb{E}[|\breve{G}_i| \mathbbm{1}\{\mathcal{E}\}] & \leq \sqrt{b_s^{(k)}  \mathbb{E} \left[ \mathbb{E}[S_s^{(k)}\mid R_s^{(k)}, M^{(k-1)},  L]^\top  \left\{\sum_{i=1}^{b_s^{(k)}} V_i \right\}  \mathbb{E}[S_s^{(k)}\mid R_s^{(k)}, M^{(k-1)},  L] \mathbbm{1}\{\mathcal{E}\}\right] }\\
    &\leq \sqrt{\frac{2c(b_s^{(k)})^2}{d} \mathbb{E} \left\|  \mathbb{E}[S_s^{(k)}\mid R_s^{(k)}, M^{(k-1)},  L] \right\|_2^2}\\
    &= \frac{\sqrt{2c}b_s^{(k)}}{\sqrt{d}} \sqrt{\mathrm{Tr}\left(I_{R_s^{(k)}\mid M^{(k-1)},L}(\beta)\right)}.
\end{align*}

Since $\mathbb{P}(\mathcal{E}^c) \leq d \exp(-c_{\mathcal{E}}b_s^{(k)}/d)$, and the scores are bounded, we have that   
\begin{align*}
     \sum_{i=1}^{b_s^{(k)}} \mathbb{E}|\breve{G}_i| \lesssim \sqrt{\mathrm{Tr}\left(I_{R_s^{(k)}\mid M^{(k-1)},L}(\beta)\right)} \left( \frac{b_s^{(k)}}{\sqrt{d}} + b_s^{(k)} \sqrt{d} \exp \left\{-c_{\mathcal{E}} b_s^{(k)}/(2d) \right\}\right).
\end{align*}
If $ b_s^{(k)} \gtrsim  d\log(d)$ with a sufficiently large constant, then the first term dominates. Under this assumption, due the boundedness of the score function, we can use \eqref{eq:lowerbound_prob_control} to obtain that 
\begin{equation*}
   \mathrm{Tr}\left(I_{R_s^{(k)}\mid M^{(k-1)},L}(\beta)\right) \lesssim b_s^{(k)} \left( \epsilon_s \frac{1}{\sqrt{d}} \sqrt{\mathrm{Tr}\left(I_{R_s^{(k)}\mid M^{(k-1)},L}(\beta)\right)} + \log(1/\delta_s) \delta_s b_s^{(k)} \right),
\end{equation*}
and therefore
\begin{equation*}
    \mathrm{Tr}\left(I_{R_s^{(k)}\mid M^{(k-1)},L}(\beta)\right) \lesssim (b_s^{(k)} \epsilon_s)^2 / d,
\end{equation*}
provided that $\delta_s \log (1/\delta_s) \lesssim \epsilon_s^2 / d$. 

By a conditional data-processing inequality, we also have that 
\begin{equation*}
\mathrm{Tr}\left(I_{R_s^{(k)}\mid M^{(k-1)},L}(\beta) \right)\lesssim b_s^{(k)}.
\end{equation*}
Combining this with the privacy bound and summing over terms in \eqref{eq:trace-chain-rule} gives
\begin{equation*}
    \mathrm{Tr}(I_{R,L}(\beta))
    \lesssim\sum_{s=1}^S\min\{n_s,n_s^2\epsilon_s^2/d\},
\end{equation*}
since $\sum_k b_s^{(k)}\leq n_s$ and $\sum_k(b_s^{(k)})^2\leq n_s^2$. 
\end{proof}

\section{Additional background}
\label{app-background}
We collect some background definitions and results used in the main text and appendix.

\begin{lemma}[Composition property of DP, e.g.~Theorem 3.16 in \citealp{dwork2014algorithmic}] For $\epsilon_1, \epsilon_2 > 0$ and $\delta_1, \delta_2 \geq 0$, suppose that $M_1$ and $M_2$ are $(\epsilon_1, \delta_1)$-DP and $(\epsilon_2, \delta_2)$-DP mechanisms.  It then holds that $M_1 \circ M_2$, the composition of $M_1$ and $M_2$, is $(\epsilon_1 + \epsilon_2, \delta_1 + \delta_2)$-DP. 
\label{lemma:composition}
\end{lemma}

\begin{definition}[$\ell_2$-sensitivity] \label{def-sensitivity}
Let $f$ be a function $f:\mathcal{X}^n \rightarrow \mathbb{R}^d$. The $\ell_2$ sensitivity of a function is defined as 
\begin{equation}
    \mathrm{sens}_2(f) := \sup_{D \sim D'} \|f(D) - f(D')\|_2,
\end{equation}  
where $D\sim D'$ denotes neighbouring datasets: $\sum_{i=1}^n \mathbbm{1}\{D_i \neq D_i'\} = 1$.
\end{definition}

\begin{lemma}[Gaussian mechanism]
\label{lemma:gaussian_mechanism}
Let $f$ have finite $\ell_2$-sensitivity. For any $\epsilon>0$ and $0<\delta<1$, the mechanism
\begin{equation*}
    M_G(D):=f(D)+
    \frac{\sqrt{2\log(1/\delta)+\epsilon}\,\mathrm{sens}_2(f)}
    {\epsilon}Z,
    \qquad Z\sim\mathcal{N}(0,I_d),
\end{equation*}
is $(\epsilon,\delta)$-DP.
\end{lemma}
\begin{proof}
Set $\alpha:=1+2\log(1/\delta)/\epsilon$. By \Cref{lemma:Gaussian_RDP}, the mechanism is $(\alpha,\epsilon/2)$-RDP. Applying \Cref{lemma:RDP_to_CDP} gives $(\epsilon,\delta)$-DP.
\end{proof}
\begin{definition}[R\'enyi divergence] For two probability distributions $P$ and $Q$ defined over $\mathcal{R}$, the R\'enyi divergence of order $\alpha>1$ is defined as 
\begin{equation*}
    D_\alpha(P \|Q) := \frac{1}{\alpha-1} \log \mathbb{E}_{x \sim Q} \left[\left(\frac{p(x)}{q(x)} \right)^\alpha \right], 
\end{equation*}
where $p(\cdot)$ and $q(\cdot)$ denote the density functions of $P$ and $Q$, respectively.
\end{definition}

\begin{definition}[R\'enyi differential privacy, RDP, \citealp{mironov2017renyi}] 
    For $\alpha > 1$ and $\epsilon > 0$, a randomised  mechanism $M: \mathcal{X}^n \rightarrow \mathcal{R}$ is an $(\alpha, \epsilon)$-RDP mechanism, if it satisfies
\begin{equation*}
    D_\alpha(M(D)\|M(D')) \leq \epsilon,
\end{equation*}
for any $D, D' \in \mathcal{X}^n$, such that  $\sum_{i=1}^n \mathbbm{1}\{D_i \neq D_i'\} = 1$.
\label{def:rdp}
\end{definition}

\begin{lemma}[Gaussian mechanism, Corollary 3 in \citealp{mironov2017renyi}]
Let $f : \mathcal{X}^n \to \mathbb{R}^d$ be a function with $\ell_2$-sensitivity  $\mathrm{sens}_2(f)$. Then for any $\alpha > 1$, the Gaussian mechanism defined by $M(x) = f(x) + \mathcal{N}(0, \sigma^2 I_d)$ satisfies $(\alpha, \alpha \,\mathrm{sens}_2(f)^2/(2\sigma^2))$-RDP.
\label{lemma:Gaussian_RDP}
\end{lemma}

\begin{lemma}[Proposition 1 in \citealp{mironov2017renyi}] Suppose that $M_1$ is $(\alpha, \epsilon_1)$-RDP and $M_2$ is $(\alpha, \epsilon_2)$-RDP. Then the composition of $M_1$ and $M_2$ is $(\alpha, \epsilon_1 + \epsilon_2)$-RDP. In particular, this holds even if $M_2$ is chosen adaptively based on the output from $M_1$.
\label{lemma:composition_RDP}
\end{lemma}

\begin{lemma}[Proposition 3 in \citealp{mironov2017renyi}]
If a mechanism $M$ satisfies $(\alpha, \epsilon)$-RDP for some $\alpha > 1$, then for any $\delta \in (0, 1)$, it also satisfies $(\epsilon + \log(1/\delta)/(\alpha - 1), \delta)$-CDP.
\label{lemma:RDP_to_CDP}
\end{lemma}

\begin{lemma}[Post-processing property of DP, e.g.~Proposition 2.1 in \citealp{dwork2014algorithmic}] For any measurable function $g: \, \mathcal{R} \rightarrow \mathcal{Y}$, if $M:\, \mathcal{X}^n \rightarrow \mathcal{R}$ is an $(\epsilon, \delta)$-DP mechanism, then $g \circ M$ is also $(\epsilon, \delta)$-DP. 
\label{lemma:post_processing}
\end{lemma} 

\begin{lemma}[Composition of Gaussian mechanisms]
\label{lemma:Gaussian_composition}
For $k\in[K]$, let $f_k:\mathcal X^n\times\prod_{j<k}\mathbb R^{d_j}
\to\mathbb R^{d_k}$ and $g_k:\mathbb R^{d_k}\to\mathbb R^{d_k}$
be deterministic measurable functions, where $d_k\in\mathbb N_+$.
Let $u\mapsto s_2(f_k;u)$ be a finite,
non-negative measurable function such that, for every
fixed history $u\in\prod_{j<k}\mathbb R^{d_j}$,
\[
s_2(f_k;u) \geq\sup_{x\sim x'} \|f_k(x,u)-f_k(x',u)\|_2.
\]
Consider the adaptive composition
\[
M_k(x)=g_k \left(f_k(x,u)+\xi_k\right), \qquad u=(M_1(x),\ldots,M_{k-1}(x)),
\]
where 
\[
\xi_k\mid u\sim\mathcal{N}\left(0,\frac{(2\log(1/\delta)/\epsilon+1)\,s_2(f_k;u)^2}{\epsilon/K}I_{d_k}\right).
\]
Then $(M_1(x),\ldots,M_K(x))$ satisfies $(\epsilon,\delta)$-CDP.
\end{lemma}

\begin{proof}
Set $\alpha:=2\log(1/\delta)/\epsilon+1$.
For every fixed history $u$, the sensitivity and noise variance
are the same for neighbouring datasets. Thus, by Lemmas
\ref{lemma:Gaussian_RDP} and \ref{lemma:post_processing},
the conditional mechanism is $(\alpha,\epsilon/(2K))$-RDP.
By adaptive composition, the full transcript is
$(\alpha,\epsilon/2)$-RDP, and \Cref{lemma:RDP_to_CDP}
gives $(\epsilon,\delta)$-CDP.
The conclusion for $M_K(x)$ follows by post-processing.
\end{proof}

\begin{definition}[$\delta$-approximate max-divergence] For any $\delta\geq 0$, the $\delta$-approximate max-divergence is defined as 
\begin{equation*}
        D_\infty^{\delta}(Y, Z) := \sup_{S \subset \mathrm{Supp}(Y): \, \mathbb{P}(Y \in S) > \delta} \log \left\{\frac{\mathbb{P}(Y \in S) - \delta}{\mathbb{P}(Z \in S)} \right\}.
    \end{equation*}
\end{definition}

\begin{lemma}[Corollary of Theorem 1.5 in \citealp{naor2012banach}]
\label{lemma:smooth_banach_martingale} 
Let $\{M_k\}_{k=1}^\infty$ be a martingale in $\mathbb{R}^d$, and let $\{a_k\}_{k=1}^\infty$ be a sequence of positive integers such that $\|M_{k+1} - M_k\|_2 \leq a_k$ for all $k \in \mathbb{N}$. Then for every $u > 0$ and $k \in \mathbb{N}$ we have
\begin{equation*}
\mathbb{P}(\| M_{k+1} - M_1 \|_2 > u ) \leq e^{2.5} \cdot \exp\left( - \frac{c u^2}{a_1^2 + \cdots + a_k^2} \right).
\end{equation*}
for some absolute constant $c$. 
\end{lemma}

\begin{proof}The lemma follows from combining Theorem 1.5 of \cite{naor2012banach} with the fact that $(\mathbb{R}^d, \| \cdot\|_2)$ is a Banach space with uniform modulus of smoothness of power type 2 equal to 1/2, i.e. $\rho(\tau) \leq \tau^2/2$ for all $\tau>0$, where
    \begin{equation*}
     \rho(\tau) := \sup \left\{ \frac{\|x + \tau y\|_2 + \|x - \tau y\|_2}{2} - 1: x, y  \in \mathbb{R}^d, \, \|x\|_2 = \|y\|_2 = 1\right\}.
\end{equation*}
\end{proof}

\end{document}